\documentclass[10pt]{amsart}
\usepackage[hidelinks]{hyperref}
\usepackage{tikz-cd}
\usepackage{doi,url}

\usepackage[utf8]{inputenc}
\usepackage{amsmath}
\usepackage{amssymb}
\usepackage[numbers]{natbib}
\usepackage{prettyref}
\usepackage{graphicx}
\usepackage{enumitem}
\usepackage{mathtools}
\usepackage{xcolor}
\usepackage{comment}

\numberwithin{equation}{section}

\newtheorem{theorem}{Theorem}[section]

\newtheorem{proposition}[theorem]{Proposition}
\newtheorem{corollary}[theorem]{Corollary}

\newtheorem{lemma}[theorem]{Lemma}

\newcounter{introthmcounter}
\newtheorem{thmA}{Theorem}[introthmcounter]

\theoremstyle{definition}

\newtheorem{definition}[theorem]{Definition}
\newtheorem{notation}[theorem]{Notation}

\newtheorem{example}[theorem]{Example}

\theoremstyle{remark}
\newtheorem{remark}[theorem]{Remark}

\newcommand{\bE}{\mathbb E}

\newcommand{\bR}{\mathbb R}
\newcommand{\bN}{\mathbb N}
\newcommand{\bK}{\mathbb K}

\newcommand{\bZ}{\mathbb Z}

\DeclareMathOperator{\Th}{\mathrm{Th}}

\newcommand{\cS}{\mathcal{S}}

\newcommand{\RCF}{\mathrm{RCF}}

\newcommand{\cT}{\mathcal{T}}
\def\-{\text{-}}

\newcommand{\tp}{\mathrm{tp}}

\newcommand{\dcl}{\mathrm{dcl}}
\newcommand{\convex}{\mathrm{convex}}

\newcommand{\Br}{\mathrm{Br}}
\newcommand{\br}{\mathrm{br}}
\newcommand{\var}[1]{\mathrm{#1}}
\newcommand{\val}{\mathbf{v}}
\newcommand{\res}{\mathbf{r}}
\newcommand{\rv}{\mathbf{rv}}
\newcommand{\CH}{\mathrm{CH}}
\newcommand{\cO}{\mathcal{O}}

\newcommand{\co}{\mathfrak{o}}

\newcommand{\Exponents}{\mathrm{Exponents}}

\DeclareMathOperator{\Vr}{\mathrm{V}}
\DeclareMathOperator{\lder}{{\dagger}}
\DeclareMathOperator{\height}{\mathrm{ht}}

\DeclareMathOperator{\Kr}{\mathrm{Ker}}

\title{Absorbed Types and Derivations in Exponential o-Minimal Theories}

\author{Pietro Freni}

\address{Institute of Mathematics, Czech Academy of Sciences, Žitná 25, 110 00 Praha 1, Czech Republic}

\subjclass[2020]{Primary 03C64; Secondary 12H05 12J10}

\keywords{O-absorbed cut, O-symmetric cut, exponential o-minimal theory, immediate extension, o-minimal field, o-minimality, order-convex derivation, simply exponential o-minimal theory, {$T$}-convex valuation, weakly immediate type, weak O-limit, wim-constructible extension, transserial o-minimal theory, transseries, valuation-convex derivation}

\thanks{The author was supported by EPSRC Fellowship
EP/T018461/1 at the University of Leeds, by the Czech Science Foundation grant 26-22545I, and by the Czech Academy of Sciences CAS (RVO 67985840).
}

\begin{document}

\begin{abstract}
	I analyze $\cO$-weakly immediate and $\cO$-residual types in an o-minimal expansion of an ordered field $\bE$, where $\cO$ is a convex valuation ring. The main result is a characterization of those exponential theories $T$ such that for all $(\bE, \cO)\models T_\convex$ the image of any $\cO$-weakly immediate type is given by some composition of translations, sign changes, and exponentials, of some \emph{possibly different} $\cO$-weakly immediate type. I call these theories \emph{transserial}, and they encompass simply exponential theories such as $T_{\exp}$ and $T_{an, \exp}$. A consequence of the analysis is that there are no counterexamples to \emph{Tressl's signature alternative} (cf.\ \cite{tressl2005model}) in models of transserial theories admitting an Archimedean prime model.

    The characterization has at its core some arguments that use very few but fundamental properties of the valued differential field of germs at a cut in an o-minimal structure. These are abstracted into some conditions of compatibility between the derivation and the order or the derivation and the valuation, both ultimately stemming from the mean-value-theorem in o-minimal structures. I develop some basic theory around these notions and observe that in the case of few constants (i.e.\ when the valuation ring contains the constants) they specialize to notions thoroughly studied in \cite{aschenbrenner2019asymptotic}.
\end{abstract}

\maketitle

\section{Introduction}
Let $T$ be the theory of an o-minimal expansion of an ordered field and let $T_\convex$ be the (complete) theory of models of $T$ expanded by a predicate for a non-trivial $T$-convex valuation subring (cf.\ \cite[Def.~2.7, Cor.~3.13]{dries1995t}). If $\cO$ is a convex subring of a field $\bE$, we denote by $\co$ the corresponding maximal ideal. We will use the standard notations $\val_\cO$, $\rv_\cO$, and $\res_\cO$, respectively for the valuation induced by $\cO$, the rv-sort induced by $\cO$, and the map $\res_\cO: \bE \to \cO/\co$ which is $0$ outside of $\cO$ and is given by the quotient modulo $\co$ on $\cO$. We drop the subscript $\cO$ if it is clear from the context.

In \cite{freni2024t}, the author introduced the $T$-$\lambda$-spherical completion of a model $(\bE, \cO)\models T_\convex$ (where $\lambda$ is an uncountable cardinal): this is a prime elementary $\lambda$-spherically complete extension of $(\bE_\lambda, \cO_\lambda)\succ(\bE, \cO)$, it is in general unique up to isomorphism and has the same residue field as $(\bE,\cO)$ (\cite[Thm.~B]{freni2024t}).

Toward the goal of better understanding these completions, it was shown in the same paper that if $T$ is simply exponential (that is, an expansion of a power-bounded o-minimal theory by a compatible exponential), then $(\bE_\lambda, \cO_\lambda)$ only realizes unary types $p$ over $\bE$ which are images of some other $\cO$-weakly immediate type $q$ over $\bE$ by an $\bE$-definable \emph{generalized nested exponential} (henceforth a \emph{gne}): a finite composition $g$ of translations by elements of $\bE$, sign changes and exponentials \cite[Prop.~5.27 and Cor.~5.30]{freni2024t}.

In this paper I characterize those complete exponential o-minimal theories $T$ for which this fact holds true. These are called \emph{transserial} o-minimal theories and are those satisfying the following strengthening of exponential boundedness:
\begin{enumerate}[label=(TS), ref=(TS)]
    \item\label{axiom:transserial0} for all $n\in \bN$ and for all $(n+1)$-ary $T$-definable functions $f$, setting 
    \[\dagger_n f(x_0, \ldots, x_{n-1}, t)=\frac{\partial f(x_0, \ldots, x_{n-1}, t)/\partial t}{f(x_0, \ldots, x_{n-1}, t)} \quad \text{if it is defined}\]
    and $\dagger_n f(\overline{x}, t)=\infty$ otherwise, and letting $G_m^f(\overline{x},t)\coloneqq\min\{|(\lder^k_nf)(\overline{x},t)|: k \le m\}$ (here $\lder_n^k$ is the $k$-fold iterate of $\lder_n$) one has that for some $m \in \bN$
    \[T_\convex \models \forall \overline{x},\; \exists r\in \cO, \;\forall t \in \cO,\; \big(t>r \rightarrow G_m^f(\overline{x},t) < 1 \big).\]
\end{enumerate}

It turns out that for transserial theories, a similar property to the one outlined above for $T$-$\lambda$-spherical completions, is shared by other kinds of extensions. So, to state the main result in a compact way, let us say that an elementary extension $(\bE, \cO) \prec (\bE_*, \cO_*)$ of models of $T_\convex$
\begin{enumerate}[label=(E\arabic*)]
    \item\label{df:gne-ppty} has the \emph{gne-property} if for all elements $y\in \bE_*$, there is an $\bE$-definable gne $g$ such that $\val(g^{-1}(y)-\bE)\subseteq \val(\bE)$,
\end{enumerate}
and recall that an elementary extension $(\bE, \cO) \preceq (\bE_*, \cO_*)\models T_\convex$ is said to be:
\begin{enumerate}[label=(E\arabic*), resume]
    \item\label{df:res-constructible} \emph{residually constructible} (or \emph{res-constructible}, cf.\ \cite[Def.~14]{freni2025residually}), if there is a $\dcl_T$-basis $\overline{r}$ of $\bE_*$ over $\bE$, such that $\res(\overline{r})$ is $\dcl_T$-independent over $\res(\bE)$;
    \item\label{df:wim-constructible} \emph{($\lambda$-bounded) wim-constructible} (cf.\ \cite[Def.~3.15]{freni2024t}) if there is an ordinal-indexed $\dcl_T$-basis $(x_i: i<\alpha)$ of $\bE_*$ over $\bE$, such that for all $j<\alpha$, $x_j$ is a pseudolimit of a p.c.\ sequence in $\bE \langle x_i: i<j\rangle$ without pseudolimits therein and of length $<\lambda$.
\end{enumerate}
Also recall that by \cite[Thm.~A]{freni2024t},
\begin{enumerate}[label=(\alph*)]
    \item if $(\bE, \cO) \prec (\bE_*, \cO_*)$ is wim-constructible then $\res(\bE, \cO)=\res(\bE_*, \cO_*)$, hence if $y \in \bE_*\setminus \bE$, $\val(y-\bE)\subseteq \val(\bE)$ means that $y$ is a pseudolimit;
    \item if $(\bE, \cO) \prec (\bE_*, \cO_*)$ is residually constructible, then $\bE_*$ does not contain pseudolimits of limitless p.c.\ sequences in $\bE$, hence if $y \in \bE_*\setminus \bE$, $\val(y-\bE)\subseteq \val(\bE)$ means that $y=c+dz$ for some $c,d \in \bE$, and $z\in \bE_*\setminus \bE$ with $\res(z)\notin \res(\bE, \cO)$.
\end{enumerate}
Note that a $T$-$\lambda$-spherical completion is a wim-constructible extension (in fact a maximal $\lambda$-bounded one).
The main result is thus stated as follows:

\begin{thmA}\label{introthm:transserial_char}
    The following are equivalent for an exponential o-minimal theory $T$:
    \begin{enumerate}[ref=(\arabic*)]
        \item\label{introthmenum:transserial} $T$ is transserial;
        \item\label{introthmenum:exp-bdd_gne_at_special} $T$ is exponentially bounded and for all $(\bE, \cO) \models T_\convex$, all $\cO < b <\bE^{>\cO}$ in some elementary extension $\bE_*\succ \bE$, for every $y \in \bE \langle b \rangle$ there is an $\bE$-definable gne $g$, such that $y \equiv_\bE g(b)$;
        \item\label{introthmenum:gne_wim_1} all 1-$\dcl_T$-dimensional wim-constructible extensions of models of $T_\convex$ have the gne-property;
        \item\label{introthmenum:gne_res_1} all 1-$\dcl_T$-dimensional res-constructible extensions of models of $T_\convex$ have the gne-property;
        \item\label{introthmenum:gne_wim} all wim-constructible extensions of models of $T_\convex$ have the gne-property;
        \item\label{introthmenum:gne_res} all res-constructible extensions of models of $T_\convex$ have the gne-property.
    \end{enumerate}
\end{thmA}

The name \emph{transserial} comes from the field of logarithmic exponential transseries, where certain configurations such as formal asymptotic solutions of the equation $\log(f(x)-x)=f(\log_2(x))$ are prohibited.

I also show that if a transserial theory admits an Archimedean prime model, then Tressl's signature alternative (cf.\ \cite[Def.~3.16]{tressl2005model} or see Subsection~\ref{ssec:first_applications2}) holds true for every cut over every model. Moreover, all Hardy fields of a transserial theory have the same Rosenlicht levels (as defined in \cite[Def.~5.2]{berarducci2015surreal} in the context of surreal numbers with the natural valuation, see Definition~\ref{def:rosenlicht_levels}).

\begin{thmA}\label{introthm:transserial_consequences}
    Let $T$ be a transserial theory, then:
    \begin{enumerate}
        \item wim-constructible and residually constructible extensions do not add new Rosenlicht levels, in particular for all $\bK\prec \bK'\models T$, the Hardy fields of $\bK$ and $\bK'$ have the same Rosenlicht levels;
        \item if $T$ has an Archimedean prime model, then for every $\bE\models T$, every unary type $p$ over $\bE$ satisfies Tressl's signature alternative.
    \end{enumerate}
\end{thmA}

Note that the condition (1) is a weakening of the notion of ($\bZ$-)\emph{levelled} o-minimal theory of \cite{marker1997levelled}. 

\smallskip 

The way these results are obtained is an analysis of the ordered differential valued fields $(\bE\langle x\rangle, \cO_x, \partial_x)$ where
\begin{enumerate}[label=(C\arabic*)]
    \item $\bE\langle x\rangle\coloneqq\dcl_T(\bE, x)$ is a 1-$\dcl_T$-dimensional elementary extension of $\bE\models T$ and $\cO_x$ is a $T$-convex valuation subring;
    \item $\partial_x$ is the unique non-trivial $\bE$-linear $T$-derivation with $\partial_x (x)=1$ (cf.\ \cite{fornasiero2021generic}, \cite{kaplan2023t});
    \item $\val_{\cO_x}(x-\bE)\subseteq \val_{\cO_x}(\bE)$.
\end{enumerate}

Note that this situation comprises the cases in which $(\bE \langle x\rangle, \cO_x)\succ (\bE, \cO)$ is wim-constructible or res-constructible.

A key ingredient in Theorem~\ref{introthm:transserial_char} is that in the setting above, if $T$ is exponential, an element $y \in \bE \langle x \rangle$ is of the form $g(z)$ for an $\bE$-definable gne and a $z$ such that $\val(z-\bE)\subseteq \val(\bE)$, if and only if there is a natural number $m$ such that $\val(\lder_x^m \partial_x(y))\in \val(\bE)$ (Corollary~\ref{cor:main-cor}).

\smallskip

The treatment points out some general properties of ordered differential valued fields of the form $(\bE \langle x\rangle, \cO_x, \partial_x)$ satisfying (C1) and (C2) above. When $x>\bE$, and $\cO_x$ is the convex hull of $\bE$, $(\bE \langle x\rangle, \cO_x, \partial_x)$ is the Hardy field of the o-minimal structure $\bE$ which is an example of an $H$-field in the sense of \cite{aschenbrenner2002fields}.

In a similar vein, even when $\tp(x/\bE)$ is not a definable type, many important properties are captured by the notions of an \emph{order-convex} (resp.\ \emph{logarithmically ---}) derivation and a \emph{valuation-convex} (resp.\ \emph{logarithmically ---}) derivation (Definitions~\ref{def:order-convex_der} and \ref{def:v-convex_derivation}).

I call a derivation $\partial$ on an ordered field $\bE_*$, \emph{order-convex} if it maps order-convex sets onto order-convex subsets of $\partial\bE_*$; similarly, it is \emph{logarithmically order-convex} if the corresponding logarithmic derivative $\lder(t)\coloneqq (\partial t)/t$ sends order-convex sets to order-convex subsets of $\lder(\bE^{\neq0}_*)$. These can be regarded as traces of the mean-value theorem for functions definable in o-minimal ordered fields.
The notions of valuation-convex and logarithmically valuation-convex arise naturally as traces of the corresponding notions of order-convexity with respect to the minimal order-convex valuation, specifically logarithmic convexity has several weakenings that are sufficient for most purposes.

As expected, when $\cO_*$ contains the field of constants, these notions specialize to several notions already studied in the literature (\cite{rosenlicht1980differential},\cite{rosenlicht1983rank}, \cite{aschenbrenner2002fields}, and \cite{aschenbrenner2019asymptotic}) such as pre-H-field, asymptotic field, pre-d-valued field and pre-d-valued field of H-type (see Subsection~\ref{ssec:few-constants}).

Given suitable convexity assumptions on the derivation on a valued field $(\bE_*, \cO_*)$ with constants $\bE\coloneqq\Kr(\partial)$, the condition $\val(x-\bE)\subseteq \val(\bE)$ is equivalent to the fact that the derivation ``at $x$'' $\partial_x=\partial/\partial x$ has the property that $y \succ \partial_x\cO_x\Rightarrow \lder_x(y)\prec y$. We will say that such derivations are \emph{absorbed} (Definition~\ref{def:absorbed-type-der}), a property that will be at the center of several arguments.


\subsection{Acknowledgments} I thank Vincenzo Mantova for several conversations on the topic and for encouraging me to try to convert some previous proofs of the main theorems into general arguments about valued differential fields.

\section{Symmetric and absorbed types}
Throughout the paper we will freely use many basic facts about o-minimal theories (for which we refer the reader to \cite{dries1998tame}), and about $T$-convex valuations in the sense of \cite{dries1995t}, \cite{dries1997t}\nocite{dries1998correction}.

In this section we revisit some basic facts and notions about the relation between unary types (i.e.\ cuts) over an o-minimal ordered field and a convex valuation on such a field. As such, this section is mostly book-keeping. Throughout the section we fix a theory $T$ of some o-minimal ordered field and a model $\bE\models T$. $\bE_*\succ \bE$ will denote an elementary extension of $\bE$. If $\cO$ is a convex subring of $\bE$, then $\cO_*^-$ and $\cO_*^+$ will denote the convex subrings of $\bE_*$
\[\cO_*^-\coloneqq \CH_{\bE_*}(\cO) \qquad \cO_*^+\coloneqq\{t \in \bE_*: |t|<\bE^{>\cO}\},\]
where $\CH_{\bE_*}$ denotes the convex hull. We will denote by $\co_*^-$ and $\co_*^+$ their respective maximal ideals. Note that
\[\co_*^-\coloneqq \{t \in \bE_*: |t|<\bE^{>\co}\} \qquad \co_*^+\coloneqq \CH_{\bE_*}(\co).\]

\subsection{Weakly immediate types and types with sign 0}
In this subsection we will recall and relate the notions of types with sign $0$ from \cite[Def.~3.10]{tressl2005model}, and of weakly $\cO$-immediate type from \cite[Sec.~2]{freni2024t}.

\begin{definition}
    If $p$ is a convex type over $\bE$, then the \emph{breadth} of $p$ is defined as the (partial) type
    \[\Br(p)(\var{x}):=\big\{ |\var{x}|<b-a : a,b \in \bE\cup \{\pm\infty\}, \; p(\var{y})\vdash a<\var{y}<b\big\}.\]
    If $x \in \bE_* \setminus \bE$, we write $\Br(x/\bE)$ for $\Br(\tp(x/\bE))$. Note that $\Br(x/\bE)$ is a type-definable convex subgroup (cf.\ e.g.\ \cite[Rmk.~2.3]{freni2024t})
\end{definition}

\begin{remark}\label{rmk:relation_to_additive_invariance_group}
    Let $p$ be a complete type over $\bE$ and suppose $\bE_*\succ \bE$ realizes $p$. Notice that in general for all $x \in p(\bE_*)$ we have $x-p(\bE_*) \subseteq \Br(p)(\bE_*)$. In fact if $x' \in p(\bE_*)$, then $|x-x'|<b-a$ for all $a,b$ such that $a<p(\bE_*)<b$.
    On the other hand we also have $|x-\bE|>\Br(p)(\bE)$. In fact if there were $m \in \Br(p)(\bE)$ and $c\in \bE$ such that $|x-c| \le m$, we would have $c-m<x<c+m$ and $2m \in \Br(p)(\bE)$ contradicting the fact that $\Br(p)(\bE)$ is a group. More specifically we have
    \[\Br(p)(\bE)=\{c \in \bE: p(\bE_*)+c = p(\bE_*)\}.\]
    In fact, clearly for all $c \in \Br(p)(\bE)$ we have $c+p(\bE_*) = p(\bE_*)$, because $|p(\bE_*)-\bE|>\Br(p)(\bE)$, on the other hand, if $c > b-a$ for some $a<x<b$, then $2c + p(\bE_*) > b >p(\bE_*)$. In particular $\Br(p)(\bE)$ is what in \cite[Sec.~3]{tressl2005model} is called the \emph{(additive) invariance} group of $p$.
\end{remark}

\begin{lemma}\label{lem:equiv_0-sign}
    The following are equivalent for $x \in \bE_* \setminus \bE$:
    \begin{enumerate}
        \item $\Br(x/\bE)(\bE_*) + x = \tp(x/\bE)(\bE_*)$;
        \item there are no convex additive subgroups $G$ of $\bE$, $c \in \bE$, and $\sigma \in \{\pm1\}$, such that $G < \sigma x - c <\bE^{>G}$.
    \end{enumerate}
    \begin{proof}
        Let $p=\tp(x/\bE)$ and $M=\Br(p)$ to lighten the notation.
        
        $(\lnot 1) \Rightarrow (\lnot 2)$ If $(x+M(\bE_*))\setminus p(\bE_*)\neq \emptyset$, then there is $c \in \bE \cap (x+ M(\bE_*))\setminus p(\bE_*)$, whence $|x-c| \in M(\bE_*)$. But by Remark~\ref{rmk:relation_to_additive_invariance_group} we then also have that $|x-c|>M(\bE)$, so $x \in c \pm M(\bE_*)^{>M(\bE)}$.

        $(\lnot 2) \Rightarrow (\lnot 1)$. If there were $G$, $c$, $\sigma$ as in (2) then it would be $M(\bE)=G$ and $x \in c \pm M(\bE_*)^{>M(\bE)}$, so $M(\bE_*) + x =M(\bE_*)+c \supsetneq \tp(x/\bE)(\bE_*)$.
    \end{proof}
\end{lemma}

\begin{remark}
    In \cite[Def.~3.10]{tressl2005model} types $p=\tp(x/\bE)$ satisfying the equivalent conditions of Lemma~\ref{lem:equiv_0-sign} are called \emph{types with sign $0$}. The types that do not have sign $0$ are called \emph{signed} (their sign being the unique $\sigma\in \{\pm1\}$ for which there is a convex subgroup $G$ s.t.\ $G<\sigma x-c<\bE^{>G}$). In \cite{kuhlmann2020selected}, signed cuts are called \emph{ball-cuts}, and somewhere else they are called \emph{asymmetric} (see \cite[Sec.~3]{kuhlmann2020selected} for a survey). We will be mostly interested in the notion of non-ball or 0-signed cuts, so we will stick with the only already established positive nomenclature \emph{symmetric}. This seems to be deprecated in the context of ordered abelian groups that are not $2$-divisible (see \cite[Sec.~3]{kuhlmann2020selected}), but we are working with o-minimal expansions of an ordered field.
\end{remark}

\begin{definition}
    We call unary types $p=\tp(x/\bE)$ satisfying the equivalent conditions of Lemma~\ref{lem:equiv_0-sign} \emph{symmetric}.
    We will say that $p$ is \emph{(weakly) $\cO$-symmetric} if it is symmetric and $\Br(p)(\bE)$ is a $\cO$-module.
    
    Recall from \cite[Sec.~2]{freni2024t} that $\tp(x/\bE)$ is \emph{$\cO$-weakly immediate} (\emph{$\cO$-wim} for short) if one of the equivalent conditions hold
    \begin{enumerate}
        \item $\val_{\cO_*}(x-\bE)$ has no maximum for \emph{some} convex subring $\cO_*$ of $\bE_*$ s.t.\ $\cO_*\cap \bE=\cO$;
        \item $\val_{\cO_*}(x-\bE)$ has no maximum for \emph{every} convex subring $\cO_*$ of $\bE_*$ s.t.\ $\cO_*\cap \bE=\cO$;
        \item $x$ is a pseudolimit for some $\cO$-p.c.\ sequence $(x_i)_{i<\lambda}$ in $\bE$ with no pseudolimit in $\bE$.
    \end{enumerate}
    We will say that $x\in \bE_*\setminus \bE$ is symmetric (resp.\ $\cO$-wim over $\bE$) if $\tp(x/\bE)$ is symmetric (resp. $\cO$-wim).
\end{definition}

\begin{remark}
    A type is symmetric if and only if it is $\CH_\bE(\bZ)$-symmetric.
\end{remark}

\begin{remark}
    If $M$ is the partial type of a convex subgroup of $\bE$ and $\cO$ is a convex subring of $\bE$, then $M(\bE_*)$ is a $\cO$-module if and only if $M(\bE)$ is a $\cO$-module, therefore if $x$ is $\cO$-symmetric, then $\Br(x/\bE)(\bE_*)$ is a $\cO$-module (and hence a $\cO_*^-$-module) for all $\bE_*\succ \bE$.
    In fact, clearly if $M(\bE_*)$ is a $\cO$-module, then $M(\bE)=M(\bE_*)\cap\bE$ is a $\cO$-module. Conversely, assume that $M(\bE)$ is a $\cO$-module and let $m \in M(\bE_*)$, $r \in \cO$. Suppose toward contradiction that $rm > c\in \bE$ for some $c\ge M(\bE)$, then $m>c/r>M(\bE)$, because $M(\bE)$ is a $\cO$-module, and thus $m\notin M(\bE_*)$, contradiction.
\end{remark}

\begin{remark}
    By \cite[Lemma~2.12]{freni2024t} if $x$ is $\cO$-wim over $\bE$, then $x$ is symmetric and moreover $\Br(x/\bE)(\bE_*)$ is a $\cO_*$-module for every convex subring $\cO_*$ of $\bE_*$ extending $\cO$. Lemma~\ref{lem:0-sign_vs_wim} below gives a converse: it shows that if $x$ is $\cO$-symmetric, then it is either $\cO$-weakly immediate or what in \cite[Def.~2.3]{tressl2006pseudo} is called a \emph{$\cO$-limit} and in \cite[Def.~41]{freni2025residually} is called a $\cO$-dense type (although there the definition is given in the context of a $T$-convex $\cO$).
\end{remark}

\begin{lemma}\label{lem:0-sign_vs_wim}
    Let $\cO$ be a convex subring of $\bE$.
    If $x \in \bE_* \setminus \bE$ is $\cO$-symmetric, then $\val_{\cO_*^-}(x-\bE) \subseteq \val_{\cO_*^-}(\bE)$, moreover exactly one of the following holds:
    \begin{enumerate}
        \item $\val_{\cO_*^-}(x-\bE)$ has no maximum (so $x$ is $\cO$-wim); 
        \item $\Br(x/\bE)(\bE_*) = a \co_*^-$ for some $a \in \bE$ (so $x$ is a $\cO$-limit).
    \end{enumerate}
    \begin{proof}
        To see that $\val_{\cO_*^-}(x-\bE) \subseteq \val_{\cO_*^-}(\bE)$ note that if $\val(x-c) \notin \val (\bE)$ for some $c \in \bE$, then $x-c$ is trivially not symmetric and so $x$ is not as well.
        
        Suppose that $\val_{\cO_*^-}(x-\bE)$ has a maximum: then there is $c \in \bE$ such that $\rv(x-c)\notin \rv(\bE)$. Moreover it must be $\val(x-c) \in \val (\bE)$, otherwise we would contradict $x+\Br(x/\bE)(\bE_*) = \tp(x/\bE)(\bE_*)$. Finally it must be $\rv(x-c)=\rv(x'-c)$ for all $x\equiv_\bE x'$, otherwise $\Br(x/\bE)(\bE_*)$ would not be a $\cO$-module, thus we have $\tp(x/\bE)(\bE_*)=x+a\co_*^-$.
    \end{proof}
\end{lemma}

\begin{remark}
    Note that $a \co_*^-$ is a $\cO_*^+$-module if and only if $\cO_*^-=\cO_*^+$, that is, when $\bE_*$ does not realize the cut above $\cO$. Thus $x$ is $\cO$-wim if and only if it is symmetric and for all elementary extensions $\bE_*$ of $\bE \langle x \rangle$ and all convex subrings $\cO_*\subseteq \bE_*$ extending $\cO$, $\Br(x/\bE)(\bE_*)$ is a $\cO_*$-module. One may therefore call $\cO$-wim types \emph{strongly $\cO$-symmetric}. We will stick with the established (and shorter) terminology $\cO$-wim.
\end{remark}

If $x$ is $\cO$-wim, then $\res_{\cO_*^-}(\bE \langle x \rangle)=\res_{\cO_*^-}(\bE)$.
The argument is the same as in \cite[Prop.~3.3]{freni2024t}, where however it is given in the context of a $T$-convex $\cO$.

\begin{lemma}\label{lem:wim-cofres_orto}
    Let $\cO$ be a convex valuation ring of $\bE$ and $\cO_*=\CH_{\bE_*}(\cO)$. Suppose $x \in \bE_*$ is $\cO$-wim and $f: \bE_* \to \bE_*$ is an $\bE$-definable function with $f(x) \in \cO_*^{>\co_*}$, then $\res_{\cO_*}(f(x))\in \res_{\cO_*}(\bE)$.
    \begin{proof}
        Set $z\coloneqq f(x)$ and suppose toward contradiction that $\res_{\cO_*}(z) \notin \res_{\cO_*}(\bE)$. Then $f$ is non-constant and we can find $x_- \in \bE^{<x}$, $x_+ \in \bE^{>x}$ such that $f|: (x_-, x_+) \to (z_-, z_+)$ is a diffeomorphism, $z_+, z_- \in \cO_*$, and $z_+/z_- \in \cO_*^{>\co_*}$. By \cite[Lem.~2.11]{freni2024t} we can find p.c.\ sequences 
        $(x_i^-)_{i<\lambda}$ increasing and cofinal in $\bE\cap (x_-, x)$, and $(x_i^+)_{i<\lambda}$ decreasing and coinitial in $\bE\cap (x,x_+)$. Set $z_i^{\pm}\coloneqq f(x_i^{\pm})$, and note that $(z_i^-)_{i<\lambda}$ is cofinal in $(z_-, z)\cap \bE$ and $(z_i^+)_{i<\lambda}$ is coinitial in $(z_+, z)\cap \bE$.
        Since $\res_{\cO_*}(z) \notin \res_{\cO_*}(\bE)$, up to extracting a subsequence we would have that $z_i^{\pm}-z_{j}^{\pm} \in \cO_*\setminus \co_*$ for all $i<j$. But then
        \[x_{i+1}^{\pm}-x_{i}^\pm=(f^{-1})'(\zeta_i^\pm)(z_{i+1}^\pm-z_{i}^\pm)\]
        for some $\zeta_i^{\pm}$ between $z_i^{\pm}$ and $z_{i+1}^\pm$. So $\val(f^{-1})'(\zeta_i^+)$ and $\val(f^{-1})'(\zeta_i^-)$ would both be increasing and $|(f^{-1})'|$ would obtain a minimum value inside $\tp(z/\bE)(\bE_*)$, against the hypotheses that $z \notin \bE$ and that $f$ is $\bE$-definable.
    \end{proof}
\end{lemma}

In \cite{tressl2005model}, Tressl defines the \emph{invariance ring} of a type $p$ as the ring $V_p:=\{a \in \bE: a\Br(p)(\bE) \subseteq \Br(p)(\bE)\}$, that is, the smallest convex subring $\cO$ such that $\Br(p)(\bE)$ is a $\cO$-module. We will adopt a notation similar to the one for the breadth.

\begin{definition}
    If $M$ is a convex partial $\bE$-type defining a subgroup of $\bE$, then we set $\Vr(M)$ to be the type
    \[\Vr(M)(\var{x}) := \{M(\var{x}\cdot m): m \in M(\bE)\}\]
    Notice that $\Vr(\Br(p))(\bE)=V_p$. If $x \in \bE_* \setminus \bE$, we will write $\Vr(x/\bE)$ for $\Vr(\Br(x/\bE))$.
\end{definition}

\subsection{Weak \texorpdfstring{$\cO$}{O}-limits and \texorpdfstring{$\cO$}{O}-absorbed types}

In \cite[Def.~41]{freni2025residually}, an element $x \in \bE_*\setminus \bE$ was called \emph{$\cO$-residually cofinal} if $\res_{\cO_*^-}(x)\notin \res_{\cO_*^-}(\bE)$ and \emph{$\cO$-residually dense} if furthermore $\res_{\cO_*^-}(x)$ had a dense cut over $\res_{\cO_*^-}(\bE)$. In \cite[Def.~2.3]{tressl2006pseudo}, an image under an $\bE$-definable \emph{affine} function of a $\cO$-residually dense element was called a \emph{$\cO$-limit}. Drawing from that, I will call the image under an $\bE$-definable affine function of a \emph{$\cO$-residually cofinal} element a \emph{weak $\cO$-limit}.

\begin{definition}
    If $x \in \bE_*\succ \bE$ and $\cO$ is a convex subring of $\bE$, then we will say that $x$ is \emph{weak $\cO$-limit} if there are $c, d \in \bE$, such that $\res_{\cO_*^-}(d(x-c)) \notin \res_{\cO_*^-}(\bE)$.
\end{definition}

\begin{lemma}\label{lem:weak_O_lim_char}
    The following are equivalent:
    \begin{enumerate}
        \item $x$ is a weak $\cO$-limit;
        \item $\val_{\cO_*^-}(x-\bE) \subseteq \val_{\cO_*^-}(\bE)$ but $\rv_{\cO_*^-}(x-\bE) \not\subseteq\rv_{\cO_*^-}(\bE)$;
        \item there is $d \in \bE$ such that $d\co_*^- \subseteq \Br(x/\bE)(\bE_*) \subsetneq d\cO_*^-$.
    \end{enumerate}
    \begin{proof}
        $(1) \Leftrightarrow (2)$ is obvious. 
        $(1) \Rightarrow (3)$. Let $c, d \in \bE$ be such that $\res_{\cO_*^-}(d(x-c)) \notin \res_{\cO_*^-}(\bE)$. Since by construction $\res_{\cO_*^-}(\bE) \subseteq \res_{\cO_*^-}(\bE_*)$ is a cofinal ordered field extension, we can find $x_-\in \bE^{<x}$, $x_+ \in \bE^{>x}$, such that $(x_+-x_-)d \in \cO_*^-$, and furthermore for all such $x_-, x_+$ we must have $(x_+-x_-)d > \co_*^-$.
        $(3) \Rightarrow (1)$. Since $\Br(x/\bE)(\bE_*) \subsetneq d\cO_*^-$, we must have $x \in \bE+ d\cO_*^-$ so there is $c \in \bE$ such that $(x-c)/d \in \cO_*^-$. On the other hand, if we had $x \in \bE+ d\co_*^-$, then by (3) we would have $\Br(x/\bE)(\bE_*)= d \co_*^-$ and necessarily $x$ is not symmetric (because if it was $x \notin \bE+\Br(x/\bE)(\bE_*)$). But then for some $c\in \bE$, $|x-c| \in (d\co_*^-)^{>d\co}$, which implies $\val_{\cO_*^-}(x-c) \notin\val_{\cO_*^-}(\bE)$, contradicting (2).
    \end{proof}
\end{lemma}    

\begin{remark}
    $x$ is a $\cO$-limit if and only if it is a $\cO$-symmetric weak $\cO$-limit.
\end{remark}

In the interest of uniformity of treatment we also give the following definitions, which might be thought of as an $\cO$-relative weakenings of the notion of symmetric type and will play a quite fundamental role in the next Section.

\begin{definition}
    We say that $x \in \bE_*\setminus \bE$ (or its type over $\bE$) is \emph{(strongly) $\cO$-absorbed} if $\val_{\cO_*^-}(x-\bE)\subseteq \val_{\cO_*^-}(\bE)$. We say that $x$ (or its type) is \emph{weakly $\cO$-absorbed} if $\val_{\cO_*^+}(x-\bE)\subseteq \val_{\cO_*^+}(\bE)$.
\end{definition}

\begin{remark}
    If $x$ is weakly $\cO$-absorbed but not strongly $\cO$-absorbed, then there are $c,d \in \bE$, such that $d(x-c) \in (\cO_*^+ \setminus \cO_*^-) \cup (\co_*^-\setminus \co_*^+)$.
\end{remark}

\begin{remark}
    An $x\in \bE_*\setminus \bE$ is strongly $\CH(\bZ)$-absorbed if and only if it is symmetric.
\end{remark}

\subsection{Distanced sequences}
The main property of symmetric types is that they can be approximated from both sides by sufficiently distanced sequences (Lemma~\ref{lem:at_least_one_good_sequence}).
The main property we will use about strongly $\cO$-absorbed types is that they enjoy a ``$\cO$-weakening'' of such a property (Lemma~\ref{lem:two_O-good_sequences}).

\begin{lemma}\label{lem:at_least_one_good_sequence}
    If $x \in \bE_* \setminus \bE$, then at least one of the following holds:
    \begin{enumerate}
        \item there is a cofinal increasing sequence $(x_{i})_{i<\lambda}$ in $\bE^{<x}$ such that $x_{j}-x_{i}>\Br(x/\bE)(\bE_*)$ for all $i<j<\lambda$;
        \item there is a coinitial decreasing sequence $(x_{i})_{i<\lambda}$ in $\bE^{>x}$ such that $x_{i}-x_{j}>\Br(x/\bE)(\bE_*)$ for all $i<j<\lambda$.
    \end{enumerate}
    If furthermore $x$ has is symmetric, then both (1) and (2) hold.
    \begin{proof}
        If both fail, then there are $x_+ \in \bE^{>x}$ and $x_- \in \bE^{<x}$ such that
        \[\begin{aligned}
            \exists z_- \in \big(x_- + \Br(x/\bE)(\bE_*)\big) \cap \tp(x/\bE)(\bE_*)\\
            \exists z_+ \in \big(x_+ + \Br(x/\bE)(\bE_*)\big) \cap \tp(x/\bE)(\bE_*)
        \end{aligned}\]
        But then $x_+-x_- = (x_+-z_+) - (x_- -z_-) + (z_+-z_-)$ is in $\Br(x/\bE)(\bE_*)$ because all addends are, contradicting the definition of $\Br(x/\bE)$.
    \end{proof}
\end{lemma}

\begin{notation}
    Recall the following colon notation for convex subgroups of $\bE_*$ 
    \[(A:B)=(A:_{\bE_*}B)\coloneqq \{t \in \bE_*: t B \subseteq A\}.\]
\end{notation}

\begin{lemma}\label{lem:two_O-good_sequences}
    Let $x \in \bE_* \setminus \bE$, $\cO$ be a convex subring of $\bE$, and $\cO_*^-:=\CH_{\bE_*}(\cO)$. If $\val_{\cO_*^-}(x-\bE) \subseteq \val_{\cO_*^-}(\bE)$, then there are a cofinal increasing sequence $(x_{i}^-)_{i<\lambda}$ in $\bE^{<x}$ and a coinitial decreasing sequence $(x_{i}^+)_{i<\lambda}$ in $\bE^{>x}$ such that $|x_{j}^{\pm}-x_{i}^{\pm}|>( \Br(x/\bE)(\bE_*): \cO)$ for all $i<j<\lambda$.
    \begin{proof}
        Let $M:=\Br(x/\bE)(\bE_*)$ and suppose toward contradiction that one of such sequences does not exist. Then there would be $c\in \bE$ such that 
        \[\exists z \in (c + (M:\cO)) \cap \tp(x/\bE)(\bE_*).\]
        Note also that $x$ is not symmetric, for otherwise both sequences would exist by Lemma~\ref{lem:at_least_one_good_sequence} since $(M:\cO)\subseteq M$.
        So $x$ is a weak $\cO$-limit and by Lemma~\ref{lem:weak_O_lim_char} for some $d \in \bE$, $d\co_*^- \subseteq M \subsetneq d\cO_*$, and $(M:\cO)=\co_*^- d$.
        Since by hypothesis $z-c \in d' (\cO_*^{>\co_*})$ for some $d' \in \bE$, we get
        \[z-c \in d'\cO_*^{>\co_*} \cap d \co_* \neq \emptyset,\]
        which implies that $d'\prec d$. But then $|z-c| \in \cO_*^-d' \subseteq \CH_{\bE_*} (M \cap \bE)$, which contradicts the definition of $M$.
    \end{proof}
\end{lemma}

\section{Definable functions and absorbed types}\label{sec:def_fn}

In this section we establish some basic but important facts about images of $\cO$-absorbed cuts under definable functions and their derivative. This might be seen as a primitive analysis of the differential valued ordered fields of the form $(\bE\langle x \rangle, \partial_x, \cO_x)$ where $\bE\models T$ and
\begin{enumerate}
    \item $x$ is weakly $\cO$-absorbed over $\bE$ for a convex subring $\cO$ of $\bE$;
    \item  $\partial_x$ is the derivation given by $\partial_x (y) = f'(x)$ where $f: \bE\langle x \rangle \to \bE \langle x \rangle$ is an $\bE$-definable function such that $f(x)=y$ (it is not hard to see that the definition of $\partial_x(y)$ does not  depend on the choice of the particular $f$);
    \item if $x$ is \emph{strongly} $\cO$-absorbed, then $\cO_x:=\CH_{\bE\langle x \rangle}(\cO)$, otherwise $\cO_x=\{t \in \bE\langle x \rangle: |t|<\bE^{>\cO}\}$ and $\cO$ is $T$-convex.
\end{enumerate}

It should be remarked that:
\begin{enumerate}[label=(\Roman*)]
    \item $\bE=\{y \in \bE \langle x \rangle: \partial_x(y) =0\}$;
    \item when $\cO$ is $T$-convex, then (3) just amounts to $\cO_x=\{t \in \bE\langle x \rangle: |t|<\bE^{>\cO}\}$, because in a such case if $x$ is strongly $\cO$-absorbed, then by \cite[Thm.~A]{freni2024t} or by \cite[Prop.~7.1]{tressl2005model} and Lemma~\ref{lem:wim-cofres_orto} we have $\{t \in \bE\langle x \rangle: |t|<\bE^{>\cO}\}=\CH_{\bE\langle x \rangle}(\bE)$.
\end{enumerate}

Almost all the results in this subsection are recovered in a more abstract setting in Section~\ref{sec:fields_of_germs}, but we think this section is good for developing an intuition on relatively concrete examples.

\subsection{Definable functions and strongly \texorpdfstring{$\cO$}{O}-absorbed types}\label{ssec:germs_at_absorbed_types}
Throughout the remainder of this section, we fix an elementary extension $\bE_* \succ \bE$ of models of $T$, and a convex valuation subring $\cO$ of $\bE$. The techniques and arguments here are quite standard and are similar to those found, for example, in \cite{tyne2003t} or \cite{freni2024t}.

\begin{notation}
    We will denote by $\cO_*^-$ and $\cO_*^+$ the convex subrings of $\bE_*$ given by $\cO_*^-=\CH_{\bE_*}(\cO)$, $\cO_*^+=\{t \in \bE_*: |t|<\bE^{>\cO}\}$ and by $\co_*^-$ and $\co_*^+$ their respective maximal ideals.
    Given $x \in \bE_*\setminus \bE$, we will denote by $B_x^+\coloneqq \Br(x/\bE)(\bE_*)$ the set defined by the breadth of $x$ over $\bE$, and by $I_x^-$ the set
    \[I_x^-\coloneqq \left\{t \in \bE_*:\exists x_-\in \bE^{<x}, x_+ \in \bE^{>x}, |t|<1/(x_+-x_-)\right\}.\]
    Note that $t \notin B_x^+ \Leftrightarrow 1/t\in I_x^-$, that is, setting $e\coloneqq \{t: |t|<1/n,\; n \in \bN\}$, $I_x^- = (e:B_x^+)$.
\end{notation}

\begin{remark}\label{rmk:division}
    Note that for any pair $A, B$ of convex subgroups of $\bE_*$, we have, for all $x \in \bE_*^{\neq0}$, $x \in B \cdot (e:A) \Leftrightarrow 1/x \notin (A:B)$.
\end{remark}

Of the following proposition we will mainly use (1) and (2).

\begin{proposition}\label{prop:techincal}
    If $x \in \bE_*\setminus \bE$ is strongly $\cO$-absorbed, $\cO$ is $T$-convex and $f$ is a unary $\bE$-definable function such that $f(x)f'(x) \neq 0$, then
    \begin{enumerate}[label=(\arabic*)]
        \item\label{tech:lder_O-} $1/f^\dagger(x)\notin (B_x^+:\cO) \Leftrightarrow f^{\dagger}(x) \in \cO \cdot I_x^- \Longleftrightarrow \val_{\cO_*^-}(f(x)) \in \val_{\cO_*^-}(\bE)$;
        \item\label{tech:der_O-} $1/f'(x) \notin (B_x^+:\cO) \Leftrightarrow f'(x) \in \cO \cdot I_x^- \Longleftrightarrow f(x) \in \bE + \cO_*^-$;
        \item\label{tech:lder_o+} $1/f^\dagger(x) \notin (B_x^+: \co) \Leftrightarrow f^{\dagger}(x) \in \co \cdot I_x^- \Longleftrightarrow \rv_{\cO_*^+}(f(x)) \in \rv_{\cO_*^+}(\bE)$;
        \item\label{tech:der_o+}  $1/f'(x) \notin (B_x^+: \co)\Leftrightarrow f'(x) \in \co \cdot I_x^- \Longleftrightarrow f(x) \in \bE + \co_*^+$;
        
        \item\label{tech:lder_O+} $1/f^{\dagger}(x) \notin (B_x^+:\cO_*^+) \Leftrightarrow f^\dagger(x) \in \cO_*^+ \cdot I_x^- \Longleftrightarrow \val_{\cO_*^+}(f(x)) \in \val_{\cO_*^+}(\bE)$;
        \item\label{tech:der_O+} $1/f'(x) \notin (B_x^+:\cO_*^+) \Leftrightarrow f'(x) \in \cO_*^+ \cdot I_x^- \Longleftrightarrow f(x) \in \bE+ \cO_*^+$;
        \item\label{tech:lder_o-} $1/f^\dagger(x) \notin (B_x^+: \co_*^-) \Leftrightarrow f^{\dagger}(x) \in \co_*^- \cdot I_x^- \Longleftrightarrow \rv_{\cO_*^-}(f(x)) \in \rv_{\cO_*^-}(\bE)$;
        \item\label{tech:der_o-} $1/f'(x) \notin (B_x^+: \co_*^-)\Leftrightarrow f'(x) \in \co_*^- \cdot I_x^- \Longleftrightarrow f(x) \in \bE + \co_*^-$.
    \end{enumerate}
    Moreover
    \begin{enumerate}[label=(\Alph*)]
        \item\label{techcaveats:A} the $\Longrightarrow$s in \ref{tech:der_O-}-\ref{tech:der_o+} and \ref{tech:der_O+}-\ref{tech:der_o-} hold for all $x \in \bE_* \setminus \bE$ and all convex $\cO$;
        \item\label{techcaveats:B} the $\Longrightarrow$s in \ref{tech:lder_O-} and \ref{tech:lder_O+} hold for all $x \in \bE_*\setminus \bE$ and all $T$-convex $\cO$;
        \item\label{techcaveats:C} the $\Longleftarrow$s hold for all strongly $\cO$-absorbed $x \in \bE_*\setminus \bE$ and all convex $\cO$.
    \end{enumerate}
    \begin{proof}
        Note that in (1) to (8), the first equivalence is a trivial consequence of Remark~\ref{rmk:division}.
        
        \ref{tech:lder_o+}-\ref{tech:der_o+}, $\Longrightarrow$. If $|f^{\dagger}(x)|$ (resp.\ $f'(x)$) is in $\co_*^+ \cdot I_{x}^-= \co \cdot I_x^-$, then there are $r \in \co$, $\tilde{x}_- \in \bE^{<x}$, $\tilde{x}_+\in \bE^{>x}$ such that $|f^{\dagger}(x)| < r/(\tilde{x}_+-\tilde{x}_-)$ (resp. $|f'(x)|<r/(\tilde{x}_+-\tilde{x}_-)$). 
        By o-minimality, we can then pick $x_- \in (\tilde{x}_-, x) \cap \bE$, and $x_+ \in (x,\tilde{x}_+) \cap \bE$ such that on $(x_-, x_+)$, $f$ is monotone, non-zero, differentiable, and $f^{\dagger} < r/(\tilde{x}_+-\tilde{x}_-)$ (resp.\ $f' < r/(\tilde{x}_+-\tilde{x}_-)$). But then, by the mean value theorem for some $\xi \in (x_-, x_+)$
        \[\left|\frac{f(x_+)- f(x_-)}{f(\xi)}\right|= \big|f^{\dagger}(\xi) (x_+-x_-) \big| < r \quad (\text{resp.}\; |f(x_+)-f(x_-)| < r)\]
        and hence $f(x_+) \sim f(x_-)$ (resp.\ $f(x_+)-f(x_-) \in \co$).

        \ref{tech:lder_o+}-\ref{tech:der_o+}, $\Longleftarrow$. Suppose that $|f^{\dagger}(x)| > \co \cdot I_x^-$ (resp.\ $|f'(x)|>\co \cdot I_x^-$), then $1/f^\dagger \in (B_x^+:\co)$ (resp.\ $1/|f'(x)|\in (B_x^+:\co)$) either on $(x_-, x) \cap \bE$ or on $(x, x_+) \cap \bE$. Suppose without loss of generality that it is on $(x_-, x)\cap \bE$. Since $\val_{\cO_*^-}(x-\bE) \subseteq \val_{\cO_*^-}(\bE)$, by Lemma~\ref{lem:two_O-good_sequences} we can find a cofinal increasing sequence $(x_i^-)_{i<\lambda}$ in $(x_-, x)$ such that $|x_{j}-x_i|>(B_x^+:\cO)$. For such sequence
        \[\left|\frac{f(x_j^-)- f(x_i^-)}{f(\xi_{i,j})}\right|= \big|f^{\dagger}(\xi_{i,j}) (x_j^- -x_i^-) \big| > \co \quad (\text{resp.}\; |f(x_j^-)-f(x_i^-)| >\co)\]
        but then $\rv_{\cO}(f(x_i^-))$ is eventually strictly monotone and $\rv_{\cO_*^+}(f(x))\notin \rv_{\cO_*^+}(\bE)$ (resp.\ $f(x)\notin \bE + \co_*^+$).
 
        \ref{tech:lder_o-}-\ref{tech:der_o-} $\Longrightarrow$. Suppose that $f^\dagger(x) \in \co_*^- \cdot I_x^-$ (resp.\ $f'(x)\in \co_*^-\cdot I_x^-$). Then for some $x_+\in \bE^{>x}$, $x_-\in \bE^{<x}$, $f^\dagger(x)(x_+-x_-)\in \co_*^-$ (resp.\ $f'(x)(x_+-x_-)\in \co_*^-$) and furthermore by o-minimality we can pick such $x_-, x_+$ so that $f$, $|f^\dagger|$, and $|f'|$ are monotone. Thus either for all $t\in (x_-, x)$ or for all $t \in (x, x_+)$, we have that 
        $f^\dagger(t) (x_+-x_-) \in \co_*^-$ (resp.\ $f'(t) (x_+-x_-) \in \co_*^-$). Say it is for all $t\in (x_-, x)$, then the mean value theorem ensures that there is $\xi \in (x_-, x)$ s.t.\ 
        \[\frac{f(x)-f(x_-)}{f(\xi)}=f^\dagger(\xi)(x-x_-) \in \co_*^- \quad (\text{resp.}\;f(x)-f(x_-) \in \co_*^-),\]
        because $0<x-x_-<x_+-x_-\in \co_*^-/f^\dagger(\xi)$ (resp.\ $\in \co_*^-/ f'(\xi)$).
        
        \ref{tech:lder_o-}-\ref{tech:der_o-} $\Longleftarrow$. Conversely if $f^\dagger(x) \notin \co_*^- \cdot I_x^-$ (resp.\ $f'(x) \notin \co_*^- \cdot I_x^-$), then for all $x_+\in \bE^{>x}$, $x_-\in \bE^{<x}$, $f^\dagger(x)(x_+-x_-)\notin\co_*^-$ (resp.\ $f'(x) \notin \co_*^-\cdot I_x^-)$.
        Without loss of generality, we can assume that $|f'|$ is increasing on some $\bE$-definable neighborhood $(\tilde{x}_-, \tilde{x}_+)$ of $x$, furthermore since $x$ is strongly $\cO$-absorbed, we can assume that $\tilde{x}_+-x\succeq_{\cO_*^-}\tilde{x}_+-\tilde{x}_-$. 
        
        Since by o-minimality $(\tilde{x}_+-\tilde{x}_-)f^\dagger(t) \notin\co_*^-$ (resp.\ $(\tilde{x}_+-\tilde{x}_-)f'(t)\notin\co_*^-$) for all $t$ in some possibly smaller $\bE$-definable neighborhood of $x$, we can pick $x_- \in (\tilde{x}_-, x) \cap \bE$, such that $f^\dagger(x_-)(\tilde{x}_+-\tilde{x}_-)\notin \co_*^-$ (resp.\ $f'(x_-)(\tilde{x}_+-\tilde{x}_-)\notin \co_*^-$).

        Now note that since $x_-<x$ we will have $\tilde{x}_+-x_-\succeq_{\cO_*^-} \tilde{x}_+-\tilde{x}_-$, so $f^\dagger(x_-)(\tilde{x}_+-x_-)\notin \co_*^-$ (resp.\ $f'(x_-)(\tilde{x}_+-x_-)\notin \co_*^-$).
        
        But then since $|f^\dagger|$ (resp.\ $|f'|$) is increasing on $(\tilde{x}_-, \tilde{x}_+)$, $(\tilde{x}_+-x_-)f^\dagger(t)\notin\co_*^-$ (resp.\ $(\tilde{x}_+-x_-)f'(t)\notin\co_*^-$) for all $t \in (x_-,\tilde{x}_+)$. Now by Lemma~\ref{lem:two_O-good_sequences} we can pick sequences $(x_i^-)_{i<\lambda}$ increasing cofinal in $(x_-, x)\cap \bE$ and $(x_i^+)_{i<\lambda}$ decreasing coinitial in $(x, \tilde{x}_+)\cap \bE$ such that for all $i\neq j$, $|x_i-x_j|>(B_x^+:\cO)$.
        The mean value theorem then ensures that both sequences $\big(\rv_{\cO} f(x_i^-)\big)_{i<\lambda}$ and $\big(\rv_{\cO} f(x_i^+)\big)_{i<\lambda}$, (resp.\ $(f(x_i^{+})+\co)_{i<\lambda}$ and $(f(x_i^{-})+\co)_{i<\lambda}$) are monotone, so $\rv_{\cO_*^-}(f(x))\notin \rv_{\cO_*^-}(\bE)$ (resp.\ $f(x) \notin \bE+\co_*^-$).

        \ref{tech:der_O-}. Argue as in \ref{tech:der_o+}, replacing $\co_*^+$ with $\cO_*^-$.
        
        \ref{tech:lder_O-} $\Longleftarrow$. Argue as in \ref{tech:lder_o+} $\Longleftarrow$, replacing $\co_*^+$ with $\cO_*^-$ and $\rv_{\co_*^+}$ with $\val_{\cO_*^-}$.
        
        \ref{tech:der_O+}. Argue as in \ref{tech:der_o-}, replacing $\co_*^-$ with $\cO_*^+$.

        \ref{tech:lder_O+} $\Longleftarrow$. Argue as in \ref{tech:lder_o-} $\Longleftarrow$, replacing $\co_*^-$ with $\cO_*^+$ and $\rv_{\cO_*^-}$ with $\val_{\cO_*^+}$.

        \ref{tech:lder_O-} and \ref{tech:lder_O+} $\Longrightarrow$. If the theory $T$ is exponential, then this follows from (2) and (6). If the theory $T$ is power-bounded, then the conclusions of (1) and (5) always hold given the hypothesis on $x$.
    \end{proof}
\end{proposition}

\begin{remark}
    In Proposition~\ref{prop:techincal}, (1)-(3) and (7) are equivalences of types over $\bE$ (pro-definable sets), the remaining are equivalences of type unions (ind-pro-definable sets).
\end{remark}

\begin{lemma}\label{lem:br_image_new}
    Let $x\in \bE_* \setminus \bE$ and $f: \bE_* \to \bE_*$ be $\bE$-definable. Suppose that $\val_{\cO_*^-}(x-\bE)$ and $\val_{\cO_*^-}(f(x)-\bE)$ are both contained in $\val_{\cO_*^-}(\bE)$. Then $f'(x) \cdot \cO \cdot I_{f(x)}^- = \cO \cdot I_x^-$, so
    $f'(x) (B_x^+: \cO) = (B_{f(x)}^+: \cO)$.
    \begin{proof}
         Since $f(x)$ is $\cO_*^-$-absorbed, we have that $\val_{\cO_*^-}(f(x)-f(c)) \in \val_{\cO_*^-}(\bE)$ for all $c \in \bE$, thus by Proposition~\ref{prop:techincal}\ref{tech:lder_O-} and \ref{techcaveats:C}
         \[\frac{f'(x)}{f(x)-f(c)} \in \cO \cdot I_x^-.\]
         Now if $f(x)$ is symmetric, then $f(x)-f(c) \notin B_x^+$ for all $c \in \bE$ and thus $1/(f(x)-\bE) \subseteq I_x^-$ cofinally, so $f'(x) \cdot I_{f(x)}^- \subseteq \cO \cdot I_x^-$.
         If instead $f(x)$ is not symmetric, there is $c \in \bE$ such that $f(x)-f(c) \in B_{f(x)}^+ \setminus B_{f(x)}^-$ whence $1/(f(x)-f(c)) \in I_{f(x)}^+\setminus I_{f(x)}^-$ and again we derive $f'(x) \cdot I_{f(x)}^-\subseteq \cO \cdot I_x^-$. To prove the other inclusion, apply the already derived inclusion to $y=f(x)$ and $f^{-1}$.
    \end{proof}
\end{lemma}

\begin{corollary}
    If $x,y\in\bE_*$ are symmetric over $\bE$ and $y \in \bE \langle x \rangle$, then $\Vr(x/\bE)=\Vr(y/\bE)$.
    \begin{proof}
        Write $y=f(x)$ for some $\bE$-definable function $f$. Then apply Lemma~\ref{lem:br_image_new} with $\cO=\CH(\bZ)$.
    \end{proof}
\end{corollary}

\begin{lemma}\label{lem:sOabs_to_sOabs_der}
    Let $x\in \bE_* \setminus \bE$ and let $f, g: \bE_* \to \bE_*$ be $\bE$-definable functions. Suppose that $x$ and $f(x)$ are strongly $\cO$-absorbed over $\bE$ and $f(x)\equiv_\bE g(x)$. If $r \in \bE$ is such that $B_{f(x)}^+\subseteq r \cdot (B_{f(x)}^+:\cO)$, then $|f'(x)-g'(x)| < |f'(x)|/r$.
    \begin{proof}
        Let $h=f-g$ so that $h(x) \in B_{f(x)}^+$.
        Suppose that $|h'(x)|\ge r |f'(x)|$ for some $r \in \bE^{>0}$, so that $|h'/f'|\ge r$ in an $\bE$-definable neighborhood $(x_-, x_+)$ of $x$ and for all $\tilde{x}$ in such a neighborhood, by the (Cauchy) mean value theorem, there is $\xi$ between $x$ and $\tilde{x}$ such that
        \[\left |\frac{h(x)-h(\tilde{x})}{f(x)-f(\tilde{x})} \right| = \left| \frac{h'(\xi)}{f'(\xi)} \right| \ge r.\]
        Up to precomposing $f$ and $g$ with a sign-change and replacing $x$ with $-x$, we can assume without loss of generality that $|h|$ is strictly increasing and not vanishing in a neighborhood of $x$.
        But then for all large enough $c \in \bE^{<x}$ we get 
        \[|h(x)| = \max\{|h(x)|, |h(c)|\} \ge |h(x)-h(c)| \ge |f(x)-f(c)|\cdot r.\]
        But since $f(x)$ is strongly $\cO$-absorbed, $|f(x)-f(c)|>(B_{f(x)}^+:\cO)$, so $h(x) \notin r\cdot (B_{f(x)}^+:\cO)$. Since $h(x) \in B_{f(x)}^+$ by hypothesis, we have that
        $B_{f(x)}^+\not\subseteq r \cdot \big(B_{f(x)}^+:\cO)$, contradicting the hypothesis on $r$.
    \end{proof}
\end{lemma}

\begin{corollary}\label{cor:sOabs_to_sOabs_der}
    Let $x \in \bE_*\setminus \bE$ and $f,g,h: \bE_* \to \bE_*$ be $\bE$-definable functions. Then the following hold:
    \begin{enumerate}
        \item if $x \in \bE_* \setminus \bE$ is strongly $\cO$-absorbed and $h(x) \in (B_x^+:\cO)$, then $h'(x) \in \co_*^-$, furthermore if $\cO$ is $T$-convex, then $h'(x) \in \co_*^+$;
        \item if $x \in \bE_* \setminus \bE$ is $\cO$-wim and $f(x)\equiv_\bE g(x)$ are strongly $\cO$-absorbed, then $f(x), g(x)$ are $\cO$-wim and moreover $f'(x) \sim_{\cO_*^-} g'(x)$.
    \end{enumerate}
    \begin{proof}
        (1) is a direct application of Lemma~\ref{lem:sOabs_to_sOabs_der} with $f(x)=x$ and $g(x)=x+h(x)$. The ``furthermore'' follows from the fact that if $\cO$ is $T$-convex and $x$ is strongly $\cO$-absorbed, then by \cite[Thm.~A]{freni2024t} $\co_*^-\cap \bE \langle x \rangle = \co_*^+\cap \bE \langle x \rangle$.
        (2) if $x$ is $\cO$-wim, then $B_x^+$ is a $\cO$-module. Furthermore since $f(x), g(x)$ are strongly $\cO$-absorbed, by Lemma~\ref{lem:wim-cofres_orto}, we must have that $f(x)$ and $g(x)$ are $\cO$-wim as well, so $B_{f(x)}$ is a $\cO$-module as well, but then Lemma~\ref{lem:sOabs_to_sOabs_der} implies $|f'(x)-g'(x)|<f'(x)/r$ for all $r \in \cO$, whence $f'(x) \sim_{\cO_*^-} g'(x)$.
    \end{proof}
\end{corollary}

\begin{lemma}\label{lem:vder_v}
    Let $x\in \bE_* \setminus \bE$ be strongly $\cO$-absorbed and let $f: \bE_* \to \bE_*$ be $\bE$-definable. If $\val_{\cO_*^-}(f'(x)) \in \val_{\cO_*^-}(\bE)$, then $\val_{\cO_*^-}(f(x)) \in \val_{\cO_*^-}(\bE)$.
    \begin{proof}
        Up to precomposing with a sign change and replacing $x$ with $-x$, we can assume without loss of generality that $|f|$ is strictly increasing in an $\bE$-definable neighborhood of $x$.
        Since $\val_{\cO_*^-}(f'(x)) \in \val_{\cO_*^-}(\bE)$,
        we can find $x_- \in \bE^{<x}$ and $x_+ \in \bE^{>x}$, such that on $[x_-, x_+]$, $|f|$ is differentiable and $\val_{\cO_*}(f')$ is constantly equal to $\val_{\cO_*}(f'(x)) \in \val_{\cO_*^-}(\bE)$.
        Note that $f(x_-) \preceq_{\cO_*^-} f(x) \preceq_{\cO_*^-} f(x_+)$ and that if one of the two $\preceq_{\cO_*^-}$ is not strict, the thesis holds, so we might assume $f(x_-) \prec_{\cO_*^-} f(x) \prec_{\cO_*^-} f(x_+)$.
        But then
        \[f(x) \sim_{\cO_*^-} f(x)-f(x_-)=f'(\xi)(x-x_-) \in \val_{\cO_*^-}(\bE)\]
        for some $\xi \in (x_-, x)$, because $\val_{\cO_*^-}(f'(\xi))\in \val_{\cO_*^-} (\bE)$ and $\val_{\cO_*^-}(x-x_-)\in\val_{\cO_*^-}(\bE)$.
    \end{proof}
\end{lemma}

\begin{proposition}\label{prop:absorbed_der_char}
    Let $x\in \bE_* \setminus \bE$ be $\cO_*^-$-absorbed and let $f: \bE_* \to \bE_*$ be $\bE$-definable. Then the following are equivalent:
    \begin{enumerate}
        \item $\val_{\cO_*^-}(f(x)-\bE)\subseteq \val_{\cO_*^-}(\bE)$
        \item $\val_{\cO_*^-}(f'(x)) \in \val_{\cO_*^-}(\bE)$.
    \end{enumerate}
    \begin{proof}
        $(2)\Rightarrow (1)$ is by Lemma~\ref{lem:vder_v}, because $f'=(f-c)'$ for all $c \in \bE$. 
        
        For $(1) \Rightarrow (2)$, note that if (1) holds, then by Lemma~\ref{lem:wim-cofres_orto} there are two possibilities: $x$ and $f(x)$ are either both weak $\cO$-limits (Case 1) or both $\cO$-wim (Case 2).

        \textit{Case 1.} Both $\val_{\cO_*^-}(f(x)-\bE)$ and $\val_{\cO_*^-}(x-\bE)$ have a maximum. Then there is $c_0$ such that $\val_{\cO_*^-}(x-c_0)$ is maximal and thus $\cO \cdot I_x^- = \cO_*^- / (x-c_0)$ and $\co \cdot I_x^-=\co_*^+/(x-c_0)$ and there is $c_1$ such that $\val_{\cO_*^-}(f(x)-c_1) \in \val_{\cO_*^-}(\bE)$ is maximal.
        But then, by Proposition~\ref{prop:techincal}\ref{tech:lder_O-}\ref{techcaveats:C}, $(1)$ implies $f'(x) \preceq_{\cO_*^-} (f(x)-c_1)/(x-c_0)$. On the other hand, since $\rv_{\cO_*^-}(f(x))\notin \rv_{\cO_*^-}(\bE)$, a fortiori $\rv_{\cO_*^+}(f(x))\notin \rv_{\cO_*^+}(\bE)$, so by Proposition~\ref{prop:techincal}\ref{tech:der_o+}\ref{techcaveats:C}, $f'(x) \succeq_{\cO_*^-} (f(x)-c_1)/(x-c_0)$.
        So $\val_{\cO_*^-}(f'(x)) = \val_{\cO_*^-}(f(x)-c_1) - \val(x-c_0)$, so $(2)$ holds.
        
        \textit{Case 2.} Since $f(x)$ and $x$ are $\cO$-wim we get for all $\bE$-definable $h$ such that $h(x) \in \Br(x/\bE)(\bE\langle x \rangle)$, $f(x+h(x))\equiv_\bE f(x)$, so by Corollary~\ref{cor:sOabs_to_sOabs_der}(2) and (1), 
        \[f'(x) \sim_{\cO_*^-} f'(x+h(x)) (1+h'(x)) \sim_{\cO_*^-} f'(x+h(x)),\]
        thus $\rv_{\cO_*^-}f'\big(x+\Br(x/\bE)(\bE\langle x \rangle)\big)$ is a singleton and a fortiori it must be $\val_{\cO_*^-}(f'(x)) \in \val_{\cO_*^-}(\bE)$.
    \end{proof}
\end{proposition}

The second point of the following Proposition had already been observed by Tressl (see \cite[7.1]{tressl2005model}).

\begin{proposition}\label{prop:breadth-ortho}
    Let $x\in \bE_*\setminus \bE$ be symmetric over $\bE$. Let $f:\bE_* \to \bE_*$ be $\bE$-definable and suppose that $f(x) \in B_x^+ \setminus B_x^-$, where $B_x^-\coloneqq\CH_{\bE_*}(B_x^+\cap \bE)$. Then:
    \begin{enumerate}
        \item $1/f^\dagger(x) \in B_x^+\setminus B_x^-$;
        \item $1/f'(x) \in \Vr(x/\bE)(\bE_*)\setminus \CH_{\bE_*}(\Vr(x/\bE)(\bE))$.
    \end{enumerate}
	\begin{proof}
        To simplify the notation let $V_x^+\coloneqq \Vr(x/\bE)(\bE_*)$ and $V_x^-\coloneqq \CH_{\bE_*}(\Vr(x/\bE)(\bE))$.
		Let $f:\bE_* \to \bE_*$ be $\bE$-definable and suppose that $f(x) \in \Br(x/\bE)(\bE \langle x \rangle)$. Then by Lemma~\ref{lem:sOabs_to_sOabs_der} we have $|f'(x)| < 1 / r$ for all $r \in V_x^-$, that is, $1/|f'(x)|>V_x^-$ and in particular $f'(x)B_x^-\subseteq B_x^-$. It follows that $1/f^\dagger(x)\notin B_x^-$, because otherwise $f(x) \in (f'(x)B_x^-)\setminus B_x^-$, contradicting $f'(x)B_x^-\subseteq B_x^-$. On the other hand $1/f^\dagger(x) \in B_x^+$ by Proposition~\ref{prop:techincal}(1) and we have shown (1).
        
        To finish proving (2) we only need to prove that $1/f'(x) \in V_x^+$. 
        Suppose that $1/|f'(x)|>a$ for some $a\in \bE^{>V_x^+}$, then
        \[\big(B_x^+ \setminus (f'(x)B_x^+)\big) \cap \bE \supseteq \big(B_x^+ \setminus a^{-1}B_x^+\big) \cap \bE \neq \emptyset.\]
        But then since $\val_{V_x^-} f(x) \notin \val_{V_x^-} (\bE)$, by Proposition~\ref{prop:techincal}(1) we have $f(x) \in f'(x) B_x^+ \subseteq a^{-1}B_x^+$ so we can find $d \in B_x^+\cap \bE$ such that $d >|f(x)|$, contradicting $|f(x)| >B_x^+\cap \bE$. 
	\end{proof}
\end{proposition}

\subsection{Definable functions at the \texorpdfstring{$\cO$}{O}-special type}\label{ssec:germs_at_special_types}
In this section we give results similar to the ones in the section above, but for types that are weakly $\cO$-absorbed.

\begin{lemma}\label{lem:absorption_special}
    Let $b \in \bE_*$ be such that $b \in (\cO_*^+\setminus \cO_*^-)\cup (\co_*^- \setminus \co_*^+)$ and let $f$ be a unary $\bE$-definable function with $f(b)f'(b) \neq 0$. Then
    \begin{enumerate}
        \item $1/f'(b) \notin \co_*^+ \Longleftrightarrow f(b) \in \bE+\cO_*^+$;
        \item $1/f'(b) \notin \cO_*^+ \Longleftrightarrow f(b) \in \bE + \co_*^+$;
        \item $1/f^\dagger(b) \notin \cO_*^+\Longleftrightarrow \rv_{\cO_*^+}(f(b)) \in \rv_{\cO_*^+}(\bE)$.
    \end{enumerate}
    If furthermore $\cO$ is $T$-convex, or $T$ is exponential and $\exp(\cO)\subseteq \cO$, then
    \begin{enumerate}[resume]
        \item $1/f^\dagger(b) \notin \co_*^+ \Longleftrightarrow \val_{\cO_*^+}(f(b)) \in \val_{\cO_*^+}(\bE)$;
        \item $\val_{\cO_*^+}(f(b)-\bE) \subseteq \val_{\cO_*^+}(\bE) \Longleftrightarrow \val_{\cO_*^+}(f'(b)) \in \val_{\cO_*^+}(\bE)$.
    \end{enumerate}
    \begin{proof}
        I prove the statement for $b$ such that $\cO<b<\bE^{>\cO}$. The other cases can be deduced from that using $f'(1/b)=-b^2 \cdot \partial_b(f(1/b))$, and $f'(-b)=-\partial_b(f(-b))$.
        
        I only give the proof of (1), (2) and (5), as the proof of (3) is similar to the proofs of (2) and (4) follows from (1) when the theory is exponential and is trivial when the theory is power-bounded.

        (1). Suppose that $1/f'(b) \in \co_*^+$ so $|f'(b)|>\cO_*^+$ and $|f'|> \cO_*^+$ in an $\bE$-definable neighborhood of $b$. Now pick sequences $(b_i^-)_{i<\lambda_-}$ cofinal increasing in $\cO$ and $(b_i^+)_{i<\lambda_+}$ coinitial decreasing in $\bE^{>\cO}$ such that $|b_i^\pm -b_j^\pm|>1$ and note that then $f(b_i^\pm)-f(b_j^\pm)>\cO$, so $f(b_i^+)+\cO_*^+$ is strictly decreasing and $f(b_i^-)+\cO_*^+$ is strictly increasing, so $(f(b)+\cO_*^+) \cap \bE=\emptyset$.
        
        Conversely if $1/f'(b) \notin \co_*^+$, then $f'(b) \in \cO_*^+$, so since $|f'|$ must be monotone or constant in an $\bE$-definable neighborhood $(b_-, b_+)$ of $b$, we have that on one among $(b_-, b)$ and $(b, b_+)$ we must have $f' \in \cO_*^+$. But then the mean value theorem implies that for some $\sigma \in \{\pm\}$, $|f(b_{\sigma})-f(b)| \in \cO_*^+$.

        (3). Note that $\Longrightarrow$ follows from Proposition~\ref{prop:techincal}(3) after observing that $I_b^-=\co_*^+$.
        Conversely suppose that $1/f'(b) \in \cO_*^+$, so $f'(b) \notin \co_*^+$. Arguing as in (1) we see that there are sequences $(b_i^-)_{i<\lambda_-}$ cofinal increasing in $\cO$ and $(b_i^+)_{i<\lambda_+}$ coinitial decreasing in $\bE^{>\cO}$ such that $|b_i^\pm -b_j^\pm|>1$. For one $\sigma \in \{\pm\}$, we must have by the mean value theorem that $f(b_i^\sigma)-f(b_j^\sigma)>\co$ for all large enough $i,j$, but then $(f(b_i^\sigma)+\co)_{i<\lambda_\sigma}$ is eventually increasing, so $f(b) \notin \bE+\co_*^+$.

        (5). We know by \cite[Thm.~A]{freni2024t} that $f(b)$ cannot be $\cO$-weakly immediate or a weak $\cO$-limit. So $\val_{\cO_*^+}(f(b)-\bE)$ always has a maximum $\val_{\cO_*^+}(f(b)-c)$ and $\rv_{\cO_*^-}(f(b)-c) \notin \rv_{\cO_*^-}(\bE)$, hence by (4) $(f(b)-c)/f'(b) \in \cO_*^+$, so $f(b)-c \succeq_{\cO_*^+}f'(b)$.
        
        If $\val_{\cO_*^+}(f(b)-c) \in \val_{\cO_*^+}(\bE)$, then
        $(f(b)-c)/f'(b) \in \cO_*^+\setminus \co_*^+$, so $\val_{\cO_*}(f'(b)) = \val_{\cO_*^+}(f(b)-c)\in \val_{\cO_*}(\bE)$.

        Suppose that $\val_{\cO_*^*}(f'(b))\in \val_{\cO_*^+}(\bE)$, so either $f'(b)=0$ or $f'(b)d \in \cO_*^+\setminus \co_*^+$ for some $d \in \bE$. The case $f'(b)=0$ is trivial, so we may assume $f'(b)d \in \cO_*^+\setminus \co_*^+$.
        Since $f'=(f-c)'$ for all $c \in \bE$, it suffices to show that $\val_{\cO_*^+}(f'(b))\in \val(\bE)$ implies $\val_{\cO_*^+}(f(b)) \in \val(\bE)$ or equivalently by (2) that $f^\dagger(b) \in \cO_*^+$.
        
        Since $df'(b)\in \cO_*^+$, it suffices to show that $df'(b) \notin \co_*^+\Rightarrow df(b) \notin \co_*^+$, but this follows from \cite[Lem.~2.24]{freni2024t}.    
    \end{proof}
\end{lemma}

\begin{remark}
    If $\cO$ is $T$-convex and $b$ is as in Lemma~\ref{lem:absorption_special}, then $y \in \bE \langle b\rangle\setminus \bE$ is weakly $\cO$-absorbed if and only if there are $c,d \in \bE$ such that 
    \[(y-c)d \in (\cO_*^+ \setminus \cO_*^-) \cup (\co_*^- \setminus \co_*^+).\]
\end{remark}

\section{Valued differential fields of germs}\label{sec:fields_of_germs}

In this section we abstract the main properties of the valued (ordered) differential fields of germs deduced in the previous section, and develop their basic theory. In this sense, methodologically, we draw inspiration from \cite{aschenbrenner2019asymptotic}. In fact, several notions introduced here can be thought as relaxations of notions studied in \cite{aschenbrenner2019asymptotic} to the case where the constants of the derivation are not included in the valuation ring, as we will show in Subsection~\ref{ssec:few-constants}.

The main result of the section is Corollary~\ref{cor:main-cor}, which will be used in Section~\ref{sec:transserial}. Toward that Corollary not all the theory theory developed in this section is needed, but many things seem to be interesting in their own right.

\subsection{Some terminology on valued differential fields}
In this subsection we define and recall some basic but fundamental notions in valued differential fields.

\begin{definition}
    If $(\bE_*, \cO_*)$ is a field and $\bE\le \bE_*$ is a subfield, we will define the \emph{lower $\cO_*$-breadth} and \emph{upper $\cO_*$-breadth} over $\bE$ of an element $x \in \bE_*$ respectively as
    \[\begin{aligned}
        \br_{\cO_*}(x/\bE)\coloneqq &\,\bigcap_{c \in \bE}(x-c)\co_*=\{y\in \bE_*: y \prec x-\bE\}\\
        \Br_{\cO_*}(x/\bE)\coloneqq &\,\bigcap_{c \in \bE}(x-c) \cO_* = \{y \in \bE_* : y \preceq x-\bE\}
    \end{aligned}\]
    We will say that $y$ is \emph{$\cO_*$-absorbed} over $\bE$ if $\val(y-\bE)\subseteq  \val(\bE)$. We will also define the \emph{lower and upper $\cO_*$-breadth} of a derivation $\partial$ on $\bE_*$ respectively as
    \[
    \begin{aligned}
        \br_{\cO_*}(\partial)\coloneqq &\,(\co_*:_{\bE_*} \partial \cO_*) = \{y \in \bE_*: y \partial\cO_*\prec 1\},\\
        \Br_{\cO_*}(\partial)\coloneqq &\,(\co_* :_{\bE_*} \partial \co_*) = \{y \in \bE_*: y \partial \co_* \prec 1\}.
    \end{aligned}\]
\end{definition}

\begin{remark}\label{rmk:basic_der-lder-O-o_equalities}
    Note that in any valued differential field $(\bE_*, \cO_*, \partial)$ we have 
    \[\cO_* \cdot \lder(\cO_*\setminus \co_*)=\cO_* \cdot \partial \cO_* \quad \text{and} \quad \cO_*\cdot \lder(1+\co_*)=\cO_*\cdot \partial\co_*.\]
\end{remark}

\begin{remark}\label{rmk:basic_inclusion}
    Note that if $(\bE_*, \cO_*)$ is a valued field and $\partial: \bE_*\to \bE_*$ is a derivation, then 
    \[\co_*\partial \co_* \subseteq \co_*\partial \cO_*\subseteq \cO_*\partial \co_*\subseteq \cO_* \partial \cO_*,\]
    in fact if $\varepsilon \prec 1$, $r \preceq 1$, then $\varepsilon \partial r=\partial(r\varepsilon)-r\partial\varepsilon \in \partial \co_* + \cO_*\partial \co_*\subseteq \cO_*\partial\co_*$. So $\cO_*\partial \co_*$ can only be $\cO_* \partial \cO_*$ or $\co_*\partial\cO_*$.
\end{remark}

\begin{lemma}\label{lem:br_basic}
    Let $(\bE_*, \cO_*, \partial)$ be a valued differential field, $\bE\coloneqq \Kr(\partial)$, $a \in \bE_*^{\neq0}$, and $x \in \bE_*\setminus \bE$.
    \begin{enumerate}
        \item $\br_{\cO_*}(x/\bE)\subseteq \Br_{\cO_*}(x/\bE)$ and $\br_{\cO_*}(\partial)\subseteq \Br_{\cO_*}(\partial)$;
        \item for all $a \in \bE_*^{\neq 0}$, $\br_{\cO_*}(\partial/a) = a\br_{\cO_*}(\partial)$ and $\Br_{\cO_*}(\partial/a) = a\Br_{\cO_*}(\partial)$;
        \item if $\val(x-\bE) \subseteq \val(\bE)$, then $\br_{\cO_*}(\partial/\partial x)\subseteq \br_{\cO_*}(x/\bE)$ and $\Br_{\cO_*}(\partial/\partial x) \subseteq \Br_{\cO_*}(x/\bE)$;
        \item if $\rv(x-\bE)\subseteq \rv(\bE)$, then $\Br_{\cO_*}(\partial/\partial x)\subseteq \br_{\cO_*}(x/\bE)=\Br_{\cO_*}(x/\bE)$.
    \end{enumerate} 
    \begin{proof}
        (1) and (2) are trivial. As for (3) and (4) note that for all $c \in \bE$, $\val(x-c)\in \val (\bE)$ (resp.\ $\rv(x-c)\in \rv(\bE)$) implies $\lder(x-c)\in \cO_*\partial \cO_*$ (resp.\ $\lder(x-c)\in \cO_*\partial \co_*$) whence $1 \in (x-c)\cO_*\partial \cO_*/\partial x$ (resp.\ $1 \in (x-c)\cO_*\partial \co_*/\partial x$) and thus $x-c\succ \br_{\cO_*}(\partial/\partial x)$ (resp.\ $x-c\succ \Br_{\cO_*}(\partial/\partial x)$). The last equality in (4) follows from the fact that $\rv(x-\bE)\subseteq \rv(\bE^{\neq0})$ if and only if $\val(x-\bE)$ has no maximum. The second inclusion in (3) follows from the first and Remark~\ref{rmk:basic_inclusion} when $\Br_{\cO*}(x/\bE)\neq \br_{\cO_*}(x/\bE)$ and is an instance of (4) when $\Br_{\cO*}(x/\bE) = \br_{\cO_*}(x/\bE)$.
    \end{proof}
\end{lemma}

In the next subsections it will be good to keep in mind a very basic trichotomy on derivations in a valued field.

\begin{lemma}\label{lem:inclusion_cases}
    If $(\bE_*, \cO_*)$ is a valued field and $\partial: \bE_*\to \bE_*$ is a derivation, then one and only one of the following holds:
    \begin{enumerate}
        \item[(W)]\label{der-type:W} $\cO_*\partial \cO_*$ is not principal and $\co_* \partial \co_* = \co_* \partial \cO_*= \cO_*\partial\co_*=\cO_*\partial\cO_*$;
        \item[(G)]\label{der-type:G} $\co_*\partial\co_*=\co_*\partial \cO_*\subsetneq \cO_*\partial \co_*=\cO_*\partial \cO_*$;
        \item[(R)]\label{der-type:R} $\co_*\partial\co_*=\co_*\partial \cO_*= \cO_*\partial \co_*\subsetneq\cO_*\partial \cO_*$.
    \end{enumerate}
    \begin{proof}
        If $\cO_*\partial \cO_*$ is not principal, then $\co_*\partial\cO_*=\cO_*\partial\cO_*$ so by Remark~\ref{rmk:basic_inclusion} we have the last three equalities in (W), but then $\cO_*\partial \co_*$ is also not principal and thus $\co_*\partial\co_*=\cO_*\partial \co_*$ and we have the first one.

        Now suppose that $\cO_*\partial \cO_*= \cO_* \partial r$ for some $r \in \cO_*$. Then we must have $\co_*\partial \cO_*=\co_*\partial r \subsetneq \cO_*\partial r = \cO_*\partial \cO_*$.
        
        If there is $\varepsilon \in \co_*$ such that $\cO_* \partial \cO_* = \cO_* \partial \varepsilon$, then $\cO_*\partial \varepsilon = \cO_* \partial \co_* = \cO_*\partial \cO_*$ and on the other hand $\co_*\partial \co_*=\co_*\partial \varepsilon=\co_*\partial \cO_*$ and we have (G).

        If instead for all $\varepsilon \in \co_*$ $\cO_*\partial \varepsilon \subsetneq \cO_* \partial \cO_*=\cO_* \partial r$, then $\cO_*\partial \co_* \subsetneq \cO_* \partial r=\cO_*\partial \cO_*$, and thus it must be $\co_*\partial r=\co_*\partial \cO_*=\cO_* \partial \co_*$, but then $\co_*\partial \co_*=\co_* \cdot \cO_*\partial \co_*=\co_* \cdot \cO_*\partial r=\co_*\partial \cO_*$ and we have (R).
    \end{proof}
\end{lemma}

\begin{definition}
    We say that a derivation $\partial$ on a valued field  $(\bE_*, \cO_*)$ is of \emph{type (W)} (resp.\ \emph{type (G)}, \emph{type (R)}) if it satisfies the corresponding condition of Lemma~\ref{lem:inclusion_cases}.
\end{definition}

\begin{remark}
    We won't be concerned with type (G) because if $\val(\bE_*)^{>0}$ has no minimum element (so in particular if $(\bE_*, \cO_*)$ is a RCVF), then there can be no derivations $\partial$ of type (G) on $\bE_*$.

    In fact, suppose that $\cO_*\partial\cO_*$ is principal, say $\cO_*\partial r=\cO_* \partial (\cO_*)$, and that $\val(\bE_*)^{>0}$ has no minimum, then for all $\varepsilon \prec 1$, $\partial \varepsilon \prec \partial r$. In fact, we can find $\rho$ such that $\varepsilon \preceq \rho^2 \prec \rho \prec 1$. Now $\varepsilon = \rho^2 s$ for some $s \preceq 1$ and 
    \[\partial\varepsilon = \partial(s\rho^2)=\rho^2 \partial s + 2 \rho s \partial \rho \in \co_* \partial s + \co_* \partial \rho \prec \partial r\]
    therefore $\partial \varepsilon \prec \partial r$.
\end{remark}

\begin{example}\label{ex:type_G}
    A simple example of derivation of type (G) is given by the one of the Hardy field of rational functions over the reals $(\bE_*, \partial, \cO_*)=(\bR(x), \partial/\partial x, \CH(\bR))$: in fact then $\co_*=\cO_*/x$ and $\partial \cO_*=\partial \co_*$, so $\cO_*\partial\co_*=\cO_*\partial \cO_*=\cO_*/x^2$ whereas $\co_* \partial\co_*=\co_* \partial \cO_*=\cO_*/x^3$.
\end{example}

\begin{remark}
    Note that the type of $\partial$ is always the same as the type of $a\partial$ for any $a \in \bE_*$.
\end{remark}

\begin{remark}\label{rmk:type-R}
    Derivations of type (R) are those of \cite[Lem.~4.4.8]{aschenbrenner2019asymptotic}.
    In fact, if $\partial$ has type (R) w.r.t.\ $\cO_*$, with $\cO_*\partial\cO_*=\cO_* \partial r$, then $\partial_r\coloneqq\partial/\partial r$ is such that $\partial \cO_*\subseteq \cO_*$ and $\partial \co_*\subseteq \co_*$, so $\res_{\cO_*}(\partial_r):\res_{\cO_*}(\bE_*) \to \res_{\cO_*}(\bE_*)$ is a derivation and it is non-trivial because $\partial_r(r)=1$.
\end{remark}

We will make frequent use of the following basic fact.

\begin{lemma}\label{lem:module-derivatives}
    Let $\partial$ be a derivation on $\bE_*$, $\cO_*$ be a valuation subring and $\bE\coloneqq \Kr(\partial)$. Then for all $a \in \bE_*^{\neq0}$ the following hold:
    \begin{enumerate}
        \item $\cO_*\partial (a\cO_*)=a\cO_*\partial\cO_* + (\partial a) \cdot \cO_*$;
        \item for all $\cO_*$-modules $M$, $M \partial\cO_* \subseteq \cO_*\partial M$;
        \item if $\cO_*\partial (aM)\subsetneq \partial a \cdot M + a\cO_* \partial M$, then $\val(a) \notin \val(\bE)$, and there is $m_0 \in M^{\neq0}$ such that $\val(M^{\succeq m_0})\cap \val(\bE)=\emptyset$ and for all $m\in M^{\succeq m_0}$, $\lder(m) \sim -\lder(a)$, furthermore for such $m_0$, for all $m \prec m_0$, $\lder(a)\succ \partial m/m_0$;
        \item if $\cO_*\partial(a\co_*)\subsetneq \co_* \partial a + a \cO_* \partial \co_*$, then $\partial$ must have type (G) and $\co_*=\varepsilon_0 \cO_*$ for some $\varepsilon_0$ s.t.\ $\val(\varepsilon_0)\notin \val(\bE)$ and $\lder(\varepsilon_0)\sim -\lder(a)$.
    \end{enumerate}
    \begin{proof}
        (1) Note that by the Leibniz identity and since $1 \in \cO_*$ we have
        \[(\partial a)\cdot \cO_* \subseteq \cO_*\partial (a\cO_*)\subseteq a\cO_*\partial\cO_* + (\partial a) \cdot \cO_*.\]
        Suppose toward contradiction that there is $s \preceq 1$ such that $a\partial s \succ \partial (ar)$ for all $r \preceq 1$, that is $\partial s/r\succ \lder a + \lder r$. Note that then, setting $r=1$ we get $\partial s \succ \lder a$, so for all $r\preceq 1$, we should have $\partial s/r \succ \lder r$, that is $\partial s\succ \partial r$, but taking $s=r$ we get a contradiction.
        
        (2) Note that by (1), we have $m \partial \cO_* \subseteq \cO_* \partial (m\cO_*)$, taking unions over $m \in M$ on both sides yields the inclusion.
        
        (3) Clearly it must be $\val(a)\notin \val (\bE)$, for otherwise the equality would hold. Also $M$ cannot be generated by constants: in fact if $M=S\cO_*$ for some $S \subseteq \bE$, then 
        \[\cO_* \partial (aM)=\cO_* \partial (Sa\cO_*)=\cO_* S \partial (a\cO_*)=M \cdot (a\cO_*\partial \cO_* + (\partial a)\cO_*),\]
        which is easily seen to contain $(\partial a) \cdot M$ and thus also $\cO_* \partial M$ and we would again have equality. Thus there must be $m_1$ such that $\val(M^{\succeq m_1})\cap \val(\bE)=\emptyset$.
        For the rest note that if there is $m_0 \in M$, such that $m_0 \partial a \succ \partial (am)$ for all $m \in M$ or $a\partial m_0 \succ \partial (am)$ for all $m \in M$, then in both cases from taking $m=m_0$ we see that it must be $\lder(a) \sim -\lder(m_0)$. But then, again in both cases, we have $(m_0/m)\lder(a) \succ \lder(a)+ \lder(m)$ for all $m \in M^{\succeq m_0}$ and thus a fortiori $-\lder a \sim \lder m$.
        Finally note that if $m\prec m_0$ and $m_0 \partial a \succ \partial(am)$, then $\lder(a) \succ \lder(a) m/m_0 + \partial m/m_0$ which entails $\lder(a) \succ \partial m/m_0$ because $m/m_0\prec 1$.

        (4) Suppose that $\cO_* \partial(a \co_*)\subsetneq \partial a \cdot \co_*+ a\cO_* \partial \co_*$, then by (3) we can find $\varepsilon_0 \in \co_*$ such that for all $\varepsilon\succeq \varepsilon_0$, $\partial \varepsilon \neq 0$ and $\lder(\varepsilon_0)\sim \lder(\varepsilon)\sim \lder(1/a)$. This implies $\co_*=\varepsilon_0 \cO_*$: in fact, if there was some $\varepsilon_0\prec \varepsilon\prec 1$, then for at least one $\delta \in \{\varepsilon, \varepsilon_0/\varepsilon\}$, we would have $\varepsilon_0 \preceq \delta^2 \prec \delta \prec 1$, which would give $\lder(\varepsilon_0) \sim \lder(\delta^2) \sim 2\lder(\delta)$, a contradiction. Thus $\val(\varepsilon_0) = \min (\val(\bE_*)^{>0})$ and $\val(\varepsilon_0) \notin \val(\bE)$.

        Now observe that again by (3) we have for all $\varepsilon \prec \varepsilon_0$, that $\lder(\varepsilon_0)\sim \lder(a) \succ \partial \varepsilon/\varepsilon_0$, whence $\partial \varepsilon_0 \succ \partial \varepsilon$, and for all $\varepsilon \asymp \varepsilon_0$, $\partial \varepsilon_0 \asymp \partial \varepsilon$, thus $\partial (\co_*)=\cO_* \partial \varepsilon_0$, in particular $\cO_*\partial\co_*$ is principal (and thus also $\cO_* \partial \cO_*$ is).

        To conclude that $\partial$ has type (G) we need to show that $\cO_*\partial\cO_*=\cO_* \partial \varepsilon_0$ as well. For this note for all $r \asymp 1$ we have $\lder(\varepsilon_0) \sim \lder(\varepsilon_0r)=\lder(\varepsilon_0) +\lder(r)$, so $\lder(\varepsilon_0)\succ \lder(r) \sim \partial r$ and $\partial\varepsilon_0 \succ \cO_*\varepsilon_0 \partial r=\co_* r$ so $\partial\varepsilon_0 \succeq \partial r$.
    \end{proof}
\end{lemma}

\begin{remark}
    Concerning Lemma~\ref{lem:module-derivatives}(4), it is worth noting that the strict inclusion can occur when $\partial$ has type (G). In fact in Example~\ref{ex:type_G} we have
    \[\cO_*\partial(\co_* x)=\cO_* \partial \cO_*=\cO_* \partial \co_*=\cO_*/x^2,\]
    whereas $x\cO_*\partial \co_*=\cO_*/x$ and $\co_* \partial x=\co_*=\cO_* /x$.
\end{remark}

\begin{lemma}\label{lem:unbounded_integrals_case_v}
    If $\partial$ is a non-trivial derivation on $(\bE_*, \cO_*)$, then $\val(\partial \bE_*)$ is coinitial in $\val(\bE_*)$.
    \begin{proof}
        Let $\bE\coloneqq \Kr(\partial)$ and $\partial_t\coloneqq \partial/\partial t$ for some non-constant $t$. Note that $\val(\bE) \subseteq \val(\partial_t\bE_*)=\val(\partial \bE_*)-\val(\partial t)$ whence the statement holds as soon as $\val(\bE)$ is unbounded. Suppose $\val(\bE)$ is bounded and choose $t \succ \bE$. We are going to show that for all $x \succ \bE$ there is $y$ such that $\partial_t y \succ x$. 
        Note that
        $\partial _t(x^n t) = x^n(1+t\lder_t(x^n))$, so since for at least one $n \in\{1,2\}$ we must have $-1 \not\sim t \lder_t(x^n)$ we have $\partial_t(x^n t) \succeq x^n\succeq x$.
    \end{proof}
\end{lemma}

\subsection{(Pre-)order-convexity and order convex derivations}\label{ssec:order-convexity}

This subsection is dedicated to the abstraction of some properties common to ordered differential field of the form $(\bE_*, <, \partial)$ with $\bE_*= \bE \langle x \rangle\succ \bE \models T$ for some complete o-minimal theory $T$ and $\partial=\partial_x$ the only $\bE$-linear $T$-derivation with $\partial_x x = 1$. The goal really is to abstract analogous properties for the corresponding valued differential fields $(\bE_*, \cO_*, \partial)$ where $(\bE_*, \partial)$ is as before and $\cO_*$ is some convex subring of $(\bE_*,<)$. Such valuation-theoretic abstraction will actually be the subject of the next subsection (\ref{ssec:val-convexity}), but it is the reason some preliminary terminology is developed for preorders (so that it can be used for the valuation-dominance/divisibility preorder $\preceq_{\cO_*}$).

\begin{definition}
    For a preorder $(Y, \preceq)$ and a subset $T\subseteq Y$ we will write $\CH_{Y}^{\preceq}(T)\coloneqq \{y \in Y: \exists t_0,\; t_1 \in T, t_0 \preceq t\preceq t_1\}$ for the $\preceq$-convex hull of $T$ in $Y$ and we will call $T$ $\preceq$-convex if $\CH_{Y}^{\preceq}(T)=T$.
    
    Let $(Y, \preceq)$ be a preorder, and let $f: X \to Y$ be a function.\footnote{the convention here is that functions ``know'' their codomain, i.e.\ a function is identified by a functional relation together with a designated superset of its range.}
    We say that a subset $S\subseteq X$ is \emph{$f$-convex} (or that $f$ is \emph{$S$-convex}) if $f(S)$ is $\preceq$-convex in $f(X)$, i.e.\ if for all $x \in X$, $f(x)\notin f(S)\Rightarrow f(x) \notin \CH_{Y}^{\preceq}(f(S))$.
\end{definition}

\begin{lemma}
    Let $(Y, \preceq)$ be a preorder, $f: X \to Y$ be a function, and $S\subseteq X$. Then $S$ is $f$-convex if and only if there is a $\preceq$-convex subset $T\subseteq Y$ such that $\forall x \in X,\; f(x) \in f(S) \leftrightarrow f(x) \in T$ or equivalently $f(S)= T \cap f(X)$.
    \begin{proof}
        If $S$ is $f$convex, then $T=\CH_{Y}^{\preceq}(f(S))$ is a set as in the statement. Conversely, given such a $T$, we deduce $f(x) \notin f(S) \Leftrightarrow f(x) \notin T \cap f(X)\Rightarrow f(x) \notin T\supseteq \CH_{Y}^{\preceq}(f(S))$. 
    \end{proof}
\end{lemma}

\begin{lemma}\label{lem:f-convex_directed_unions}
    Let $(Y, \preceq)$ be a preorder, $f: X \to Y$ be a function. Any directed union of $f$-convex sets is $f$-convex.
    \begin{proof}
        Let $\cS$ be a directed family of $f$-convex subsets of $Y$. Set $\cT\coloneqq \{\CH_Y^{\preceq}(f(S)): S \in \cS\}$ and note that $\cT$ is directed because images and convex hulls both preserve inclusions. Also note that a directed union of $\preceq$-convex sets is $\preceq$-convex, thus $T \coloneqq \bigcup \cT$ is $\preceq$-convex. Now observe that $T \cap f(X)=f(\bigcup \cS)$: in fact $f(S) \subseteq f(X) \cap T$ for all $S \in \cS$ and furthermore if $f(x) \in T$, then it must be $f(x) \in \CH_Y^{\preceq}(f(S))$ for some $S \in \cS$, whence $f(x)\in f(S)$ and $f(x) \in f(\bigcup \cS)$.
    \end{proof}
\end{lemma}

\begin{proposition}\label{prop:der_order-convexity}
    Let $T$ be a complete o-minimal theory expanding $\RCF$, $\bE \prec \bE\langle x\rangle\models T$ a 1-$\dcl_T$-dimensional elementary extension of models of $T$, and $\partial_x$ be the only $\bE$-linear $T$-derivation with $\partial_x x=1$, then for all $y \in \bE\langle x\rangle^{\neq 0}$, $[-y, y]$ is $\partial_x$-convex and $[1/y, y]$ is $\dagger_x$-convex (here $\lder_x(t)\coloneqq \partial_xt/t$).
    \begin{proof}
        We have to show that if $f, h$ are continuously differentiable $\bE$-definable unary functions, with $|f'(x)|\le |h'(x)|$ and $|h(x)|\le y$, then $f(x) \in \bE + [-y,y]$.
        Note that we can in fact assume that $0<|h(x)|\le y$ and that $h'(x) \neq 0$, otherwise $f'(x)=0$ and the thesis is trivial. Similarly we can assume $|f'(x)|<|h'(x)|$, because otherwise $f'=h'$ in an $\bE$-definable neighborhood of $x$ and $f(x)-h(x)\in \bE$ and again the thesis is trivial.
        
        Then in an $\bE$-definable neighborhood $(x_-, x_+)$ of $x$, $|f'|\le |h'|$ and $|h|$ is strictly monotone and non-zero (note that $x_-, x_+ \in \bE \cup \{\pm\infty\}$).
        Note that either on $(x_-,x)$ or on $(x, x_+)$ we must have $|h|\le y$. So up to replacing $x$ with $-x$ and precomposing both $f$ and $h$ with a sign-change, we can without loss of generality assume that $|h|\le y$ on $(x_-, x)$.
        If $(x_-, x) \cap \bE \neq \emptyset$, then up to replacing $x_-$ by some element in $(x_-, x) \cap \bE \neq \emptyset$ we can assume that furthermore $f$ and $h$ are both continuous on $[x_-,x]$. Similarly, if $(x_-, x) \cap \bE=\emptyset$ and $\lim_{t\to x_-} f(t)$ and $\lim_{t \to x_-} h(t)$ exist in $\bE$, then up to changing the value of $f$ and $h$ only at $x_-$, we can still assume that both $f$ and $h$ are continuous on $[x_-,x]$. 
        But then by the (Cauchy) mean value theorem, we have
        \[\frac{|f(x)-f(x_{-})|}{|h(x)-h(x_{-})|}=\frac{|f'(\xi)|}{|h'(\xi)|}\le 1.\]
        So since $h$ doesn't change sign on $(x_-,x)$, we deduce
        \[|f(x)-f(x_-)|\le |h(x)-h(x_-)|\le \max\{|h(x)|, |h(x_-)|\}\le y\]
        and thus $f(x) \in \bE+[-y,y]$.

        If instead $(x_-,x)\cap \bE=\emptyset$, and $|\lim_{t\to x_-}h(t)|=\infty$ and $\lim_{t \to x_-} f(t) \in \bE$, then $f(x) \in \CH^{\le}_{\bE_*}(\bE)$ whereas $y \ge |h(x)|>\bE$ so again $f(x) \in \bE+[-y,y]$.

        Finally if $(x_-,x)\cap \bE=\emptyset$, and $|\lim_{t\to x_-^+}f(t)|=\infty$, then since $|f'|< |h'|$ it must be $|\lim_{t\to x_-^+}h(t)|=\infty$ as well. Up to replacing $h$ with $-h$ we can assume that $h'>0$ on $(x_-, x_+)$, so we must have $\lim_{t\to x_-}h(t)=-\infty$ and $h(x)<0$.
        Also up to decreasing $x_+ \in \bE^{>x}=\bE^{>x_-}$ we can assume that $f$ and $h$ are also continuous on $(x_-, x_+]$ and that $h(x_+)<0$.
        
        Since $-h'<f'<h'$, we then get that on $(x_-, x_+]$, $f-h$ is decreasing and $f+h$ is increasing, thus
        \[f(x_+)-h(x_+)+h(x)\le f(x) \le f(x_+)+h(x_+)-h(x),\]
        that is $f(x) \in f(x_+)+[-h(x_+)+h(x), h(x_+)-h(x)]$, and we can conclude because since $h(x)$ and $h(x_+)$ have the same sign $|h(x_+)-h(x)|\le y$.

        If the theory $T$ is exponential, then the $\lder_x$-convexity of $[1/y, y]$ can be easily deduced from the $\partial_x$-convexity of $[-\log|y|,\log|y|]$. The general case can be dealt with a similar but not entirely identical argument.

        Arguing as in the first paragraph we see that it is enough to prove that for $f, h$ positive, continuously differentiable, $\bE$-definable unary functions such that $|f^\dagger(x)| < |h^\dagger(x)|$ and $h(x) \in [1/y, y]$, we have $f(x) \in \bE^{>0} \cdot [1/y, y]$. Note that we can assume $y>1$ otherwise $[1/y, y] \in \{\{1\}, \emptyset\}$ and the statement is trivial.
        
        Again we can choose an $\bE$-definable neighborhood $(x_-, x_+)$ on which $f,h$ are differentiable with $|f^\dagger| \le |h^\dagger|$ and $h$ is strictly monotone. So either on $(x_-, x)$ or on $(x, x_+)$ we must have $\max\{h, 1/h\} \le y$. Again, up to replacing $x$ with $-x$ and precomposing both $f$ and $h$ with a sign-change we can without loss of generality assume that $\max\{h, 1/h\} \le y$ on $(x_-, x)$.
        
        Also up to replacing $h$ with $1/h$ we can assume that $h$ is increasing on $(x_-, x)$, so $h'>0$ there.
        We can also assume, up to increasing $x_- \in \bE^{<x}$, that $h(x)$, $h(x_-)$ are either both in $[1/y, 1]$ (when $h(x)<1$) or in $[1, y]$ (when $h(x)>1$). It then follows in particular that $1\le h(x)/h(x_-)\le y$.
        
        Again, unless we have $(x_-, x)\cap \bE=\emptyset$ and some of the limits $\lim_{t\to x_-} f(t)$ and $\lim_{t \to x_-}h(t)$ do not exist in $\bE^{>0}$, up to possibly changing the value of $f$ and $h$ at $x_-$ we can assume that both $f$ and $h$ are continuous on $[x_-, x]$.
        
        Now from the fact that $-h^\dagger<f^\dagger < h^\dagger$ on $(x_-, x)$ we deduce that $(fh)'\ge 0$ and $(f/h)'\le 0$\footnote{this specific implication was suggested by Gemini 3.1 (in a ultimately incorrect proof), see Section~\ref{sec:ai_interactions} for an account of the interaction.} so, on $[x_-,x]$, $fh$ is increasing and $f/h$ is decreasing. It thus follows that 
        \[ f(x_-)\frac{h(x_-)}{h(x)}\le f(x)\le f(x_-)\frac{h(x)}{h(x_-)}\]
        and we can conclude using that $1\le h(x)/h(x_-)\le y$.

        Finally we deal with the cases in which $(x_-, x) \cap \bE= \emptyset$ and one of the two conditions $\lim_{t \to x_-} h(t) \in \{0, \infty\}$, $\lim_{t \to x_-} f(t) \in \{0, \infty\}$ holds.
        
        If $\lim_{t \to x_-^+} h(t)\in \{0, \infty\}$ and $\lim_{t \to x_-^+} f(t)\in \bE^{>0}$, then $f(x) \in \CH_{\bE_*}^{\le}(\bE^{>0})$ whereas $y \ge \max\{h(x), 1/h(x)\}> \bE$ so we are done.
        
        If instead $\lim_{t \to x_-^+} f(t)\in \{0, \infty\}$, then since $|f^\dagger|<h^\dagger$ and $h$ is positive and increasing it must be $\lim_{t \to x_-^+} h(t) = 0$ so up to decreasing $x_+ \in \bE^{>x}=\bE^{>x_-}$ we can assume that $0<h(x)<h(x_+)<1$. Also, arguing as above we deduce that $fh$ and $f/h$ must be respectively increasing and decreasing so again $f(x) \in f(x_+) \cdot [h(x)/h(x_+), h(x_+)/h(x)]\subseteq [1/y, y]$ and we can conclude.
    \end{proof}
\end{proposition}

\begin{definition}\label{def:order-convex_der}
    Let $\bE_*$ be an ordered field. We will say that $\partial: \bE_* \to \bE_*$ is \emph{order-convex}, if for all $y \in \bE_*^{\ge0}$, $[-y,y]$ is $\partial$-convex. We will say that $\partial$ is \emph{logarithmically order-convex} if for all $y \in \bE_*^{\ge1}$, $[1/y,y]$ is $\lder$-convex. Finally we call \emph{order-Liouville-convex} a derivation $\partial$ that is both order-convex and logarithmically order-convex.
\end{definition}

\begin{remark}
    Note that if $[-y,y]$ is $\partial$-convex, then $a+[-y,y]$ is $\partial$-convex as well. Similarly if $[1/y, y]$ is $\lder$-convex and $a \in \bE_*^{>0}$, then $[a/y, ay]$ is $\lder$-convex as well.
\end{remark}

\begin{remark}
    If $(\bE_*, \partial)$ is Liouville closed, then $\partial$ is order-convex (resp.\ logarithmically ---) if and only if $\partial: \bE_* \to \bE_*$ (resp.\ $\lder: \bE_*^{\neq 0}\to \bE_*$) maps order-convex sets onto order-convex sets.
\end{remark}

\begin{remark}[Trace of order-convexity for the standard convex valuation]\label{rmk:order-convex_to_val-convex}
    If $\bE_*$ is an ordered field, $\partial:\bE_* \to \bE_*$ is an order-convex derivation, and $\preceq_{\cO_*}$ is the dominance relation of the valuation ring $\cO_*=\CH^{\le}(\bZ)$, then for every convex subgroup $M$ of $\bE_*$, $\partial M$ is $\preceq$-convex in $\partial\bE_*$ (or in other words $M$ is $\partial$-convex for $\preceq$): this is just because $\partial$ is a group homomorphism and for subgroups, being $\preceq_{\cO_*}$-convex is the same as being order-convex.
    We will see later in Corollary~\ref{cor:order-convex_implies_val-convex} that this generalizes to the case when $\cO_*$ is any order-convex subring of $\bE_*$.
\end{remark}

\begin{remark}[Traces of logarithmic order-convexity for the standard convex valuation]\label{rmk:order-convex_to_val-convex2}
    Similarly, again with $\cO_*=\CH^{\le}(\bZ)$, if $\partial$ is logarithmically order-convex, we see that all order-convex subgroups of $\bE_*^{>0}$ are $\lder$-convex w.r.t.\ $\preceq_{\cO_*}$.
    Note that an order-convex subgroup of $(\bE_*^{>0}, \cdot)$ is either of the form $(\cO_\circ \setminus \co_\circ)^{>0}$ for some order-convex subring $\cO_\circ$ or of the form $1+M$ for $M$ some convex additive subgroup of $\co_*$. We'll see later (as a consequence of Proposition~\ref{prop:v-convexity_reduction} in the next subsection) that the $\lder$-convexity w.r.t.\ $\preceq_{\cO_*}$ of the sets of the form $1+M$ for $M$ convex subgroup of $\co_*$ also follows from the order-convexity of $\partial$.
\end{remark}

The remainder of this subsection is dedicated to drawing another consequence of logarithmic order-convexity which can be ``seen'' by the standard convex valuation (Lemma~\ref{lem:der_local_monotone} and Remark~\ref{rmk:der_local_monotone}).

\begin{remark}
    If $\bE_*$ is an ordered field and $\partial: \bE_* \to \bE_*$ is a non-trivial derivation with constants $\bE\coloneqq \Kr(\partial)$, then for all constants $a,b \in \bE$ with $a<b$, $\partial([a,b])$ is a subgroup of $\partial \bE_*$, in fact if $a\le x \le y \le b$, then $a\le a+y-x\le b$ and $a\le b+x-y\le b$, thus $\{\partial x -\partial y,\partial y - \partial x\} \subseteq \partial([a,b])$.
    Thus in particular for all $x,y \in \bE_*$, if $\partial([x,y])\ge 0$, then $\bE \cap [x,y]$ contains at most one element.
    
    It follows that if $x,y \in \bE_*$, $\partial([x,y])\ge 0$ and $\bE \cap [x,y]=\{x\}$ then $\partial$ is increasing on $[x,y]$, in fact if $x\le z\le y$, then $\partial z \le \partial y$, because $x \le x+y-z\le y$ and thus $\partial (x + y -z)=\partial y-\partial z\ge 0$. 
    A similar fact holds if we just assume that $\partial x \prec_{\cO_*} \partial y$ where $\cO_*=\CH^{\le}_{\cO_*}(\bZ)$ as shown in the next Lemma.
\end{remark}

\begin{lemma}\label{lem:der_local_monotone}
    Let $\bE_*$ be an ordered field, $\cO_*\coloneqq\CH^{\le}_{\bE_*}(\bZ)$, and $\partial: \bE_* \to \bE_*$ be a derivation.
    \begin{enumerate}
        \item Suppose that $\partial([x,y]+\cO_*)>\bE_*^{<\partial\cO_*}$ and $\partial \cO_*+\partial x \prec \partial y$; then for all $z \in [x,y]+\cO_*$, $\partial z \preceq \partial y$.
        \item Suppose that $y\succ x$ and $\lder\big(\val^{-1}[\val(y), \val(x)]\big)> \bE_*^{<\lder(\val^{-1}(0))}$ and $\lder(x)+ \lder(\cO_*\setminus\co_*) \prec \lder(y)$; then for all $z \in \val^{-1}[\val(y), \val(x)]$, $\lder(z) \preceq \lder(y)$.
    \end{enumerate}
    \begin{proof}
        To lighten the notation set $M=\cO_* \partial \cO_*$ and note that $\lder(\cO_* \setminus \co_*)$ and $\partial \cO_*$ are subgroups of $\bE_*$ and that \[\cO_*\lder(\cO_* \setminus \co_*)=\CH(\lder(\cO_* \setminus \co_*))=\CH(\partial \cO_*)=\cO_* \partial \cO_*=M.\]
        (1) Let $z\in [x,y]+\cO_*$, then $\partial(x+y-z)>\bE_*^{<\partial\cO_*}$ so
        \[\partial (y+ x) - \bE_*^{<M} = \partial(y+x) + \bE_*^{>M} > \partial z >\bE_*^{<M},\]
        i.e.\ $\partial z \in [0,\partial(y+x)]+M$.
        Since by hypothesis $\partial y \succ M+\partial x$, we have $\cO_*\partial y \supseteq [0,\partial(y+x)]+M$ so $\partial z \in M+ [0, \partial (y+x)]\subseteq \cO_* \partial y$.

        (2) Suppose without loss of generality that $0<x<y$. By hypothesis we have $\lder([x,y]\cdot(\cO_*\setminus \co_*))>\bE_*^{<M}$ and $\lder(x)+M \prec \lder(y)$. Let $z\in [x,y]\cdot(\cO_*\setminus \co_*)$, then $\lder(xy/z)>\bE_*^{<M}$ so again
        \[\lder (yx) - \bE_*^{<M} = \lder(y x) + \bE_*^{>M} > \lder (z) >\bE_*^{<M},\]
        i.e.\ $\lder(z) \in [0,\lder(xy)] + M$. Arguing as above, since $\lder(y)\succ M+\lder(x)$, we get $\lder(z) \in \cO_* \lder(y)$.
    \end{proof}
\end{lemma}

\begin{remark}\label{rmk:der_local_monotone}
    Of course the interest of the above Lemma and the preceding Remark lies in the fact that if $\partial$ is order-convex (resp.\ logarithmically order-convex), then when $[x,y]\cap \bE$ (resp. $\big(([x,y]+\cO_*) \cap (\bE+\cO_*)\big)/\cO_*$, $[\val(y), \val(x)]\cap \val(\bE)$) has at most one element, then for some $\sigma \in \{\pm1\}$ it must be $\sigma \partial[x,y]\ge 0$ (resp.\ $\sigma \partial([x,y]+\cO_*)>\bE_*^{<M}$, $\sigma \lder \big(\val^{-1}[\val(y), \val(x)]\big)>\bE_*^{<M}$).
\end{remark}

\subsection{Valuation-convex derivations}\label{ssec:val-convexity}
In this subsection we give the ``valuational traces'' of order-convexity and logarithmic order-convexity of a derivation once the order is ``coarsened'' to some convex valuation, which is what was promised at the beginning of Subsection~\ref{ssec:order-convexity}. 
We start by defining the promised conditions, which can at first be intuitively understood as valuation-theoretic analogues of the conditions of order-convexity and logarithmic order-convexity.
Note though that we have already observed (Remarks~\ref{rmk:order-convex_to_val-convex} and \ref{rmk:order-convex_to_val-convex2}) that they hold for a valued field of the form $(\bE_*, \cO_*, \partial)$ where $\partial$ is Liouville order-convex on the ordered field $(\bE_*, <)$ and $\cO_*=\CH^{\le}(\bZ)$ is the minimal order-convex subring (a more general result will be given at the end of the section Corollary~\ref{cor:order-convex_implies_val-convex}).

\begin{definition}\label{def:v-convex_derivation}
    Let $(\bE_*, \cO_*)$ be a valued field and $\partial: \bE_* \to \bE_*$ be a derivation.
    I will say that a derivation $\partial$ on $\bE_*$ is \begin{enumerate}
        \item \emph{$\cO_*$-valuation-convex} if all $\cO_*$-submodules are $\partial$-convex (w.r.t.\ $\preceq_{\cO_*}$);
        \item \emph{quasi-logarithmically $\cO_*$-valuation-convex} if it is $\cO_*$-valuation-convex and $\cO_*\setminus \co_*$ is $\lder$-convex;
        \item \emph{logarithmically $\cO_*$-valuation-convex} if it is $\cO_*$-valuation-convex and for all $x,y,z$
        \begin{equation}\tag{H1}\label{axiom:new-type-H}
            (x \preceq z \preceq y \;\& \; \lder(x) \prec \lder(y)\prec \lder(z)) \Rightarrow \exists c, \; (\partial c=0 \; \& \; x \prec c \preceq y).
        \end{equation}
    \end{enumerate}
    We will often drop the ``$\cO_*$-'' in front of the definition if it is clear from the context or specified otherwise (e.g.\ by saying ``w.r.t.\ $\cO_*$''). We also abbreviate (3) with \emph{log.\ ($\cO_*$-)valuation-convex} and (2) with \emph{q.log.\ ($\cO_*$-)valuation-convex}.
\end{definition}

\begin{remark}\label{rmk:convexity_is_first-order}
    Note that since $\partial$-convex (resp.\ $\lder$-convex) sets are closed under directed unions (Lemma~\ref{lem:f-convex_directed_unions}) and all $\cO_*$-submodules are increasing unions of principal submodules, we have that all the properties listed in Definition~\ref{def:v-convex_derivation} are in fact first-order properties. 
\end{remark}

\begin{remark}\label{rmk:convexity_invariance}
    Note that if $a \in \bE_*^{\neq0}$, then a subgroup $S\le (\bE_*,+)$ is $\partial$-convex if and only if it is $(a \partial)$-convex.
    Similarly a subgroup $S\le (\bE_*^{\neq 0}, \cdot)$ is $\lder$-convex if and only if it is $(a\lder)$-convex. In particular all the properties listed in Definition~\ref{def:v-convex_derivation} are invariant under rescaling the derivation.
\end{remark}

\begin{remark}
    Note that if $S$ is a subgroup of $(\bE_*, +)$ then it is $\preceq$-convex if and only if it is a $\cO_*$-submodule of $\bE_*$. Since $\partial: (\bE_*, +) \to (\bE_*,+)$ and $\lder: (\bE_*^{\neq0}, \cdot)\to (\bE_*,+)$ are group homomorphisms, we have the following simplified characterizations of $\partial$-convexity and $\lder$-convexity for subgroups:
    \begin{enumerate}[label=(\Alph*)]
        \item $S\le (\bE_*, +)$ is $\partial$-convex if and only if $\forall x \in \bE_*,\; \partial x \notin \partial S \rightarrow \partial x \succ \partial S$;
        \item $S\le (\bE_*^{\neq0}, \cdot)$ is $\lder$-convex if and only if $\forall x \in \bE_*^{\neq0},\; \lder x \notin \lder S \rightarrow \lder x \succ \lder S$.
    \end{enumerate}
\end{remark}

\begin{remark}\label{rmk:Liouville-convex_naming}
     If $(\bE_*, \partial)$ is Liouville-closed, then being q.log.\ $\cO_*$-valuation-convex is the same as saying that for all $\cO_*$-modules $M\subseteq\cO_*$, $\partial M$ and $\lder((1+M) \setminus \co_*M)$ are $\cO_*$-modules.
\end{remark}

\begin{remark}[About logarithmic valuation-convexity]\label{rmk:about_new-H_axiom}
    Note that axiom (\ref{axiom:new-type-H}) can be reformulated by saying that if there are no constants $c$ with $x \prec c \preceq y$ and $\lder(x) \prec \lder(y)$, then for all $x \preceq z_1 \preceq z_2\preceq y$, $\lder(z_1)\preceq \lder(z_2)$.
    
    Also, assuming $\lder$-convexity of $\cO_* \setminus \co_*$, we have that $\val(z_1)=\val(z_2) \notin \val(\Kr(\partial))$ implies that $\lder(z_1)-\lder(z_2) \in \lder(\cO_*\setminus \co_*)\prec \lder(z_1)$, so (\ref{axiom:new-type-H}) can in fact be replaced, in the definition of logarithmic $\cO_*$-valuation-convexity with the condition that $\cO_* \setminus \co_*$ is $\lder$-convex together with the weaker axiom
    \begin{equation}\tag{H0}\label{axiom:new-type-H-relative}
            (x \prec z \prec y \;\& \; \lder(x) \prec \lder(y)\prec \lder(z)) \Rightarrow \exists c, \; (\partial c=0 \; \& \; x \prec c \preceq y).
    \end{equation}
    Note that unlike (\ref{axiom:new-type-H}), (\ref{axiom:new-type-H-relative}) stays valid when enlarging the valuation ring.
\end{remark}

We now set to give in Proposition~\ref{prop:v-convexity_reduction} below some equivalent formulations of valuation-convexity. Notice that formulation (3) therein together with Remark~\ref{rmk:order-convex_to_val-convex} delivers on the promise we made in Remark~\ref{rmk:order-convex_to_val-convex2} and, in a sense, justifies the naming choices made in Definition~\ref{def:v-convex_derivation}(2 and 3).

\begin{lemma}\label{lem:lder_vs_partial-convexity}
    Let $(\bE_*, \cO_*, \partial)$ be a valued differential field with constants $\bE\coloneqq \Kr(\partial)$, $M$ be an $\cO_*$-submodule of $\bE_*$, and $y \in \bE_*^{\neq 0}$. Then the following hold:
    \begin{enumerate}
        \item if $M \subseteq \co_*$, $y \notin \bE^{\neq 0} (1+M)$ and $yM$ is $\partial$-convex w.r.t.\ $\preceq_{\cO_*}$, then $\lder(y) \succ \partial M$;
        \item if $y \notin \bE + M$ and $1+M/y$ is $\lder$-convex w.r.t.\ $\preceq_{\cO_*}$, then $\partial y \succ \partial M$.
    \end{enumerate}
    \begin{proof}
        (1)
        Note that we have 
        \[y \in \bE^{\neq 0} (1+M) \Leftrightarrow \exists c \in \bE^{\neq 0}, \; y \in c+ Mc \Leftrightarrow \exists c \in\bE^{\neq0}, \; y-c \in My,\]
        where the last $\Leftrightarrow$ is because $M\prec 1$, so $y-c \in Mc$ implies $y\sim c$ and so $Mc=My$. Thus we have $y \notin \bE^{\neq 0} (1+M)$ if and only if $ y \notin \bE + yM$, from which in turn, by $\partial$-convexity of $yM$, we deduce $\partial y \succ \partial (yM)$. If $\cO_*\partial (yM) = M \partial y + y \cO_*\partial M$, then we are done. If instead $\cO_*\partial (yM) \subsetneq M \partial y + y \cO_*\partial M$, then by Lemma~\ref{lem:module-derivatives}(3), we have $m_0 \in M$ such that $\lder(y) \asymp \lder(m)\succ \partial m$ for all $m \in M^{\succeq m_0}$ and $\lder(y)\succ \partial m/m_0 \succ \partial m$ for all $m \prec m_0$.

        (2)
        Note that since $y \notin \bE+M$, in particular $y \succ M$, so $M/y\prec 1$ and arguing as in $(1)$, we get $y \notin \bE+M \Rightarrow y \notin \bE^{\neq0}(1+M/y)$.
        
        So $y\notin \bE+M$, by $\lder$-convexity of $(1+M/y)$, implies that $\lder(y) \succ \partial (M/y)$, and again we can conclude immediately if $\partial(M/y)=\cO_*\partial M / y + \partial y/y^2 M$. Otherwise we have, by Lemma~\ref{lem:module-derivatives}(3), that for some $m_0 \in M$, for all $m\in M^{\succeq m_0}$, $\lder(y)\sim \lder(m)$ so $\partial y \sim y/m \partial m\succ \partial m$, and for all $m \prec m_0$, $\lder(y) \succ \partial m/m_0$, so $\partial y \succ y/m_0 \partial m \succ \partial m$.
    \end{proof}
\end{lemma}

\begin{proposition}\label{prop:v-convexity_reduction}
    Let $(\bE_*, \cO_*)$ be a valued field, $\partial: \bE_* \to \bE_*$ be a derivation, and $\bE=\Kr(\partial)$. Then the following are equivalent:
    \begin{enumerate}
        \item $\partial$ is $\cO_*$-valuation-convex;
        \item all $\cO_*$-submodules $M\subseteq \co_*$ are $\partial$-convex w.r.t.\ $\preceq_{\cO_*}$;
        \item for all $\cO_*$-submodules $M \subseteq \co_*$, $(1+M)$ is $\lder$-convex w.r.t.\ $\preceq_{\cO_*}$.
    \end{enumerate}
    \begin{proof}
        Clearly $(1) \Rightarrow (2)$. Also $(1)\Rightarrow (3)$ is by Lemma~\ref{lem:lder_vs_partial-convexity}(1), and $(3)\Rightarrow (2)$ is by Lemma~\ref{lem:lder_vs_partial-convexity}(2).
        So we are left to prove $(2)\Rightarrow (1)$.
        By Lemma~\ref{lem:f-convex_directed_unions} (see also Remark~\ref{rmk:convexity_is_first-order}), we can reduce to proving that for all $x \succeq 1$ we have $y \notin x\cO_* + \bE \Rightarrow \partial y \succ \partial (x\cO_*)$.
        
        If there is a constant $c\in \bE$ such that $x \prec c$, then we can immediately conclude as $\partial (x\cO_*)= c\partial ((x/c) \cO_*)$ and thus we have
        \[\partial y \notin \partial (x \cO_*) \Leftrightarrow \partial (y/c) \notin \partial ((x/c)\cO_*) \Rightarrow \partial (y/c) \succ \partial ((x/c)\cO_*) \Leftrightarrow \partial y \succ \partial(x\cO_*).\]
        If instead $x \succeq \bE$, we can use a standard manipulation from the theory of asymptotic fields. First observe that if $x \succeq \bE$ we have for all $y \in \bE_*^\neq0$
        \[y \notin x\cO_* + \bE \Longleftrightarrow y \notin x \cO_* \Longrightarrow 1/y \notin (x/y^2)\cO_* + \bE,\]
        where the first equivalence is because $x\cO_* + \bE=x\cO_*$ and the second is because, since $y\succ x\succeq \bE$, $1/y \prec \bE^{\neq 0}$ and $x/y^2 \prec \bE^{\neq 0}$.
        Now from (2) and Lemma~\ref{lem:module-derivatives}, we deduce that
        \[\partial\left(\frac{1}{y}\right)=-\frac{\partial y}{y^2} \notin \partial\left(\frac{x}{y^2}\cO_*\right)= \frac{x}{y^2} \cO_*\partial \cO_* + \left(\frac{\partial x}{y^2}-\frac{2x}{y^3} \partial y\right) \cO_*\]
        so $\partial y \succ x \partial \cO_*$ and $\partial y \succ \partial x$, thus $\partial y \succ \partial (x\cO_*)$ which is what was to be proven.
    \end{proof}
\end{proposition}

\begin{remark}
    Let $(\bE_*, \cO_*)$ be a valued field, $\partial$ be a derivation on $\bE_*$, $\bE\coloneqq \Kr(\partial)$.
    If $N$ is $\partial$-convex, then for every $c \in \bE$, $cN$ is $\partial$-convex. So if $\val(\bE)$ is order-dense in $\val(\bE_*)$ (in particular when $(\bE_*, \cO_*, \partial)$ has \emph{many constants}, cf.\ \cite[p.192]{aschenbrenner2019asymptotic}, \cite{scanlon2000model}), then $\partial$ is valuation-convex if and only if $\cO_*$ is $\partial$-convex.

    As we will see in Corollary~\ref{cor:convexity_in_pv} below, provided there are no elements in $\bE_*$ that are pseudolimits for some sequence of constants (or, more strongly, that there are no elements that are absorbed over the constants), valuation-convexity of $\partial$ can be deduced from the $\partial$-convexity of just $\cO_*$ and some instances of $\lder$-convexity. Although this situation is not of particular interest in this paper, it will help draw a parallel with the more widely studied asymptotic fields from \cite[Sec.~9.1]{aschenbrenner2019asymptotic}, where the constants have trivial valuation (cf.\ Subsection~\ref{ssec:few-constants} below).
\end{remark}

\begin{remark}\label{rmk:note_before_weak_lder-conv}
    Note that by Lemma~\ref{lem:module-derivatives}(1-4), if $\lder(x)\succeq \partial \cO_*$ then $\cO_* \partial (\cO_* x)=\cO_* \partial x$ and if $\lder(x)\succ \partial \co_*$, then $\cO_*\partial(x\co_*)\subseteq \co_*\partial x$ (with equality holding when $\partial$ has not type (G)).
    The condition $\lder(x)\succ \partial \co_*$ is ensured by $\rv(x) \notin \rv(\bE)$ when $1+\co_*$ is $\lder$-convex, whereas when $\cO_* \setminus \co_*$ is $\lder$-convex $\val(x) \notin \val(\bE)$ ensures the stronger $\lder(x) \succ \partial \cO_*$.

    As will be apparent below, often one still gets some meaningful consequences from a weaker condition which we now name.
\end{remark}

\begin{definition}
    If $(\bE_*, \cO_*, \partial)$ is a valued differential field, we say that \emph{$\cO_* \setminus \co_*$ is weakly $\lder$-convex} (w.r.t.\ $\preceq$) when for all $x \in \bE_*^{\neq0}$, $\val(x) \notin \val(\bE)\Rightarrow \lder(x) \succ \partial \co_*$.
\end{definition}

\begin{remark}
    Note
    \[\partial \text{ is valuation-convex} \Rightarrow 1+\co_* \text{ is $\lder$-convex} \Rightarrow \cO_* \setminus \co_* \text{ is weakly $\lder$-convex}.\]
\end{remark}

\begin{lemma}\label{lem:weak_lder-convex_vs_lder-convex}
    If $\cO_*\setminus \co_*$ is weakly $\lder$-convex and $\co_*$ is $\partial$-convex then $1+\co_*$ is $\lder$-convex. The converse holds provided that $\cO_*$ is $\partial$-convex.
    \begin{proof}
        Suppose that $\cO_*\setminus \co_*$ is weakly $\lder$-convex and $\co_*$ is $\partial$-convex. Let $\rv(x) \notin \rv(\bE)$. If $\val(x) \notin \val(\bE)$, then $\lder(x)\succ \partial\co_*$ by the weak $\lder$-convexity of $\cO_*\setminus \co_*$, otherwise there is $c \in \bE^{\neq0}$ such that $x/c \in \cO_* \setminus (\bE+\co_*)$, but then by $\partial$-convexity of $\co_*$ we have $\partial(x/c) \succ \partial \co_*$ but $\lder(x) \asymp \partial(x/c)$ because $c\asymp x$ so we are done.
    
        Now suppose that $1+\co_*$ is $\lder$-convex and $\cO_*$ is $\partial$-convex. Assume $x \notin \bE+\co_*$, i.e. $x-\bE \succeq 1$. If $\rv(x-\bE)\not \subseteq \rv(\bE)$, then there is $c \in \bE$ such that $\lder(x-c)\succ \partial \co_*$, but since $x-c\succ \co_*$ it must be $\partial x \succ \cO_*(x-c)\partial \co_* \supseteq \co_*\partial \co_*$. If instead $\rv(x-\bE)\subseteq \rv(\bE)$, then $\val(x-\bE)$ has no maximum and $x - \bE\succeq 1$ actually implies $x-\bE\succ 1$, so $x\notin \bE+\cO_*$ and $\partial x\succ \partial \cO_* \supseteq \partial \co_*$.
    \end{proof}
\end{lemma}

\begin{lemma}\label{lem:valuation-monotonicity}
    Let $(\bE_*, \cO_*)$ be a valued field, $\partial: \bE_* \to \bE_*$ be a derivation, $\bE=\Kr(\partial)$, and $y \in \bE_*^{\neq0}$. Suppose that one of the two following conditions holds
    \begin{enumerate}
        \item $1+\co_*$ is $\lder$-convex w.r.t.\ $\preceq_{\cO_*}$, and $\rv_{\cO_*}(y)\notin \rv_{\cO_*}(\bE)$;
        \item $\cO_* \setminus \co_*$ is weakly $\lder$-convex w.r.t.\ $\preceq_{\cO_*}$, and $\val_{\cO_*}(y)\notin \val_{\cO_*}(\bE)$.
    \end{enumerate}
    Then for all $x \in \bE_*$, $x \preceq _{\cO_*}y \Rightarrow \partial x \preceq_{\cO_*} \partial y$, and  $x \prec_{\cO_*} y \Rightarrow \partial x \prec_{\cO_*} \partial y$.
    \begin{proof}
        This follows immediately from the previous remark: in both cases the hypothesis on $y$ yields $\partial(y\cO_*)\subseteq \cO_*\partial y$ and $\partial (y \co_*)\subseteq \co_* \partial y$. So when $x\in \cO_* y$ (resp.\ $x \in \co_*y$) we get $\partial x \in \cO_* \partial y$ (resp.\ $\partial x \in \co_* \partial y$).
    \end{proof}
\end{lemma}

\begin{corollary}\label{cor:convexity_in_pv}
    Let $(\bE_*, \cO_*, \partial)$ be a valued differential field, $\bE \coloneqq \Kr(\partial)$, and $y \in \bE_*^{\neq 0}$.
    \begin{enumerate}
        \item If $\cO_* \setminus \co_*$ is weakly $\lder$-convex, $\val_{\cO_*}(x-\bE)\not \subseteq \val_{\cO_*}(\bE)$ for all $x \in \bE_* \setminus \bE$, and $\val_{\cO_*}(y)\notin \val_{\cO_*}(\bE)$, then $y\cO_*$ is $\partial$-convex w.r.t.\ $\preceq_{\cO_*}$.
        \item If $1+\co_*$ is $\lder$-convex, $\rv_{\cO_*}(x-\bE) \not \subseteq \rv_{\cO_*}(\bE)$ for all $x \in \bE_*\setminus \bE$, and $\rv_{\cO_*}(y) \notin \rv_{\cO_*}(\bE)$, then $y\cO_*$ is $\partial$-convex w.r.t.\ $\preceq_{\cO_*}$.
    \end{enumerate}
    In particular when one of the two above applies, and $\cO_*$ is $\partial$-convex, then $\partial$ is valuation-convex.
    \begin{proof}
        In both cases we need to show that $x \notin \bE+y\cO_*\Rightarrow \partial x \succ \partial(y \cO_*)$.
        Assuming (1), pick $c\in \bE$ such that $\val(x-c)\notin \val(\bE)$ (resp.\ assuming (2), s.t. $\rv(x-c)\notin \rv(\bE)$). Then $x-c \succ y$ and respectively hypotheses (1) or (2) of Lemma~\ref{lem:valuation-monotonicity} hold, so we get $\partial x = \partial(x-c) \succ \partial y$. Finally using the hypothesis on $y$, we derive $\partial(y\cO_*)\subseteq \cO_* \partial y$ and conclude.
        For the last clause recall that if $\cO_*$ is $\partial$-convex, then so are all $y \cO_*$ with $\val(y)\in \val(\bE)$.
    \end{proof}
\end{corollary}

\begin{remark}
    Note that if there is $y \in \val^{-1}(\bE) \setminus \rv^{-1}(\bE)$ (so in particular if $\partial$ has type (R)) valuation-convexity of $\partial$ is in fact equivalent to the condition of $1+\co_*$ being $\lder$-convex together with $\rv(x-\bE)\not \subseteq \rv(\bE)$ for all $x\in \bE_* \setminus \bE$.
\end{remark}

Valuation-convexity of a derivation is preserved under coarseings of the valuation.

\begin{lemma}\label{lem:convex_der_val_relaxation}
    Suppose that $\partial$ is a valuation-convex derivation on $(\bE_*, \cO_*)$. Then $\partial$ is valuation-convex w.r.t.\ $(\bE_*, \cO_\circ)$ for every valuation subring $\cO_\circ\supseteq \cO_*$.
    \begin{proof}
        We are going to prove in order
        \begin{enumerate}
        \item if $N$ is a $\cO_\circ$-submodule and $\val_{\cO_\circ}(\partial N)$ has no minimum, then $\partial N$ is $\preceq_{\cO_\circ}$-convex in $\partial \bE_*$;
        \item if $\partial\cO_\circ$ is $\preceq_{\cO_\circ}$-convex in $\partial \bE_*$, then $\partial$ is valuation-convex on $(\bE_*, \cO_\circ)$;
        \item $\partial\cO_\circ$ is $\preceq_{\cO_\circ}$-convex.
        \end{enumerate}
        
        (1) Note that since $\partial$ is $\cO_*$-valuation-convex, $\partial x \notin \partial N\Rightarrow \partial x \succ_{\cO_*} \partial N$.
        Now for all $n \in N$, we have $n_1 \in N$ with $\partial n_1\succ_{\cO_\circ} \partial n$ so since $\partial x \succ_{\cO_*} \partial n_1$ we have $\partial x \succ_{\cO_\circ} \partial n$ as desired.

        (2) We need to show that if $N$ is a $\cO_\circ$-submodule, then $\partial N$ is $\preceq_{\cO_\circ}$-convex in $\partial \bE_*$. By Lemma~\ref{lem:f-convex_directed_unions}, we can reduce to the case $N= a \cO_\circ$. Also note that $\partial N$ is $\preceq_{\cO_\circ}$-convex if and only if $\partial N = \partial \bE_* \cap \cO_\circ \partial N$.
        
        Note that since $\partial(a \cO_\circ)$ will be $\preceq_{\cO_*}$-convex in $\partial\bE_*$ (i.e.\ $\partial(a\cO_\circ)=\partial\bE_* \cap \cO_*\partial(a\cO_\circ)$) it suffices to show that $\cO_*\partial(a\cO_\circ)$ is a $\cO_\circ$-module.
        
        If we have $\cO_*\partial(a\cO_\circ) = a \cO_* \partial \cO_\circ + \partial a\cdot \cO_\circ$ then this follows directly from the fact that $\partial \cO_\circ$ is $\preceq_{\cO_\circ}$-convex in $\partial \bE_*$. If not, then by Lemma~\ref{lem:module-derivatives}(3), we should have that $\lder(x) \sim_{\cO_*} -\lder(a)$ for all $x \in \cO_\circ$ with $x\succ_{\cO_*} x_0$ for some $x_0\in \cO_\circ$, but this is not possible as $\cO_\circ$ is a subring so $x^2 \in \cO_\circ$ and $\lder(x^2)=2\lder(x)\not\sim_{\cO_*} \lder(x)$.

        (3) Since we already know that $\partial\cO_\circ$ is $\preceq_{\cO_*}$-convex it suffices to show that $\cO_*\partial \cO_\circ$ is a $\cO_\circ$-module. Note that if $\val_{\cO_*}(\bE)$ is coinitial in $\val_{\cO_*}(\cO_\circ)$ or if $\val_{\cO_\circ}(\partial \cO_\circ)$ has no minimum, then this is trivial. So we can assume that $\cO_{\circ} \partial\cO_{\circ}=\cO_{\circ} \partial r$ and that $\partial$ has type (R) w.r.t.\ $\cO_\circ$, because otherwise we could pick such $r$ in $\co_\circ\subseteq\co_*\subseteq \cO_*$ and we would have $\cO_\circ \partial\cO_{\circ} = \cO_\circ\partial r= \cO_{\circ} \partial \cO_*$, but by Lemma~\ref{lem:module-derivatives}(2), $\cO_{\circ} \partial \cO_*\subseteq \cO_*\partial \cO_{\circ}$.

        Now, setting $\partial_r\coloneqq \partial/\partial r$ we have that $\res_{\cO_\circ}(\partial_r): \res_{\cO_\circ}(\bE_*)\to \res_{\cO_{\circ}}(\bE_*)$ is a non-trivial derivation which is valuation-convex w.r.t.\ $\res_{\cO_\circ}(\cO_*)$, but then for all $x \in \cO_\circ$, by Lemma~\ref{lem:unbounded_integrals_case_v} we can find $y\in \cO_\circ$ such that $\partial_r y \succ_{\cO_*} x$, so $\cO_\circ \subseteq \CH^{\preceq_{\cO_*}}_{\bE_*}(\partial_r \cO_\circ)= \cO_* \partial_r \cO_\circ$ and $\cO_\circ\partial \cO_\circ = \cO_\circ \partial r\subseteq \cO_* \partial \cO_\circ$.
    \end{proof}
\end{lemma}

\begin{corollary}\label{cor:order-convex_implies_val-convex}
    Let $\bE_*$ be an ordered field and $\partial:\bE_* \to \bE_*$ be a derivation and $\cO_*$ be a convex subring of $\bE_*$, then
    \begin{enumerate}
        \item if $\partial$ is order-convex, then it is $\cO_*$-valuation-convex;
        \item if $\partial$ is order-Liouville-convex and $\cO_*\setminus \co_*$ is $\lder$-convex w.r.t.\ $\preceq_{\cO_*}$, then $\partial$ is logarithmically $\cO_*$-valuation-convex.
    \end{enumerate}
    \begin{proof}
        (1) Note that it is valuation-convex for $\CH^{\le}(\bZ)$ by Remark~\ref{rmk:order-convex_to_val-convex}, then apply Lemma~\ref{lem:convex_der_val_relaxation} to conclude that $\partial$ is in fact $\cO_*$-valuation-convex.

        (2) By (1), we only need to show that $\partial$ satisfies (\ref{axiom:new-type-H-relative}).
        Let $\cO_\circ=\CH^{\le}_{\bE_*}(\bZ)$.
        Suppose that there are no constants $c$ with $x \preceq_{\cO_*} c \preceq_{\cO_*} y$, then a fortiori there are no constants $c$ with $x \preceq_{\cO_\circ} c \preceq_{\cO_\circ} y$, so by Lemma~\ref{lem:der_local_monotone}, we can conclude that if $\lder(x) \prec \lder(y)$, we have that for all $x\prec_{\cO_\circ} z \prec_{\cO_\circ} y$, and thus a fortiori for all $x \prec_{\cO_*} \prec z \prec_{\cO_*} y$, we have $\lder(z)\preceq_{\cO_\circ} \lder(y)$ and a fortiori $\lder(z) \preceq_{\cO_*}\lder(y)$.
    \end{proof}
\end{corollary}

\subsection{The case of few constants}\label{ssec:few-constants}
When the constants are included in the valuation ring, the notions of convexity given in Definition~\ref{def:v-convex_derivation} specialize to notions already present in the literature and thoroughly studied in \cite{aschenbrenner2019asymptotic}.

\begin{remark}[Asymptotic and pre-d-valued fields]
    Recall the following notions from \cite[Sec.~9.1, 10.1]{aschenbrenner2019asymptotic}. A valued differential field $(\bE_*, \cO_*, \partial)$ is
    \begin{enumerate}
        \item \emph{asymptotic} if for all $x,y\prec 1$ (or equivalently, by \cite[Prop.~9.1.3]{aschenbrenner2019asymptotic}, all $x,y \notin \cO_* \setminus \co_*)$, $x\prec y \Leftrightarrow \partial x \prec \partial y$;
        \item \emph{d-valued} if it asymptotic and $\res(\Kr(\partial))=\res(\bE_*)$;
        \item \emph{pre-d-valued} if for all $y\preceq 1$ and all $x \prec 1$ (or equivalently all $x \notin \cO_* \setminus \co_*$), $\lder(x) \succ \partial y$;
        \item \emph{asymptotic of type H}, if it is asymptotic and furthermore satisfies for all $x, y\in \bE_*^{\neq 0}$ 
        \begin{equation}\tag{H}\label{axiom:type-H}
            x\prec y\prec 1 \Rightarrow \lder(x) \succeq \lder(y).
        \end{equation} 
    \end{enumerate}
    Also recall that every d-valued field is pre-d-valued, \cite[p.369]{aschenbrenner2019asymptotic}. 
    The next Lemma shows how, when the constants of $\partial$ are included in $\cO_*$ (i.e.\ there are \emph{few constants}, \cite[p.192]{aschenbrenner2019asymptotic}), the notions of convexity of the previous subsection relate to the ones recalled in (1)-(4) above.
\end{remark}

\begin{lemma}\label{lem:pre-d-valued}
    Let $\partial$ be a derivation on $(\bE_*, \cO_*)$ with constants $\bE=\Kr(\partial) \subseteq \cO_*$, then the following hold:
    \begin{enumerate}
        \item $(\bE_*, \cO_*, \partial)$ is asymptotic if and only if $\cO_* \setminus \co_*$ is weakly $\lder$-convex;
        \item $\partial$ is $\cO_*$-valuation-convex if and only if $(\bE_*, \cO_*, \partial)$ is asymptotic and $\co_*$ is $\partial$-convex;
        \item $(\bE_*, \cO_*, \partial)$ is d-valued if and only if $\partial$ is valuation-convex and not of type (R); 
        \item $(\bE_*, \cO_*, \partial)$ is pre-d-valued if and only if $\cO_* \setminus \co_*$ is $\lder$-convex; 
        \item $(\bE_*, \cO_*, \partial)$ satisfies condition (\ref{axiom:new-type-H-relative}) if and only if it satisfies (\ref{axiom:type-H}).
    \end{enumerate}
    \begin{proof}
        (1). That when $\partial$ is valuation-convex and has few constants it is asymptotic immediately follows from Lemma~\ref{lem:valuation-monotonicity}. The converse follows from Corollary~\ref{cor:convexity_in_pv}. Alternatively, \cite[Prop.~9.1.3]{aschenbrenner2019asymptotic} yields the equivalence at once.

        (2). By (1), we only need to show that if $\cO_* \setminus \co_*$ is weakly $\lder$-convex and $\co_*$ is $\partial$ convex, then $\partial$ is valuation-convex. By Lemma~\ref{lem:weak_lder-convex_vs_lder-convex}, we have that then $1+\co_*$ is $\lder$-convex, so since for all $x \notin \bE$, $\rv(x-\bE)\notin \rv(\bE)$, by Corollary~\ref{cor:convexity_in_pv}(2) we get that all $y\cO_*$ with $\val(y) \neq 0$ are $\partial$-convex. But then by Proposition~\ref{prop:v-convexity_reduction} we get that also $\cO_*$ must be $\partial$-convex.

        (3). If $(\bE_*, \cO_*, \partial)$ is d-valued, then it has not type (R) as $\res(\bE)=\res(\bE_*)$ so $\partial\cO_*=\partial \co_*$. Furthermore for all $x \in \bE_*\setminus \bE$, $\val(x-\bE)\not\subseteq \val(\bE)$. Thus since $\cO_*\setminus \co_*$ is weakly $\lder$-convex, by Corollary~\ref{cor:convexity_in_pv}, all $y\cO_*$ with $y \prec 1$ are valuation-convex and by Proposition~\ref{prop:v-convexity_reduction}, $\partial$ is valuation-convex.
        Conversely if $\partial$ is valuation-convex and not of type (R), then $(\bE_*,\cO_*, \partial)$ is asymptotic and furthermore $\res(\bE)=\res(\bE_*)$.

        (4). $\cO_* \setminus \co_*$ is $\lder$-convex if and only if $\val(x) \notin \val(\bE)\Rightarrow \lder(x) \succ \partial \cO_*$. But since $\bE \subseteq \cO_*$, we have $\val(x) \notin \val(\bE) \Leftrightarrow x \not \asymp 1$ and thus this is equivalent to $\lder(x) \succ \partial \cO_*$ for all $x\not \asymp 1$ which is being pre-d-valued.

        (5). Note that since $\bE^{\neq0} \subseteq \cO_*\setminus \co_*$, we have that for $1\prec x\prec y$, $\not\exists c \in \bE,\, 1 \prec c \preceq y$, so (\ref{axiom:new-type-H-relative}) is equivalent to (\ref{axiom:type-H}).
    \end{proof}
\end{lemma}

\begin{example}
    The classic example of a non-d-valued pre-d-valued field, \cite[Example at the end of Subsection~9.1]{aschenbrenner2019asymptotic}, shows also that one cannot dispense with the hypothesis that $\co_*$ is $\partial$-convex in Lemma~\ref{lem:pre-d-valued}(2): in fact, it is an example of a pre-d-valued field of type H $(\bE_*, \cO_*, \partial)$ where $\co_*$ is not $\partial$-convex w.r.t.\ $\preceq_{\cO_*}$.
\end{example}

\begin{corollary}
    If $\partial$ is an order-convex derivation on the ordered field $\bE_*$ with a bounded set of constants $\bE\coloneqq \Kr(\partial)$, then:
    \begin{enumerate}
        \item either $\partial(\bE_*^{>\bE})<0$ or $\partial(\bE_*^{>\bE})>0$;
        \item if $\partial(\bE_*^{>\bE})>0$, $\cO_*\supseteq \bE$ is a convex subring of $\bE_*$, and $\cO_* \setminus \co_*$ is $\lder$-convex w.r.t.\ $\preceq_{\cO_*}$, then $(\bE_*, \cO_*, \partial)$ is a pre-$H$-field in the sense of \emph{\cite[p.385]{aschenbrenner2019asymptotic}}.
    \end{enumerate}
    \begin{proof}
        (1) Suppose there were $x,y>\bE$ such that $\partial x>0$ and $\partial y<0$, then, as already observed, by order-convexity there would be a constant between $x$ and $y$, but this contradicts the hypothesis on $x$ and $y$.

        (2) By Lemma~\ref{lem:pre-d-valued} and the third hypothesis, $(\bE_*, \cO_*, \partial)$ is pre-d-valued, furthermore by the first two hypothesis $\partial(\bE_*^{>\cO_*})>0$.
    \end{proof}
\end{corollary}

\subsection{Valued fields with convex derivation} In this subsection we collect some basic properties that can be derived from some several convexity hypothesis of a derivation with respect to a valuation. 
The interest here laying mainly in valued differential fields arising from an order-Liouville-convex ordered differential field with a convex valuation ring, logarithmic valuation-convexity seems the more natural condition, however most of the time we will need less. 

\begin{remark}
    Note that $\res(\bE)=\res(\bE_*)$ if and only if $\partial \cO_* = \partial \co_*$ if and only if $\lder(\cO_*\setminus \co_*)=\lder(1+\co_*)$.
    
    In fact $\partial y \in \partial \cO_* \setminus \partial\co_*$ if and only if $y \in (\bE+\cO_*)\setminus (\bE+\co_*)$, i.e.\ iff there is a constant $c \in \bE$ such that $\res(y-c)\notin \res(\bE)$. And similarly $\lder y \in \lder(\cO_*\setminus \co_*) \setminus \lder(1+\co_*)$, if and only if there is $c \in \bE^{\neq 0}$ such that $\res(y/c)\notin \res(\bE)$. Conversely if $\res(\bE_*)\neq \res(\bE)$, then there is $y \in\cO_*\setminus (\bE+\co_*)=(\cO_*\setminus \co_*)\setminus \big(\bE^{\neq^0}(1+\co_*)\big)$, thus $\partial y \in \partial \cO_* \setminus \partial \co_*$ and $\lder(y) \in \lder(\cO_*\setminus \co_*) \setminus \lder(1+\co_*)$.
\end{remark}

\begin{lemma}\label{lem:attempted_equivalence}
    Let $(\bE_*, \cO_*, \partial)$ be a valued differential field and $\bE\coloneqq\Kr(\partial)$. If $\partial$ has \emph{not} type (R), then the following are equivalent
    \begin{enumerate}
        \item $\co_*$ is $\partial$-convex;
        \item $\cO_*$ is $\partial$-convex and $\res(\bE)=\res(\bE_*)$. 
    \end{enumerate}
    \begin{proof}
        $(2)\Rightarrow (1)$. If $\cO_*$ is $\partial$-convex and $y \notin \partial \co_*$, then either $y \notin \partial \cO_*$ and $\partial y \notin \cO_*\partial \cO_*\supseteq\cO_* \partial \co_*$, or $\partial y \in \partial \cO_*\setminus \partial \co_*$, but then for some $c \in \bE$, $\res(y-c)\notin \res(\bE)$, contradicting the hypothesis.

        $(1) \Rightarrow (2)$ If $(\bE_*, \cO_*, \partial)$ has type (W) or (G), then $\cO_*\partial \co_*=\cO_*\partial\cO_*$ so if $\co_*$ is $\partial$-convex, $\cO_*$ is as well. But then we have 
        \[y \in \bE+\cO_*\Leftrightarrow \partial y \succ \partial\co_*\Leftrightarrow \partial y \succ\partial \cO_* \Leftrightarrow y \in \bE+\co_*,\]
        from which it follows that $\res(\bE)=\res(\bE_*)$.
    \end{proof}
\end{lemma}

\begin{corollary}\label{cor:convex_and_type-W}
    Let $(\bE_*, \cO_*, \partial)$ be a valued differential field and $\bE\coloneqq\Kr(\partial)$. If $\partial$ has not type (R) and $\co_*$ is $\partial$-convex, then
    \[\cO_*\setminus \co_* \text{ is $\lder$-convex} \Leftrightarrow 1+\co_* \text{ is $\lder$-convex} \Leftrightarrow \cO_* \setminus \co_* \text{is weakly $\lder$-convex}\]
    \begin{proof}
        Note that since $\partial$ has not type (R), we have $\cO_* \partial \cO_*=\cO_*\partial \co_*$, so if $\cO_* \setminus \co_*$ is weakly $\lder$-convex, then it is $\lder$-convex. So we only need to show that under our hypothesis, if $1+\co_*$ is $\lder$-convex then so is $\cO_* \setminus \co_*$, but this follows immediately from Lemma~\ref{lem:attempted_equivalence}, since as $\res(\bE)=\res(\bE_*)$, we have $\rv(x) \notin \rv(\bE)\Rightarrow \val(x) \notin \val(\bE)$ for all $x \in \bE_*\setminus \bE$.
    \end{proof}
\end{corollary}

\begin{remark}\label{rmk:convex_and_type-R}
    If $(\bE_*, \cO_*, \partial)$ is of type (R) with $\cO_*\partial\cO_*=\cO_*\partial r$ for $\partial r \in \partial \cO_*\setminus \partial \co_*$, as observed in Remark~\ref{rmk:type-R}, $\partial_r\coloneqq \partial/\partial r$ is a small derivation with non-trivial residue $\res(\partial_r)$.
    In this situation if $\co_*$ is $\partial$-convex then $\res(\bE)=\Kr(\res(\partial_r))$. If furthermore $\cO_*$ is $\partial$-convex, then $\co_*$ is $\partial$-convex if and only if $\res(\bE)=\Kr(\res(\partial_r))$.
    In particular if $\res(\bE_*, \cO_*,\partial)$ is of type (R) and $\cO_*$-valuation-convex, then $\res_{\cO_*}(\bE_*)\neq \res_{\cO_*}(\bE)$.
\end{remark}

\begin{remark}\label{rmk:br(partial)_ppty}
    Recall the definitions of $\br_{\cO_*}(x/\bE)$ and $\br_{\cO_*}(\partial)$ which we are going to relate in the next Lemma. Note that $1/y \in \br_{\cO_*}(\partial) \Leftrightarrow \partial \cO_*/y \prec 1 \Leftrightarrow y \succ \partial \cO_*$. If $\cO_*$ is $\partial$-convex and $\cO_* \setminus \co_*$ is $\lder$-convex, then
    \[y \in \bE+\cO_* \Leftrightarrow 1/\partial y \notin \br_{\cO_*}(\partial) \quad \text{and} \quad \val(y) \in \val(\bE) \Leftrightarrow 1/\lder(y) \notin \br_{\cO_*}(\partial).\]
\end{remark}

\begin{lemma}\label{lem:breadth-comparison}
    Let $\partial$ be a valuation-convex derivation on $(\bE_*, \cO_*)$ with constants $\bE\coloneqq \Kr(\partial)$, and let $x \in \bE_* \setminus \bE$.
    \begin{enumerate}
        \item $\br_{\cO_*}(x/\bE)= \br_{\cO_*}(\partial/\partial x) \cap \partial^{-1}(\co_* \partial x)$;
        \item If $\partial$ has not type (G), $\Br_{\cO_*}(x/\bE)=\Br_{\cO_*}(\partial/\partial x) \cap \partial^{-1}(\cO_* \partial x)$.
        \item If $\val(x-\bE)\subseteq \val(\bE^{\neq0})$, then $\partial(\br_{\cO_*}(\partial/\partial x))\prec \partial x$
    \end{enumerate}
    Moreover in (1) and (2) the $\supseteq$ inclusions hold also without the assumptions that $\partial$ is valuation-convex and that $\partial$ has not type (G).
    \begin{proof}
        Observe that without any assumption on $\partial$ we have
        \[\begin{aligned}
            &\br_{\cO_*}(x/\bE)\!=\!\{y : x \notin \bE + y \cO_*\}\supseteq\{ y: \partial x \succ \partial(y\cO_*)\}\!=\!\{y: \partial x\succ \partial y, \partial x \succ y\partial \cO_*\},\\
            &\Br_{\cO_*}(x/\bE)\!=\!\{y : x \notin \bE + y \co_*\}\supseteq\{y: \partial x \succ \partial(y\co_*)\}\!\supseteq\!\{y: \partial x\succeq \partial y, \partial x \succ y\partial \co_*\},
        \end{aligned}\]
        where the last set-relation in each line is respectively due to Lemma~\ref{lem:module-derivatives}(1) and (4).
        Now observe that if $\partial$ is valuation-convex, then the middle inclusion of each line is in fact an equality and that, by Lemma~\ref{lem:module-derivatives}(4), when $\partial$ has not type (G), the last inclusion of the second line is an equality as well.
        Finally observe that assertion (3) follows from (1) and Lemma~\ref{lem:br_basic}(3).
    \end{proof}
\end{lemma}

We wish to record in the following Corollary, two observations about how the existence of elements with certain properties, together with some convexity assumptions forces a derivation to have type (R) or (W).

\begin{corollary}\label{cor:der-types_and_absorbed-elements}
    Let $(\bE_*, \cO_*, \partial)$ be a valued differential field, and $\bE\coloneqq \Kr(\partial)$.
    \begin{enumerate}
        \item If $\cO_*$ and $\co_*$ are $\partial$-convex (so in particular if $\partial$ is $\cO_*$-valuation-convex), then $\partial$ has type (R) if and only if $\res_{\cO_*}(\bE_*)\neq \res_{\cO_*}(\bE)$;
        \item If $\partial$ is $\cO_*$-valuation-convex and there is $x \in \bE_*\setminus \bE$ such that $\rv_{\cO_*}(x-\bE)\subseteq \rv_{\cO_*}(\bE)$, then $\partial$ has type (W).
    \end{enumerate}
    \begin{proof}
        (1) is by Remark~\ref{rmk:convex_and_type-R} and Lemma~\ref{lem:attempted_equivalence}.
        
        (2) is because $\lder(x-\bE)\subseteq \lder(1+\co_*)\subseteq \cO_* \partial \co_*\subseteq \cO_* \partial \cO_*$, so for all $c \in \bE$, $1 \in (x-c)\cO_* \partial \cO_*/\partial x$, i.e.\ $x-c\succ \br_{\cO_*}(\partial/\partial x)$ and thus $\br_{\cO_*}(\partial/\partial x) \subseteq \br_{\cO_*}(x/\bE)$. But then by Lemma~\ref{lem:breadth-comparison} $\br_{\cO_*}(\partial/\partial x) = \br_{\cO_*}(x/\bE)$, so in fact $\cO_*\partial \cO_* = \cO_*\lder(x-\bE)$. But then since $\rv(x-\bE)\subseteq \rv(\bE)$ implies that $\val(x-\bE)$ has no maximum and thus $\val(\lder(x-\bE))$ has no minimum, we get that $\cO_*\partial\cO_*$ is non-principal.
    \end{proof}
\end{corollary}

\begin{example}[Main Examples]\label{main:examples}
    If $\bE_*=\bE \langle x\rangle \succ \bE \models T$ is a $1$-$\dcl_T$-dimensional extension of models of $T$, $\cO_*$ is a convex subring of $\bE_*$, and $\partial$ is an $\bE$-linear $T$-derivation on $\bE_*$, then:
    \begin{enumerate}
        \item $(\bE_*, \cO_*, \partial)$ is $\cO_*$-valuation-convex (by Proposition~\ref{prop:der_order-convexity} and Corollary~\ref{cor:order-convex_implies_val-convex}(2));
        \item if $\partial$ is of type (W), (equiv.\ $\res(\bE)=\res(\bE_*)$) then $\partial$ is logarithmically $\cO_*$-valuation-convex (by Proposition~\ref{prop:der_order-convexity}, Corollary~\ref{cor:convex_and_type-W}, Corollary~\ref{cor:order-convex_implies_val-convex}(2));
        \item if $\partial$ is of type (R), then $\partial$ is logarithmically $\cO_*$-valuation-convex provided that one of the following holds:
        \begin{enumerate}
            \item $\cO_*=\CH_{\bE_*}(\bZ)$;
            \item $T$ is exponential and $\exp(\cO_*)\subseteq \cO_*$ (by Proposition~\ref{prop:der_order-convexity} and Corollary~\ref{cor:order-convex_implies_val-convex}(2));
            \item $T$ is power-bounded and $\cO_*\in \{\cO_*^+, \cO_*^-\}$ for some $T$-convex $\cO$ in $\bE$ (since in this case by the valuation property $\val(\bE)=\val(\bE_*)$).
        \end{enumerate}
        A natural question here is whether, in the case $\cO_*\in \{\cO_*^+, \cO_*^-\}$, with $\cO$ convex in $\bE$, it is enough that $\cO$ has an o-minimal residue-field sort (in the sense of \cite{marikova2011minimal}) to ensure that $\partial$ is logarithmically valuation-convex.
    \end{enumerate}
\end{example}

We conclude this subsection with the following important observation relating the property of an element to be $\cO_*$-absorbed over the constants, with a property of the derivation ``at'' that element.

\begin{corollary}\label{cor:abs_basic_equivalence}
    Let $(\bE_*, \cO_*, \partial)$ be a valued differential field with $\partial$ quasi logarithmically $\cO_*$-valuation-convex, and $\bE\coloneqq \Kr(\partial)$. Then for all $x \in \bE_*$, the following are equivalent
    \begin{enumerate}
        \item $\val(x-\bE) \subseteq \val(\bE)$;
        \item $\br_{\cO_*}(x/\bE) =  \br_{\cO_*}(\partial) \cdot \partial x$;
        \item $\partial(\br_{\cO_*}(\partial) \cdot \partial x) \prec \partial x$.
    \end{enumerate}
    \begin{proof}
        Note that $(3)\Leftrightarrow (2)$ follow directly from Lemma~\ref{lem:breadth-comparison}(1), as $(3)$ is equivalent to $\partial (\br_{\cO_*}(\partial /\partial x))\prec \partial x$ i.e.\ $\br_{\cO_*}(\partial /\partial x)\subseteq \partial^{-1}(x \co_*)$.
        $(1) \Rightarrow (2)$ is Lemma~\ref{lem:breadth-comparison}(3). $\lnot(1) \Rightarrow \lnot(2)$ is because when $\cO_*\setminus \co_*$ is $\lder$-convex and $\val(x-c) \notin \val(\bE)$ for some $c \in \bE$, then $x-c \in \br_{\cO_*}(\partial)\partial x$.
    \end{proof}
\end{corollary}

\begin{remark}
    Note that it follows from Corollary~\ref{cor:der-types_and_absorbed-elements} that if there is an element $y\in \bE_*\setminus \bE$ with $\rv(y-\bE) \subseteq \rv(\bE)$, then in Corollary~\ref{cor:abs_basic_equivalence} we only need $\partial$ to be $\cO_*$-valuation-convex and (1) is in fact equivalent to $\rv(y-\bE)\subseteq \rv(\bE)$. 
\end{remark}

\subsection{Absorbed derivations}\label{ssec:absorbed-der}
Note that (3) in Corollary~\ref{cor:abs_basic_equivalence} can be restated as saying that the derivation $\partial_x\coloneqq \partial/\partial x$ (provided $\partial x \neq0$), is such that $\partial_x \br_{\cO_*}(\partial_x)\prec 1$. We thus abstract the notion of being the derivation ``at'' an absorbed element.

\begin{definition}\label{def:absorbed-type-der}
    Let $(\bE_*, \cO_*)$ be a valued field, $\partial$ be a derivation. Let $\bE\coloneqq\Kr(\partial)$.
    We will say that $\partial$ is \emph{$\cO_*$-semiabsorbed} if $\cO_*\setminus \co_*$ is weakly $\lder$-convex and $\partial(\br_{\cO_*}(\partial))\prec 1$. We will say that $\partial$ is \emph{$\cO_*$-absorbed} if furthermore there is $x\in \bE_* \setminus \bE$ such that $\val(x-\bE)\subseteq \val(\bE)$ and $\partial$ is q.log.\ valuation-convex. 
\end{definition}

\begin{remark}
    Note that $\partial\br_{\cO_*}(\partial)\prec 1$ is equivalent to 
    \begin{equation}\label{eq:absorbed-reformulation}\tag{A}
        \forall y \succ \partial \cO_*,\; \lder (y) \prec y
    \end{equation}
\end{remark}

\begin{remark}\label{rmk:w-absorbed_and_R_impplies_absorbed}
      If $\partial$ has type (R) and is semiabsorbed and q.log.\ valuation convex, then it is absorbed: indeed if $\cO_* \partial \cO_*= \cO_*\partial r \supsetneq \cO_* \partial \co_*$ then $r \in (\bE+\cO_*) \setminus (\bE+\co_*)$, so $\val(r-\bE)\subseteq \val(\bE)$.
\end{remark}

The property of being semiabsorbed is not invariant under rescaling.

\begin{lemma}\label{lem:absorbed-type_derivatives}
    Suppose that $\partial(\br_{\cO_*}(\partial)) \prec 1$. Then $\partial(\br_{\cO*}(\partial) \cdot x) \prec x$ if and only if $\lder(x)\in \cO_* \partial \cO_*$. Thus in particular if furthermore $\cO_*\setminus \co_*$ is $\lder$-convex, then for all $x \in \bE_*$, $\val(x)\in \val(\bE) \Leftrightarrow \partial(\br_{\cO_*}(\partial)\cdot x)\prec x$ and for $x \neq 0$, $\partial/x$ is semiabsorbed if and only if $\val(x) \in \val(\bE)$.
    \begin{proof}
        Since for all $b \in \br_{\cO_*}(\partial)$, $\partial b\prec 1$, we have that the condition
        \[\forall b \in \br_{\cO_*}(\partial), \; x\succ \partial(b \cdot x) = (\partial b) \cdot x + b \cdot \partial x\]
        is in fact equivalent to $\forall b \in \br_{\cO_*}(\partial), \; x \succ b \cdot \partial x$,
        that is, to $1/(\lder x) \succ \br_{\cO_*}(\partial)$, or $\lder(x) \in \cO_*\partial \cO_*$, which, when $\cO_*\setminus \co_*$ is $\lder$-convex is equivalent to $\val(x)\in \val(\bE)$.
    \end{proof}
\end{lemma}

The following Corollary is a summary of previous results.

\begin{corollary}\label{cor:absorbed-der_sumup}
    Suppose that $\partial$ is an absorbed derivation on $(\bE_*,\cO_*)$ and $\Kr(\partial)=\bE$. Then for all $y \in \bE_*$ the following are equivalent:
    \begin{enumerate}
        \item $\val(y-\bE)\subseteq \val(\bE)$;
        \item $\br_{\cO_*}(y/\bE)=\br_{\cO_*}(\partial )\cdot \partial y$;
        \item $\partial(\br_{\cO_*}(\partial) \cdot \partial y) \prec \partial y$ (that is, either $y \in \bE$ or $\partial/\partial y$ is absorbed);
        \item $\val(\partial y) \in \val (\bE)$.
    \end{enumerate}
    \begin{proof}
        The equivalences between (1), (2) and (3) are in Corollary~\ref{cor:abs_basic_equivalence}. The equivalence of (3) and (4) is Lemma~\ref{lem:absorbed-type_derivatives}.
    \end{proof}
\end{corollary}

The following easy computational observation will be particularly useful in the next subsection:
\begin{equation}\label{eq:small_trick}\tag{E4.1}
    \partial\left(\frac{u}{\partial u}\right)=\partial\left(\frac{1}{\lder u}\right)=1-\frac{\lder(\partial u)}{\lder u}.
\end{equation}

We give it here because it has the following easy consequence relevant for semiabsorbed derivations which instead will be useful in Subsection~\ref{ssec:height}.

\begin{lemma}\label{lem:der_replacement}
    Suppose that $\partial \br_{\cO_*}(\partial)\prec 1$, $\cO_* \setminus \co_*$ is $\lder$-convex, $y \in \bE_*$, $c \in \bE=\Kr(\partial)$. If $\val(y-c) \notin \val(\bE)$, then $y-c \sim \partial y/\lder\partial y$.
    \begin{proof}
        Since $\val(y-c) \notin \val(\bE)$, we have $1/\lder(y-c)\in \br_{\cO_*}(\partial)$ so 
        \[1 \succ \partial(1/\lder(y-c))=1-\lder(\partial y)/\lder(y-c).\]
        whence $\lder(\partial y)\sim \lder(y-c)$ and $y-c \sim (\partial y)^2/\partial^2y\sim \partial y / \lder\partial y$.
    \end{proof}
\end{lemma}

\subsection{Asymptotic Integration}\label{ssec:asymptotic_integration}
Recall that an \emph{asymptotic integral} for an element $y$ of $(\bE_*,\cO_*, \partial)$ is an element $x$ such that $y\sim \partial x$. We call \emph{weak asymptotic integral} of $y$ an element $x$ such that $y \asymp \partial x$. Also recall that $\partial$ has \emph{asymptotic integration} (resp.\ \emph{weak ---}) when $\rv(\partial(\bE_*))=\rv(\bE_*)$ (resp.\ $\val(\partial \bE_*) = \val(\bE_*)$).

\smallskip

The following Lemma is an adaptation of \cite[Thm.~1]{rosenlicht1983rank} and variations thereof.

\begin{lemma}\label{lem:simil_rosenlicht}
    Suppose that $y\in \bE_*^{\neq0}$ and $h, u \in \bE_* \setminus \bE$. Then:
    \begin{enumerate}
        \item $y/\lder(h)$ is not an asymptotic integral of $y$ if and only if $\lder(h) \preceq \lder(\partial h/y)$;
        \item if $\lder(u)\succ \partial \cO_*$ and $y/\partial(y/\lder(u))$ is not an asymptotic integral of $y$, then $\lder(u) \sim \lder((\partial u)/y)$ (equiv.\ $\lder(u) \succ \lder(\lder(u)/y)$);
        \item if $\lder(u)\succ \partial \co_*$ and $y/\partial(y/\lder(u))$ is not an asymptotic integral of $y$, then $\lder(u) \succeq \lder((\partial u)/y)$ (equiv.\ $\lder(u) \succeq \lder(\lder(u)/y)$).
    \end{enumerate} 
    \begin{proof}
        (1) Note that $\partial (y/\lder(h))\sim y$ if and only if $\lder(y/\lder(h))\sim \lder(h)$ if and only if $\lder(h) \succ \lder(y/\partial h) = -\lder(\partial h/y)$.
        
        Before proving (2) and (3), it is worth noting that $\lder(h)\succ \lder(y/\partial h)$ can be rewritten as $y/h\succ \partial (y/\partial h)$ and that therefore setting $h=y/\lder(u)$ we get 
        \[\partial(y/\partial(y/\lder(u))) \sim y \Longleftrightarrow \lder(u) \succ \partial(y/\partial (y/\lder(u)).\]
         
        Thus if $\lder(u) \succ \partial \cO_*$, for $y/\partial(y/\lder(u))$ to be an asymptotic integral of $y$ it suffices that $y/\partial (y/\lder(u))\preceq 1$, i.e.\ $y \preceq \partial (y/\lder(u))$, or yet again equivalently $\lder(u) \preceq \lder(y/\lder(u))=-\lder(\lder(u)/y)$, from which we derive (2).

        Similarly if $\lder(u) \succ \partial \co_*$ for $y/\partial(y/\lder(u))$ to be an asymptotic integral of $y$ it suffices that $y/\partial (y/\lder(u))\prec 1$, which can be rewritten as $\lder(u) \succ \lder(\lder(u)/y)$, which gives (3).
    \end{proof}
\end{lemma}

\begin{proposition}[Asymptotic integration]\label{prop:asymptotic_integration}
    Let $(\bE_*, \cO_*, \partial)$ be a valued differential field, $\bE=\Kr(\partial)$, $y \in \bE_*$. Suppose that one of the two following conditions holds:
    \begin{enumerate}
        \item $\cO_* \setminus \co_*$ is $\lder$-convex and $y$ doesn't have an asymptotic integral;
        \item $\cO_*\setminus \co_*$ is weakly $\lder$-convex and $y$ doesn't have a weak asymptotic integral.
    \end{enumerate}
    Then for all $u \in \bE_*\setminus \val^{-1}(\bE)$ we have $\partial(y/\lder(u))\prec y$ and $\partial(u^2/\partial u)\sim u y$, in particular whenever $\val(w) \notin \val(\bE)+\val(y)$, $w$ has an asymptotic integral.
    \begin{proof}
        Assuming that $\cO_* \setminus \co_*$ is $\lder$-convex, if for some $u\in \bE_* \setminus \val^{-1}(\bE)$ we had $\partial(y/\lder(u))\succeq y$, then $\lder(y/\lder(u))\succeq \lder(u)$, i.e.\ $\lder(u)\not \sim \lder(\partial u)-\lder(y)$ and $y$ would have an asymptotic integral by Lemma~\ref{lem:simil_rosenlicht}(2). So under condition (1) we must have $\partial(y/\lder(u))\prec y$ for all such $u$.
        Similarly assuming $\cO_*\setminus \co_*$ is weakly $\lder$-convex and that $y$ doesn't have an asymptotic integral, using Lemma~\ref{lem:simil_rosenlicht}(3), we derive that for all $u \in \bE_*\setminus \val^{-1}(\bE)$, $\partial(y/\lder(u))\preceq y$, so assuming (2) we must also have $\partial(y/\lder(u))\prec y$.

        But now rewriting $\partial(y/\lder(u))\prec y$ in terms of $\tilde{\partial}=\partial/y$ yields $\tilde{\partial}(1/\tilde{\lder}(u))\prec 1$, i.e.\ $u\tilde{\partial}^2u\sim (\tilde{\partial}u)^2$ from which we get
        \[y^{-1}\partial(u^2y/\partial(u))=\tilde{\partial}(u^2/\tilde{\partial}(u))=u\left(2-\frac{u \tilde{\partial}^2u}{(\tilde{\partial} u)^2}\right)\sim u,\]
        so the statement is proven.
    \end{proof}
\end{proposition}

\begin{corollary}
    If $(\bE_*, \cO_*, \partial)$ is a valued differential field with $\partial\neq0$. Then $\val(\partial(\bE_*\setminus \bE))$ is cofinal in $\val(\bE_*^{\neq0})$.
\end{corollary}

\begin{corollary}\label{cor:asymptotic_integration1}
    Let $(\bE_*, \cO_*, \partial)$ be a valued differential field, $\bE=\Kr(\partial)$.
    \begin{enumerate}
        \item If $\partial(\br_{\cO_*}(\partial))\prec 1$ and $\cO_* \setminus \co_*$ is $\lder$-convex, then every $y$ with $\val(y)\notin \val(\bE)$ has an asymptotic integral.
        \item If $\partial$ has type (R) and $\cO_*\setminus \co_*$ is weakly $\lder$-convex, then it has weak asymptotic integration, if furthermore $\cO_*\partial \cO_*=\cO_*\partial r$, then $\partial\br_{\cO_*}(\partial/\partial r)\prec \partial r$.
        \item If $\partial(\br_{\cO_*}(\partial))\prec 1$ and $\cO_* \setminus \co_*$ is weakly $\lder$-convex, then every $y$ with $\val(y)\notin \val(\bE)$ has a weak asymptotic integral.
    \end{enumerate}
    \begin{proof}
        Note that if $\lder(y)\succ \partial \cO_*$, then $\partial(1/\lder(y))=1-\lder(\partial y)/\lder(y)\prec 1$ so we derive $y\partial^{2}y \sim (\partial y)^2$ and hence $\partial(y^2/\partial y) \sim y$.
        
        (1) If $\cO_* \setminus \co_*$ is $\lder$-convex then $\val(y) \notin \val(\bE)\Rightarrow \lder(y)\succ \partial\cO_* \Rightarrow \partial(y^2/\partial y) \sim y$.

        (2) If $\partial$ has type (R), there is $r\in \cO_*\setminus \co_*$ such that $\cO_*\partial \cO_*=\cO_*\partial r$ and $\partial r \succ \partial \co_*$. Set $\partial_r\coloneqq \partial/\partial r$. Then $\cO_*\partial_r\cO_*=\cO_*$ so $\br_{\cO_*}(\partial_r)=\co_*$ and $\partial_r(\br_{\cO_*}(\partial_r))=\partial_r(\co_*) \prec 1$. Suppose that $\val(y)\notin \val(\bE)$, then $\lder_r(y)\succ \partial_r\co_*$. If $\lder_r(y) \succ \partial_r \cO_*$ then $\partial_r(y^2/\partial_r(y))\sim y$ and we are done. Otherwise $\lder_r(y) \asymp 1$ so $\partial_r y \asymp y$. Finally if $\val(y) \in \val(\bE)$, then for $y \asymp c \in \bE$ we have $\partial_r (rc)=c\asymp y$. So we have shown that $\partial_r$ has weak asymptotic integration and thus also $\partial$ does.
        
        (3) If $\partial$ has type (R) then this follows from (2), if instead $\partial$ has not type (R), that $\cO_* \setminus \co_*$ is weakly $\lder$-convex implies that it is in fact $\lder$-convex and this follows from (1).
    \end{proof}
\end{corollary}

\begin{remark}\label{rmk:induced_order_bijection}
    If $\partial(\br_{\cO_*}(\partial))\prec 1$ and $\partial$ is q.log.\ valuation-convex, then by Corollary~\ref{cor:absorbed-der_sumup} we have
    \[\val(\partial y) \in \val(\bE)\Leftrightarrow \partial(\br(\partial) \partial y) \prec \partial y \Leftrightarrow \val(y-\bE)\subseteq \val(\bE).\]
    So $\{x : \val(x) \notin \val(\bE)\}$, not only has asymptotic integration but is closed under asymptotic integration, in the sense that if $\val(x) \notin \val(\bE)$, then there is $y$ with $\val(y) \notin \val(\bE)$ such that $\partial y \sim x$: this is because if $\tilde{y}$ is any asymptotic integral for $x$, then for some $c \in \bE$, $\val(\tilde{y}-c)\notin \val(\bE)$.
    Thus in this case we get that $\partial$ induces an order bijection $\val(\partial): \val(\bE_*)\setminus \val(\bE) \to \val(\bE_*)\setminus \val(\bE)$.
\end{remark}

The following Lemma shows that under the assumption $\partial$ is valuation-convex, the situation where $\partial(\br_{\cO_*}(\partial))\prec 1$ and there are no asymptotic integrals for the constants is in fact, up to rescaling $\partial$, the only situation in which some weak asymptotic integral is missing.

\begin{lemma}\label{lem:asymptotic_integration2}
    If $\partial$ is valuation-convex and $y$ does not have weak asymptotic integrals, then $\partial$ has not type (R) and $\partial/y$ is semiabsorbed (i.e.\ $\partial(\br_{\cO_*}(\partial) y)\prec y$).
    \begin{proof}
        The fact that $\partial$ has not type (R) is Corollary~\ref{cor:asymptotic_integration1}(2).

        Now we show that $\partial/y$ is semiabsorbed.
        Note that since $\partial$ has not type (R), and it is $\cO_*$-valuation-convex, we have that in fact $\cO_*\setminus \co_*$ is $\lder$-convex.
        Thus, by Lemma~\ref{lem:breadth-comparison}(3), and Corollary~\ref{cor:asymptotic_integration1}(1), there can be no elements in $\bE_*$, that are absorbed over $\bE$ for otherwise $\partial$ would have weak asymptotic integration.
        We also know by Proposition~\ref{prop:asymptotic_integration}, that for all $u \in \bE_*\setminus \val^{-1}(\bE)$, $\partial(y/\lder(u))\prec y$ from which it follows that $\partial (y/\lder(u) \cO_*) \preceq y$ and thus since $y$ does not have a weak asymptotic integral, $\partial (y/\lder(u) \cO_*) \prec y$.
        Now observe that for all $t\in \bE_*$, if $\partial t \prec \lder(u)$ for all $u$ with $\val(u) \notin \val(\bE)$, then since $t$ is not absorbed over $\bE$, there is $c \in \bE$ such that $\val(t-c)\notin \val(\bE)$ whence $\partial t \prec \lder(t-c)$ and $t-c \prec 1$ so $y/\partial t\notin y\Br_{\cO_*}(\partial)$. Thus we have shown that
        \[\val(y\Br_{\cO_*}(\partial))\cap \val(y/\partial\bE_*) \subseteq \val(\cO_*\{y/\lder(u):\val(u) \notin \val(\bE)\})\]

        We first prove that $\partial/y$ is semiabsorbed under the assumption that $\val(\bE)$ has empty interior in $\val(\bE_*)$. In that case we derive the following:
        \begin{enumerate}
            \item $\cO_*\{y/\lder(u): \val(u) \!\notin\! \val(\bE)\}\!\supseteq\! y\co_*\Br_{\cO_*}(\partial) \!\supseteq\! y\co_*\br_{\cO_*}(\partial)$ so $\partial(\co_*y\br_{\cO_*}(\partial))\prec y$;
            \item if $y\br_{\cO_*}(\partial) \!\neq\! \cO_*\{y/\lder(u): \val(u) \!\notin\! \val(\bE)\}$ then $y\br_{\cO_*}(\partial) \!=\! \cO_* c$ for some $c \!\in\! \bE$.
        \end{enumerate}
        Now if $y\br_{\cO_*}(\partial)= \cO_*\{y/\lder(u): u \notin \val(\bE)\}$ we are done by (1), otherwise by (2) $y\br_{\cO_*}(\partial)=\cO_*c$, and thus by (1), $\partial (c\co_*)\prec y$, but then by Remark~\ref{rmk:basic_inclusion} it must be $\partial (c\cO_*)\preceq y$ and since $y$ has no weak asymptotic integral $y \succ \partial (c\cO_*) =\partial(y\br_{\cO_*}(\partial))$.
        
        To complete the proof, we reduce to the case when $\val(\bE)$ has empty interior by considering 
        \[\cO_\circ\coloneqq\cO_*\cdot \{x \in \bE_*^{\succeq 1}: [0, -\val(x)] \subseteq \val(\bE)\}.\]
        Note that $\cO_\circ$ is a subring because it is a $\cO_*$-submodule and if $[0, -\val(x)] \subseteq \val(\bE)$, then $[0, -\val(2x)]\subseteq \val(\bE)$ as $\val(\bE)$ is a subgroup of $\val(\bE_*)$.

        Then, by construction, $y$ still doesn't have a weak asymptotic integral for $\cO_\circ$ and by Lemma~\ref{lem:convex_der_val_relaxation} $\partial$ is $\cO_\circ$-valuation-convex, furthermore $\val_\circ(\bE)$ has empty interior in $\val_\circ(\bE_*)$.

        Thus it follows that $\partial(\br_{\cO_\circ}(\partial)y)\prec y$ and $\partial$ does not have type (R) for $\cO_\circ$, so $\cO_\circ \partial \cO_\circ=\cO_\circ\partial \co_\circ$.
        Also observe that $\cO_\circ\partial \co_\circ=\cO_*\partial \co_\circ$ and $\cO_\circ\partial \cO_\circ=\cO_*\partial \cO_\circ$ because $\val(\cO_\circ\setminus \co_\circ) \subseteq \val(\bE)$.
        Finally since
        \[\cO_\circ \partial \co_\circ=\cO_*\partial \co_\circ\subseteq\cO_*\partial \cO_*\subseteq \cO_*\partial \cO_\circ=\cO_\circ \partial \cO_\circ,\]
        we get that $\cO_* \partial \cO_*=\cO_\circ \partial \cO_\circ$.
        But then $x \in \br_{\cO_*}(\partial)\Rightarrow 1/x \succ \partial\cO_* \Rightarrow 1/x \succ_\circ \partial \cO_\circ\Rightarrow x \in \br_{\cO_\circ}(\partial)$ and we are done.
    \end{proof}
\end{lemma}

\subsection{Logarithmic height over a derivation}\label{ssec:height}
Throughout this section $(\bE_*, \cO_*)$ will denote a valued field and $\partial:\bE_* \to \bE_*$ an q.log.\ $\cO_*$-valuation-convex derivation with constants $\Kr(\partial)=\bE$. We are also going to assume that $\partial$ has not type (G).

As mentioned in the introduction we are ultimately interested in the problem of understanding, given a compatible exponential on $\bE_*$, when a non-absorbed element can be sent to an absorbed one by performing logarithms, sign-changes and translations by constants. We give a slightly more general treatment in a setting where we might not have an exponential.

\begin{definition}\label{def:log-scale}
    Given a valued field $(\bE_*, \cO_*)$, I call \emph{$\partial$-logarithmic scale} (resp.\ \emph{approximate $\partial$-logarithmic scale}) a sequence $(y_n)_{n<m}$ with $m \le \omega$ such that $\partial y_{n}\in \partial y_{n+1}\cdot(y_n - \bE)$ (resp.\ $\rv(\partial y_{n}) \in \rv((y_n - \bE)\partial y_{n+1})$) for all $n$ with $n+1<m$.
    I say that an (approximate) logarithmic scale $(y_n)_{n<m}$ is \emph{essential} if for all $n<m$, there is $c_n\in \bE$ such that $\val(y_n-c_n)\notin \val(\bE)$ and for $n+1<m$, $\partial y_n \sim (y_n-c_n)\partial y_{n+1}$.
    I call \emph{$\lder$-sequence} a sequence $(a_n)_{n<m}$ such that $a_{n}\sim \lder(a_n)$ and $\val(a_n)\notin \val(\bE)$ for all $n<m$.
\end{definition}

\begin{remark}
    The name logarithmic scale is borrowed from \cite[Def.~4.2]{dries2002minimal}.
\end{remark}

\begin{remark}\label{rmk:maximal_essential_log_scales}
    Note that an approximate essential logarithmic scale $(y_n)_{n<m}$ is maximal if and only if either $m=\omega$ or for all $y_m$ such that $\partial y_m \sim \lder(y_{m-1}-c_{m-1})$ we have $\val(y_{m}-\bE)\subseteq \val(\bE)$.
\end{remark}

\begin{lemma}\label{lem:approximate_to_exact_log_scale}
    Suppose that $\lder(\bE_*^{\neq0})\subseteq \partial(\bE_*)$. Then for every approximate $\partial$-logarithmic scale $(y_n)_{n<m}$, there is a $\partial$-logarithmic scale $(y_n)_{n<m}$ such that $\partial z_n\sim \partial y_n$ for all $n<m$.
    \begin{proof}
        Suppose that $\partial y_n \sim (y_n-c_n) \partial y_{n+1}$ with $c_n \in \bE$ for all $n<m-1$.
        Choose any $z_0$ such that $\partial z_0 \sim \partial y_0$ and $z_0-c_0 \sim y_0-c_0$ (e.g.\ $z_0=y_0$).
        We are going to show that for $n<m-1$, given $z_n$ such that $\partial z_n \sim \partial y_n$ and $z_n-c_n \sim y_n-c_n$, we can find $z_{n+1}$ such that $\partial z_{n+1} = \lder(z_{n}-c_n)$ and $z_{n+1}-c_{n+1} \sim y_{n+1}-c_{n+1}$.
        By the hypothesis on $\bE_*$, we can find $\tilde{z}$ such that $\partial \tilde{z}= \lder(z_{n}-c_n)$. Note that since $\lder(z_n-c_n)\sim \lder (y_n-c_n) \sim \partial y_{n+1}$, we then have $\partial \tilde{z}-\partial y_{n+1}\prec \partial y_{n+1}$.
        Now note that by Lemma~\ref{lem:module-derivatives}(4) we have
        \[\partial (\tilde{z}-y_{n+1}) \in \co_* \partial y_{n+1}\subseteq \cO_*\partial ((y_{n+1}-c_{n+1}) \co_*),\]
        whence by $\partial$-convexity of $\co_* (y_{n+1}-c_{n+1})$ we have $\partial(\tilde{z} - y_{n+1}) \in \partial((y_{n+1}-c_{n+1})\co_*)$, so for some $z_{n+1} \in \tilde{z} + \bE$ we have $z_{n+1}-y_{n+1} \prec y_{n+1}-c_{n+1}$ and thus $z_{n+1}-c_{n+1}\sim y_{n+1}-c_{n+1}$.
    \end{proof}
\end{lemma}

\begin{lemma}\label{lem:weak_der_replacement}
    Suppose that $\rv(y)\notin \rv(\bE)$ and $\partial y \sim \partial z$, then there is $c \in \bE$ such that $y\sim z-c$.
    Analogously if $\val(y) \notin \val(\bE)$ and $\partial y \asymp \partial z$, then there is $c\in \bE$ such that $y \asymp z-c$. 
    \begin{proof}
        Note that since $\rv(y)\notin \rv(\bE)$, we have
        by $\lder$-convexity of $(1+\co_*)$, that $\lder(y) \succ \partial \co_*$, that is $\partial y \succ y \partial \co_*$.
        Thus by Lemma~\ref{lem:module-derivatives}(4) $\cO_*\partial(y \co_*)= \co_*\partial y$ (cf.\ also Remark~\ref{rmk:note_before_weak_lder-conv}).
        But then, since $\partial(y-z) \prec \partial y$, by $\partial$-convexity of $\co_* y$, we deduce $\partial(y-z) \in \partial(y\co_*)$ and thus for some constant $c\in \bE$ we have $y-z+c\prec y$, that is $y \sim z-c$. For the second clause note that since $\partial y \asymp \partial z$ and $\val(y) \notin \val(\bE)$, $z$ cannot be absorbed over $\bE$, so there is $c$ such that $\val(z-c) \notin \val(\bE)$, then Lemma~\ref{lem:valuation-monotonicity} ensures that it must be $z-c\asymp y$ (cf.\ also Remark~\ref{rmk:induced_order_bijection}).
    \end{proof}
\end{lemma}

\begin{corollary}\label{cor:maximal_essential}
    Suppose that $m>0$ and $(y_i)_{i<m}$ and $(z_i)_{i<m}$ are essential approximate logarithmic scales such that $\partial y_0 \sim \partial z_0$. Then $\partial y_i \sim \partial z_i$ for all $i<m$ and $(y_i)_{i<m}$ is maximal if and only if $(z_i)_{i<m}$ is maximal.
    \begin{proof}
        Note that if $\val(y-c) \notin \val(\bE)$ for some $c \in \bE$, then for all $\tilde{c}\in \bE$ with $\val(y-\tilde{c}) \notin \val(\bE)$ we must have $y-c\sim y-\tilde{c}$.

        Let $\partial y_{i} \sim (y_i-c_i) \partial y_{i+1}$ and $\partial z_i \sim (z_i-d_i) \partial z_{i+1}$ with $c_i, d_i\in \bE$ for all $0 \le i<m-1$,
        Since $(y_i)_{i<m}$ and $(z_i)_{i<m}$ are essential, we have (after an additional choice of $c_{m-1}, d_{m-1} \in \bE$), $\val(y_i-c_i), \val(z_i-d_i) \notin \val(\bE)$ for all $i<m$.
        
        Now suppose we have shown $\partial y_i \sim \partial z_i$ for some $i<m$, since $\val(y_i-c_i) \notin \val(\bE)$, by Lemma~\ref{lem:weak_der_replacement}, we can find $\tilde{d}\in \bE$ such that $z_i - \tilde{d} \sim y_i-c_i$. Note that since $\val(z_i-\tilde{d})\notin \val(\bE)$ and $\val(z_i-d_i) \notin \val(\bE)$ we must have $z_i - \tilde{d} \sim z_i-d_i$.
        
        So if $i<m-1$ we have $\partial z_{i+1} \sim \lder(z_{i+1}-d_{i+1})\sim \lder(y_{i+1}-c_{i+1})\sim \partial y_{i+1}$.

        Thus we have shown that $\partial y_i \sim \partial z_{i}$ for all $i<m$ and that furthermore $z_{m-1}-d_{m-1} \sim y_{m-1}-c_{m-1}$.

        Now suppose that $(y_i)_{i<m}$ is not maximal. Then in particular $m<\omega$ and there is $y_m$ such that $\partial y_m\sim \lder(y_{m-1}-c_{m-1})\sim \lder(z_{m-1}-c_{m-1})$ and furthermore $\val(y_m-\bE)\not\subseteq \val(\bE)$, but then setting $z_m=y_m$ we have that also $(z_i)_{i<m+1}$ is an essential approximate $\partial$-logarithmic scale, so $(z_i)_{i<m}$ was not maximal as well.
    \end{proof}
\end{corollary}

\begin{corollary}
    Given $x\in \bE_*$, 
    \begin{enumerate}
        \item all maximal essential approximate $\partial$-logarithmic scales $(y_i)_{i<m}$ with $\partial y_0 \sim x$ have the same length;
        \item if $\lder(\bE_*^{\neq0})\subseteq \partial(\bE_*)$ then such length is also the length of all maximal essential $\partial$-logarithmic scales $(y_i)_{i<m}$ with $\partial y_0 = x$.
    \end{enumerate}
    \begin{proof}
        (1) follows from Lemma~\ref{cor:maximal_essential}, and (2) follows from Lemma~\ref{lem:approximate_to_exact_log_scale}.
    \end{proof}
\end{corollary}

\begin{lemma}\label{lem:log_scale_to_essential}
    For $x\in \bE_*$ and $m < \omega$, the following are equivalent 
    \begin{enumerate}
        \item there is an approximate $\partial$-logarithmic scale $(y_i)_{i<m+1}$ with $x \sim \partial y_0$ and $\val(y_{m}-\bE)\subseteq \val(\bE)$;
        \item every \emph{essential} approximate $\partial$-logarithmic scale $\overline{y}$ with $x \sim \partial y_0$ has length at most $m$.
    \end{enumerate}
    The same equivalence holds for $\partial$-logarithmic scales with $\partial y_0 = x$.
    \begin{proof}
        Clearly if (2) holds, then (1) holds as well.
        Suppose (1) holds. If $m=0$, then $\val(y_0-\bE)\subseteq \val(\bE)$, but then by Lemma~\ref{lem:weak_der_replacement} for all $z_0$ with $\partial z_0 \sim \partial y_0 \sim x$ we must have $\val(z_0-\bE) \subseteq \val(\bE)$. Thus, there are no essential approximate $\partial$-logarithmic scales $(y_i)_{i<1}$ with $\partial y_0 \sim x$.

        Note that to prove (2), it suffices to show that if (1) holds with $m>0$, then there is a maximal essential approximate $\partial$-logarithmic scale $(z_i)_{i<m}$ with $x \sim \partial z_0$.
        
        Proceed by induction on $m$. 
        
        Let $(y_i)_{i<m+1}$ be a logarithmic scale with $\val(y_{m} -\bE)\subseteq \val(\bE)$. Let $(c_i)_{i<m}\in \bE^m$ be such that $\partial y_{i}\sim (y_i-c_i) \partial y_{i+1}$ for all $i<m$.

        If $\val(y_0-c_0) \notin \val(\bE)$, then we can conclude by the inductive hypothesis applied to $x=\partial y_1 \sim \lder(y_0-c_0)$.

        If instead $\val(y_0-c_0)\in \val(\bE)$, either $\val(y_0-\bE) \subseteq \val(\bE)$ and (2) is vacuously true as observed above, or there is $c \in \bE$ such that $\val(y_0-c) \notin \val(\bE)$. Then setting $d\coloneqq c-c_0\in \bE$ we have $y_0-c_0 \sim d$. Set $\tilde{y}_0 = d y_1$.
        Note that then $\partial y_0\sim \partial (dy_1)=\partial \tilde{y}_0$, and thus by Lemma~\ref{lem:weak_der_replacement}, $y_0-c\sim \tilde{y}_0-\tilde{c}_0$ for some $\tilde{c}_0\in \bE$. Thus setting $\tilde{y}_i = y_{i+1}$ for all $0<i<m-1$, we thus get that $(\tilde{y}_i)_{i<m-1}$ is an approximate $\partial$-logarithmic scale with $\val(y_{m-1}-\bE)\subseteq \val(\bE)$ and we can apply again the inductive hypothesis.
    \end{proof}
\end{lemma}

\begin{definition}\label{def:der-height}
    I define the \emph{$\partial$-logarithmic height} of $x$ (or just the \emph{height} of $x$ over $\partial$) to be $\height(x, \partial)\coloneqq m$ if $m$ is the length of a maximal essential approximate $\partial$-logarithmic scale with $\partial x \sim\partial y_0$, or equivalently if $1+m$ is the infimum of the lengths of finite approximate logarithmic scales $(y_i)_{i<1+n}$ with $\val(y_n-\bE)\subseteq \val(\bE)$ and $\partial x\sim \partial y_0$ (where it is understood that the infimum is $\omega$ when there are no such finite logarithmic scales). 
\end{definition}

\begin{remark}
    Note that the notion of logarithmic scale starting at $\partial x$ is invariant under rescaling the derivation, so $\height(x, \partial)= \height(x, y\partial)$ for all $y\neq 0$.
\end{remark}

The next three Lemmas and the last Proposition show that if $\partial\br(\partial)\prec 1$ (so given our setting $\partial$ is semiabsorbed), then essential approximate logarithmic scales correspond to $\lder$-sequences. This will be a fundamental ingredient of the main theorems.

\begin{lemma}\label{lem:aels_props}
	Suppose that $(a_n)_{n<1+m}$ is a $\lder$-sequence, then
    \begin{enumerate}
        \item $\lder^k(a_n) \sim a_{n+k}$ for all $n\le m-k$;
        \item if $m-k \ge 0$, then $(\lder^k(a_n))_{n<1+m-k}$ is a $\dagger$-sequence.
    \end{enumerate}
	\begin{proof}
        (1) By induction on $k$. If $k=0$ this is obvious. Assume the statement holds for $k$ and suppose that $n\le m-(k+1)$ then
        \[\lder^{k+1}(a_n)\in \lder(a_{n+k} \cdot (1+\co_*))\subseteq \lder(a_{n+k}) + \cO_* \partial \co_*\sim \lder(a_{n+k})\sim a_{n+k+1}\]
        because since $n+k<m$, $\val(a_{n+k})\notin \val(\bE)$, and thus $\lder(a_{n+k})\succ \cO_* \partial \cO_* \supseteq \cO_* \partial \co_*$.

        (2) From (1) we immediately deduce that $\lder^{k+1}(a_n)\sim a_{n+k+1} \sim \lder^{k}(a_{n+1})$. Furthermore for $n< m-k$ we have $\val(\lder^k(a_n))=\val(a_{n+k})\notin \val(\bE)$.
	\end{proof}
\end{lemma}

\begin{lemma}\label{lem:essential_log-scales_are_dagger_seq}
	Suppose that $\partial \br(\partial) \prec 1$ and let $m \le \omega$. Then:
    \begin{enumerate}
        \item $(y_n)_{n<m}$ is an essential approximate $\partial$-logarithmic scale if and only if the sequence $(\partial(y_n))_{n<m}$ is a $\lder$-sequence;
        \item for all $\lder$-sequences $(a_n)_{n<m}$, there is an essential approximate $\partial$-logarithmic sequence $(y_n)_{n<m}$ such that $a_n \sim \partial y_n$ for all $n<m$.
    \end{enumerate} 
	\begin{proof}
        (1) Let $\partial y_n \sim (y_n-c_n) \partial y_{n+1}$ for $0 \le n <m-1$ and note that since $\val(y_n-c_n)\notin \val(\bE)$, by Lemma~\ref{lem:der_replacement}, for all $n<m$ we have $y_n-c_n \sim \partial y_n/\lder(\partial y_n)$ and furthermore $\val(\partial y_n)\notin \val(\bE)$, whence $\lder(\partial y_n) \sim \lder(y_n-c_n) \sim \partial y_{n+1}$. Finally observe that by Corollary~\ref{cor:absorbed-der_sumup}, we have $\val(\partial y_n) \notin \val(\bE)$ for all $n<m$.
        
        Conversely argue by induction and suppose that if $(\partial (y_n))_{n<m}$ is a $\lder$-sequence, then there are $(c_n)_{n<m}$ in $\bE$ such that $\val(y_n-c_n)\notin \val(\bE)$ and for $n<m-1$, $\lder(y_n-c_n) \sim \partial y_{n-1}$. Assume that $\partial y_m \sim \lder(\partial y_{m-1})$ and that $\val(\partial y_m) \notin \val(\bE)$, then by Corollary~\ref{cor:absorbed-der_sumup} we can find $c_m \in \bE$ such that $\val(y_m-c_m)\notin \val(\bE)$. Furthermore since $\val(y_{m-1}-c_{m-1})\notin \val(\bE)$, by  Lemma~\ref{lem:der_replacement} it must be $\partial(y_{m-1})\sim (y_{m-1}-c_{m-1}) \lder(\partial y_{m-1})\sim (y_{m-1}-c_{m-1})\partial y_m$.
        
        (2) Suppose that $\val(\lder^k(a_0))\notin \val(\bE)$ for all $k<m$ and by inductive hypothesis we have built an essential approximate logarithmic scale $(y_{i})_{i<m}$ with $\partial y_i \sim \lder^i(a)$ for all $i<m$, then $\partial y_{m-1} \sim \lder^{m-1}(a)$, so $\val(\partial y_{m-1})\notin \val(\bE)$ and by Corollary~\ref{cor:absorbed-der_sumup} we can find $c_{m-1} \in \bE$ such that $\val(y_{m-1}-c_{m-1}) \notin \val(\bE)$, but then it suffices to pick $y_{m}$ such that $\lder(y_{m-1}-c_{m-1})\sim \partial y_{m}$ which we can do by Corollary~\ref{cor:asymptotic_integration1}.
	\end{proof}
\end{lemma}

\begin{proposition}\label{prop:height_over_absorbed}
    If $\partial\br(\partial) \prec 1$, then
    \[\height(x,\partial)=\inf\{k\in \omega: \val(\lder^k(\partial x))\in \val(\bE)\}\in \omega+1.\]
    \begin{proof}
        Set $a=\partial x$. Clearly the height of $x$ is $0$ if and only if $\val(a)\in \val(\bE)$.
        
        It suffices to show that for all $m\ge 0$, there is an essential approximate logarithmic scale $(y_k)_{k < m}$ with $a \sim \partial y_0$ (in the case $m>0$) if and only if $\val(\lder^k(a))\notin \val(\bE)$ for all $k<m$.

        If there is such a sequence, then $(\partial y_i)_{i<m}$ is a $\lder$-sequence by Lemma~\ref{lem:essential_log-scales_are_dagger_seq}, so $\lder^k(\partial x)\sim \lder^k(\partial y_0) \sim \partial y_k$ and we have $\val(\lder^k(a))\notin \val(\bE)$ for all $k<m$.

        Conversely if $\val(\lder^k(a))\notin \val(\bE)$ for all $k<m$, then by Lemma~\ref{lem:essential_log-scales_are_dagger_seq}(2), there is an essential approximate $\partial$-logarithmic scale $(y_{k})_{k<m}$ with $\partial y_{0}\sim a=\partial x$.
    \end{proof}
\end{proposition}

\subsection{Back to germs at absorbed types}\label{ssec:first_applications1}
In this subsection we spell out what Proposition~\ref{prop:height_over_absorbed} has to say about models of an exponential o-minimal theory $T$. Throughout the subsection we fix an elementary extension $\bE \prec \bE_*$ of models of $T$ and a convex subring $\cO_*$ of $\bE_*$ such that $\exp(\cO_*)\subseteq \cO_*$.
Recall the following definition form \cite{freni2024t}.

\begin{definition}
    A \emph{$\bE$-definable gne of height $m$} is a composition $g:=h_{0} \circ h_1 \cdots \circ h_{m-1} \circ \tau$, where $h_i(t)=c_i + \sigma_i\exp(t)$ for $c_i \in \bE$, $\sigma_i \in \{\pm1\}$ and $\tau(t)=c_{m}+t$ is a translation by some $c_{m} \in \bE$.
    Let $y \in \bE_*$, then a gne $g$ is
    \begin{enumerate}
        \item \emph{$\cO_*$-normal} at $y$ if either $g(t)=c_0+t$ when $c_0 \succ y$ or $c_0=0$ or if $g(t)=c_0+\sigma_0\exp(g_1(t))$ where $c_0 \succ \exp(g_1(y))$ or $c_0=0$ and $g_1$ is normal at $y$; 
        \item \emph{$\cO_*$-essential} at $y$ if $g(t)=c_0+t$ or if $g(t)=c_0+ \sigma_0\exp(g_1(t))$ where $\val(\exp(g_1(t))\notin \val(\bE)$ and $g_1$ is essential at $y$.
    \end{enumerate}
\end{definition}

\begin{remark}\label{rmk:gne_to_essential}
    Given $x \in \bE_* \setminus \bE$ with $\val(x-\bE)\subseteq \val(\bE)$ and an $\bE$-definable gne $g$, by Lemma~\ref{lem:log_scale_to_essential} and Lemma~\ref{lem:approximate_to_exact_log_scale} there is always an essential gne $\tilde{g}$ of lower height such that $\val(\tilde{g}^{-1}(g(x))-\bE)\subseteq \val(\bE)$.
    When $\cO_*\in \{\cO_*^+, \cO_*^-\}$ for some convex subring $\cO$ of $\bE$, then this can be improved to a normal essential gne (note that given our assumption on $\cO_*$ it must be $\exp(\cO)\subseteq \cO$).
\end{remark}

\begin{lemma}\label{lem:gne-normalize}
    Let $x\in \bE_*\setminus \bE$ be $\cO_*$-absorbed and $\cO_*\in \{\cO_*^-, \cO_*^+\}$ for some convex $\cO \subseteq \bE$ and suppose that $g$ is an $\bE$-definable gne of height $m$, then there is a $\cO_*$-absorbed $z \in \bE_*$ and an $\bE$-definable gne $\tilde{g}$ of height $\le m$, normal and essential at $z$, such that $g(x) = \tilde{g}(z)$.
    \begin{proof}
        Note that by Remark~\ref{rmk:gne_to_essential}, we can assume that $g$ is essential. Note that by the hypothesis on $\cO_*$, we have that whether $x$ is $\cO_*$-absorbed or not only depends on $\tp(x/\bE)$. So in fact it suffices to find a gne $\tilde{g}$ normal and essential at some absorbed $\tilde{z}$ such that $\tilde{g}(\tilde{z}) \equiv_\bE g(x)$: in fact we can then consider $z=\tilde{g}^{-1}g(x) \equiv_\bE \tilde{z}$ which will be $\cO_*$-absorbed.
        
        To find such an $\bE$-definable gne $\tilde{g}$, proceed by induction on $m$. If $m=0$ it suffices to take $\tilde{g}=g$. If $m>0$ and $g=c_0+\sigma_0\exp(g_1)$, find $\tilde{g}_1$ and $\tilde{z}$ normal and essential such that $\tilde{g}_1(\tilde{z})\equiv g_1(x)$. Note that $\val(\exp(g_1(x))\notin \val(\bE)$ by essentiality of $g$ at $x$ and that $\tilde{g}_1(\tilde{z})\equiv_\bE g_1(x)$. If $c_0 \succ \exp(g_1(x))$, then we can set $\tilde{g}=c_0+\sigma_0\exp(\tilde{g}_1)$, if instead $c_0\prec \exp(g_1(x))$, then we set $\tilde{g}=\sigma_0 \exp(\tilde{g}_1)$ as then $\tilde{g}(\tilde{z})\equiv_{\bE} \tilde{g}(\tilde{z})-c_0$.
    \end{proof}
\end{lemma}    

Now suppose that $\bE_*\coloneqq \bE\langle x \rangle$ is a $1$-$\dcl_T$-dimensional elementary-extension of $\bE$, and $\partial=\partial_x$ is the only $\bE$-linear $T$-derivation with $\partial_x=1$. Then $\partial_x$ is a logarithmically $\cO_*$-valuation-convex derivation by Proposition~\ref{prop:der_order-convexity} and Corollary~\ref{cor:order-convex_implies_val-convex}, and our assumption on $\cO_*$.
Thus we have that $\partial (\br_{\cO_*}(\partial))\prec 1$ if and only if $\val_{\cO_*}(x-\bE)\subseteq \val_{\cO_*}(\bE)$.

\begin{corollary}\label{cor:main-cor}
    The following are equivalent:
    \begin{enumerate}
        \item every element of $\bE \langle x\rangle\setminus \bE$ has finite height over $\partial$;
        \item $(\bE\langle x\rangle, \cO_*)\succ (\bE, \cO_*\cap \bE)$ has the gne-property;
    \end{enumerate}
    If furthermore $\val_{\cO_*}(x-\bE)\subseteq \val_{\cO_*}(\bE)$, then they are equivalent to 
    \begin{enumerate}[resume]
        \item for every $y \in \bE \langle x \rangle\setminus \bE$, there is $k\in\bN$ such that $\val(\lder^k(\partial y)) \in \val(\bE)$.
    \end{enumerate}
    \begin{proof}
        Suppose (1), then by Lemma~\ref{lem:approximate_to_exact_log_scale}, we for all $y \in \bE\langle x \rangle$ find an $\bE$-definable gne $g$ and a $z \in \bE \langle z\rangle$ with $g(z)=y$ and $\val(z-\bE)\subseteq \val(\bE)$, thus we get (2).
        
        Conversely, suppose (2), and let $y \in \bE \langle x \rangle$, then by the gne-property, there is a logarithmic scale of finite length starting at $y$ and ending at an absorbed element, so by Lemma~\ref{lem:log_scale_to_essential} we can find a finite maximal essential logarithmic scale and thus $y$ has finite height.
        Finally, the equivalence of (1) and (3) is Proposition~\ref{prop:height_over_absorbed}.
    \end{proof}
\end{corollary}

We also point out the following generalization of \cite[Prop.~3.2]{freni2024t}. Here we drop the assumption that $\exp(\cO)\subseteq \cO$. Recall that $\cO_*^+\coloneqq\{t\in \bE_*: |t|<\bE^{>\cO}\}$.

\begin{corollary}\label{cor:wim_special_ortho}
    Let $\bE \prec \bE_* \models T$ and let $x \in \bE_*\setminus \bE$. Suppose that $\cO$ is any convex subring of $\bE$ and that $x$ is $\cO$-wim over $\bE$. Then $\res_{\cO_*^+}(\bE)=\res_{\cO_*^+}(\bE\langle x \rangle)$. In particular $\cO_*^+\cap\bE \langle x\rangle=\cO_*^-\cap \bE \langle x\rangle$ and $\cO$ extends uniquely to $\bE \langle x\rangle$.
    \begin{proof}
        This is essentially already in Example~\ref{main:examples}(2). Let $\cO_x^+\coloneqq \cO_*^+\cap \bE\langle x \rangle$ and  Consider the valued differential field $(\bE \langle x \rangle, \cO_x^+, \partial_x)$. 
        By Proposition~\ref{prop:der_order-convexity} and Corollary~\ref{cor:order-convex_implies_val-convex}(1), $\partial_x$ is q.log.\ $\cO_x$-valuation-convex. But since it contains a $\cO$-wim element over $\bE$, by Corollary~\ref{cor:der-types_and_absorbed-elements} $\partial_x$ has type (W) and $\res_{\cO_*^+}(\bE)=\res_{\cO_*^+}(\bE\langle x \rangle)$.
    \end{proof}
\end{corollary}

\subsection{Tressl signature-alternative in simply exponential o-minimal fields}\label{ssec:first_applications2}
In this subsection we show how Proposition~\ref{prop:height_over_absorbed}, can be used to partially answer a problem posed in \cite{tressl2005model}. Our setting will be the one of an elementary extension $\bE \prec \bE_*\models T$ of models of an \emph{exponential} o-minimal theory $T$.

In \cite[Def.~3.16]{tressl2005model}, Tressl considers the following condition on a unary type $p=\tp(x/\bE)$ over $\bE$: $p$ satisfies the \emph{signature alternative} if it is either non-symmetric, or
\begin{enumerate}
    \item $\Br(x/\bE)(\bE)$ is cofinal in $\Br(x/\bE)(\bE\langle x\rangle)$, that is, $\bE\langle x \rangle$ does not realize the cut above $\Br(x/\bE)(\bE)$;
    \item if $y$ is the realization of such cut, i.e.\ if $y \in \Br(y/\bE)(\bE_*)$ and $y>\Br(x/\bE)(\bE)$, then $\Br(\log y/\bE)$ is cofinal in $\Br(\log y/\bE)(\bE \langle y \rangle)$.
\end{enumerate}

Note that if (1) above is satisfied by all types over $\bE$, then all types over $\bE$ satisfy (2) as well. We point out the following immediate Corollary of Proposition~\ref{prop:breadth-ortho}.

\begin{corollary}\label{cor:finite_hight_and_br_ortho}
    If $x \in \bE_* \setminus \bE$ is symmetric, and all $y \in \bE \langle x\rangle$ have finite height in $(\bE\langle x \rangle, \cO_*^-, \partial_x)$ with $\cO_*^-=\CH(\bZ)$, then $\Br(x/\bE)(\bE)$ is cofinal in $\Br(x/\bE)(\bE\langle x\rangle)$.
    \begin{proof}
        By Proposition~\ref{prop:breadth-ortho}(1) we have for all $y \in \bE \langle x \rangle$ that if $y \in \Br(x/\bE)(\bE\langle x\rangle)\setminus \CH(\Br(x/\bE)(\bE))$, then
        \[ \lder_x^k(y)\in \Br(x/\bE)(\bE\langle x\rangle)\setminus \CH(\Br(x/\bE)(\bE)) \quad \text{for all}\; k \in \bN.\]
        But then this would imply that $\val_{\cO_*^-}(\lder^k y) \notin \val_{\cO_*^-}(\bE)$ and $y$ would have infinite height.
    \end{proof}
\end{corollary}

Recall that an o-minimal theory is \emph{simply exponential} if it is the theory of a power-bounded o-minimal structure defining restricted exponentiation expanded by a compatible exponential (see \cite{freni2024t}).

\begin{corollary}\label{cor:simply_exp_signature_alt}
    If $T$ is simply exponential admitting an Archimedean prime model, then for all $\bE\models T$, unary types over $\bE$ satisfy Tressl's signature alternative.
    \begin{proof}
        By \cite[Prop.~5.27]{freni2024t} if $T$ is simply exponential and $(\bE, \cO)\models T_\convex$, then for every $x$ $\cO$-wim over $\bE$, all elements $y \in \bE \langle x \rangle\setminus \bE$ have finite height in $(\bE\langle x \rangle, \cO_*, \partial_x)$. 
        Now let $\bE_*\succ \bE\models T$, since $T$ has an Archimedean prime model, $\cO=\CH_\bE(\bZ)$ is $T$-convex.
        Now let $x\in \bE_* \setminus \bE$ be symmetric over $\bE$, then by Lemma~\ref{lem:0-sign_vs_wim} it is either $\cO$-wim, or a $\cO$-limit. The case when $x$ is $\cO$-wim is already covered. When $x$ is a $\cO$-limit, $\Br(x/\bE)(\bE)=a\co$, so if $\Br(x/\bE)$ was not cofinal in $\Br(x/\bE)(\bE \langle x\rangle)$, $\bE \langle x \rangle$ would contain an element $\co<y<\bE^{>\co}$, but we know this does not happen by \cite[Thm.~A]{freni2024t} or by \cite[Thm.~7.1]{tressl2005model}.
    \end{proof}
\end{corollary}

We can also say something more general about symmetric cuts that are not $\cO$-limits for any convex $\cO \subseteq \bE$.

\begin{corollary}
    Let $\bE \prec \bE_* \models T$ and let $x \in \bE_*\setminus \bE$. Suppose that $x \in \bE_* \setminus \bE$ is symmetric but is not a $\Vr(x/\bE)(\bE)$-limit. Then $\Br(x/\bE)(\bE)$ is cofinal in $\Br(x/\bE)(\bE\langle x\rangle)$
    \begin{proof}
        Since $x$ is $\Vr(x/\bE)(\bE)$-symmetric by definition and not a $\Vr(x/\bE)(\bE)$-limit, it must be $\Vr(x/\bE)(\bE)$-wim.
        Set $\cO_x\coloneqq \Vr(x/\bE)(\bE\langle x\rangle)$ and let $\partial_x$ be the $\bE$-linear $T$-derivation at $x$ on $\bE \langle x\rangle$.
        By Example~\ref{main:examples}(2), $\partial_x$ is $\cO_x$-Liouville-convex. Furthermore it is $\cO_x$-absorbed by Corollary~\ref{cor:abs_basic_equivalence} and
        \[\br_{\cO_x}(\partial_x)=\br_{\cO_x}(x/\bE)=\Br(x/\bE)(\bE \langle x \rangle)=:B_x^+.\]

        But now if $f$ was a unary $\bE$-definable function such that $f(x) \in B_x^+\setminus B_x^-$, where $B_x^-=\CH^\le_{\bE\langle x\rangle}(\Br(x/\bE)(\bE))$, then by Proposition~\ref{prop:breadth-ortho}, $1/f'(x) \in \Vr(x/\bE)(\bE\langle x \rangle)=\cO_x$, but on the other hand it should be $1/|f'(x)| > \cO_x$, because $\partial_x$ is $\cO_x$-absorbed, contradiction.
    \end{proof}
\end{corollary}

\section{Transserial o-minimal theories} \label{sec:transserial}

This section is devoted to the proofs of the main theorems. We will thus characterize those theories $T$ in which for all models $(\bE, \cO) \models T_\convex$ and all $1$-dimensional wim-constructible extensions $\bE \langle x \rangle\succ \bE$, all elements of $\bE\langle x \rangle \setminus \bE$ have a finite logarithmic height over $\bE$, and prove some properties of them.

\subsection{Transseriality and exponential-boundedness}
In this subsection we introduce the definition of transseriality. 

\begin{notation}
    Let $f: \bE^{n+1} \to \bE$ be an $\bE$-definable function and $m$ an integer, recall from the introduction that we denote by $G_m^f : \bE^{n+1} \to \bE^{\ge 0} \cup \{\infty\}$, the $\bE$-definable function
    \[G_m^f(\overline{x}, t) =  \min\big\{|\lder_{n}^k f(\overline{x}, t)|: 0 \le k \le m \big\}\] 
    where $\dagger_{n}^k$ is the $k$-fold iterate of the logarithmic derivative in the last variable, and we set $\dagger_{n} f (\overline{x}, t)= \infty$ when $\dagger_{n} f (\overline{x}, -)$ is not defined at $t$.
    We also set $\tilde{G}_m^{f}=\max \{G_m^f, G_m^{\tilde{f}}\}$ where $\tilde{f}(\overline{x}, t)=f(\overline{x}, -t)$.

    If $y \in \bE_*\succ \bE\models T$ and $y \notin \bE$, we will write $\partial_y: \bE \langle y \rangle \to \bE \langle y \rangle$ for the only $\bE$-linear $T$-derivation such that $\partial_y(y)=1$.
\end{notation}

We are going to be interested in the following conditions on a theory $T$.

\begin{definition}\label{def:transserial}
    Let $\bE$ be an o-minimal expansion of an ordered field, $T=\Th(\bE)$ and $(\bE, \cO) \models T_\convex$. We say that $T$ is
    \begin{enumerate}
        \item \emph{weakly transserial}, if for each $T$-definable $(n+1)$-ary function $f$, there is a natural number $m$ such that
	    \begin{equation}\label{eqn:wtransserial_f_m}\tag{$\mathrm{WTS}^{f}_m$}
	        (\bE, \cO) \models \forall \overline{x} \in \bE^n, \; \exists r \in \cO,\; \forall t \in \cO\; \big(t>r \rightarrow G_m^f(\overline{x}, t)\in \cO\big);
        \end{equation}
        \item \emph{transserial}, if for each $T$-definable $(n+1)$-ary function $f$, there is a natural number $m$ such that
	    \begin{equation}\label{eqn:transserial_f_m}\tag{$\mathrm{TS}^{f}_m$}
	        (\bE, \cO) \models \forall \overline{x}\in \bE^n,\; \exists r,s \in \cO,\; \forall t \in \cO\; \big(t>r \rightarrow G_m^f(\overline{x}, t)\le s\big).
        \end{equation}
    \end{enumerate}
\end{definition}

\begin{remark}
	If $T$ is transserial (resp.\ weakly ---), then (\ref{eqn:transserial_f_m}) (resp.\ (\ref{eqn:wtransserial_f_m})) in Definition~\ref{def:transserial} holds also if $f$ is just $\bE$-definable. Also note that (\ref{eqn:transserial_f_m}) implies, for any $s \in \cO^{>\co}$
    \begin{equation}\label{eqn:transserial_f_m_s}\tag{$\mathrm{TS}^{f}_{m+1,s}$}
	        (\bE, \cO)\models \forall \overline{x}\in \bE^n,\; \exists r \in \cO,\; \forall t \in \cO\; \big(t>r \rightarrow G_{m+1}^f(\overline{x}, t)< s\big).
    \end{equation}
    So in the definition of transserial we can also just require that for every $f$, there is $m$ such that ($\mathrm{TS}^f_{m,1}$) holds. For convenience we will mostly use this last formulation, which is also the one given in the introduction.
\end{remark}

\begin{lemma}\label{lem:unary_functions}
    Let $T$ be an exponential o-minimal theory.
    \begin{enumerate}
        \item $T$ is weakly transserial if and only if for all $(\bE, \cO) \models T_\convex$ and every unary $\bE$-definable function $f$, there is a natural number $m$ such that
        \begin{equation}\tag{$\mathrm{WTS}^{f}_m$}
            (\bE, \cO) \models \exists r \in \cO, \forall t\in \cO, \big(t>r \rightarrow G_m^f(t) \in \cO\big).
        \end{equation}
        \item $T$ is transserial if and only if for all $(\bE, \cO) \models T_\convex$ and every unary $\bE$-definable function $f$, there is a natural number $m$ such that
        \begin{equation}\tag{$\mathrm{TS}^{f}_{m,1}$}
            (\bE, \cO) \models \exists r \in \cO, \forall t\in \cO, \big(t>r \rightarrow G_m^f(t) \le 1\big).
        \end{equation}
    \end{enumerate}
    \begin{proof}
        We prove (2), the proof of (1) is similar.
        $\Rightarrow$ is obvious.
        The converse is a compactness argument. Suppose that for some $T$-definable $n$-ary $f: \bE^{n+1} \to \bE$ we have for all $m \in \bN$
        \[T_\convex \models \exists \overline{x},\, \forall r\in \cO, \exists t\in \cO, (t > r \land G_m^f(\overline{x}, t)>1)\]
        which by the weak o-minimality of $T_\convex$ is equivalent to       \begin{equation}\label{eqn:strong_negation}
            T_\convex \models \exists \overline{x},\, \exists r\in \cO, \forall t\in \cO,\, (t > r \rightarrow G_m^f(\overline{x}, t)>1).
        \end{equation}
        Notice that since 
        \[G_{m+1}^f(\overline{x}, t)=\min\{G_{m}^f(\overline{x}, t), \lder_n (G_m^f)(\overline{x},t)\} \le G_{m}^f(\overline{x}, t),\]
        letting 
        \[\varphi_m(\overline{x}):= \big(\exists r\in \cO, \forall t\in \cO,\, (t > r \rightarrow G_m^f(\overline{x}, t)>1)\big),\]
        we have $\varphi_{m+1}(\overline{x}) \rightarrow \varphi_{m}(\overline{x})$ hence the validity of \ref{eqn:strong_negation} for all $m$ is equivalent to the consistency of the type $p(\overline{x})=\{\varphi_m(\overline{x}): m \in \bN\}$. If $(\bE, \cO)$ is a model of $T_\convex$ realizing such type and $\overline{c} \in \bE^n$ is a realization of $p$, we get that $g(t)=f(\overline{c}, t)$ is such that for all $m$, $(\bE, \cO)\models \exists r \in \cO,\, \forall t \in \cO,\, (t>r \rightarrow G_m^g(t)>1)$.
    \end{proof}
\end{lemma}

\begin{remark}\label{rmk:special_cut_reformulation}
    Let $(\bE, \cO)\models T_\convex$, and $b \in \bE_*\succ \bE$ be such that $\cO<b<\bE^{>\cO}$. Set, as usual $\cO_b^+\coloneqq \{t \in \bE \langle b \rangle: |t|<\bE^{>\cO}\}$, and $\cO_b^-\coloneqq \CH_{\bE \langle b \rangle}^{\le}(\cO)$ and recall that $(\bE\langle b \rangle, \cO_b^+)$ and $(\bE \langle b \rangle, \cO_b^-)$ are both models of $T_\convex$. 
    
    If $f:\bE \to \bE$ is a unary $\bE$-definable function then (\ref{eqn:transserial_f_m}) is equivalent to $G_m^f(b) \in \cO_b^-$, whereas (\ref{eqn:wtransserial_f_m}) is equivalent to $G_m^f(b) \in \cO_b^+$.
\end{remark}

\begin{corollary}\label{cor:wtransserial_gne}
    $T$ is weakly transserial if and only if for all $(\bE, \cO) \models T_\convex$ and $b$ as in Remark~\ref{rmk:special_cut_reformulation}, all elements of $(\bE\langle b \rangle, \cO_b^+, \partial_b)$ have finite $\partial_b$-logarithmic height.
    \begin{proof}
        By Lemma~\ref{lem:unary_functions} and Remark~\ref{rmk:special_cut_reformulation}, $T$ is weakly transserial if and only if for all $(\bE, \cO) \prec (\bE \langle b \rangle, \co_b^+)\models T_\convex$ with $\cO<b<\bE^{>\cO}$ and $b \in \cO_b$, and for all unary $\bE$-definable functions $f$, there is $m$ such that $G_m^f(b) \in \cO_b^+$. But now by Example~\ref{main:examples}, $(\bE \langle b \rangle, \cO_b, \partial_b)$ is a valued field and $\partial_b$ is a $\cO_b$-Liouville-convex derivation. Furthermore by Corollary~\ref{cor:absorbed-der_sumup}, $\partial_b$ is absorbed with $\br_{\cO_b}(\partial_b)=\co_b$. So the minimum $m$ such that $G_{m}^f(b) \in \cO_b^+$ is in fact the $\partial_b$-logarithmic height of $\exp(f(b))$ and the statement follows.
    \end{proof}
\end{corollary}

\begin{lemma}\label{lem:transserial_vs_Wtransserial}
    $T$ is transserial if and only if it is weakly transserial and exponentially bounded.
    \begin{proof}
        Clearly if $T$ is transserial, then it is weakly transserial. Furthermore it is exponentially bounded because if $f$ is a unary $T$-definable function, then (\ref{eqn:transserial_f_m}) implies that for some $m$, $\lder^m(f)(t)$ is bounded at infinity, so either $\lder^m(f)(t)$ or $\lder^{m+1}(f)(t)$ is infinitesimal at infinity.

        Conversely, consider an $\bE$-definable unary function $f: \bE\langle b\rangle \to \bE\langle b \rangle$, where $\cO<b<\bE^{>\cO}$. From weak transseriality we get, by Remark~\ref{rmk:special_cut_reformulation}, $\lder^m(f)(b)\in \cO_b^+$ for some $m$. If $(\dagger^{m} f) (b) \in \cO_b^-$ we are done.

        Otherwise note that by \cite[Lem.~2.23]{freni2024t}, we can write $(\dagger^{m} f) (b)=g(b)+\varepsilon(b)$ where $g$ is $\bK$-definable for some elementary residue section $\bK$ of $\bE$, $|g(b)|\ge 1$, and $\varepsilon(b) \in \co_b$.
        Now $\lder^{j+1}(g+\varepsilon) = \lder^{j+1} g + \varepsilon_{j+1}$ where $\varepsilon_1=\varepsilon$ and $\varepsilon_{j+1}:=\lder(1+\varepsilon_j/\lder^j g)$ and by \cite[Lem.~2.24]{freni2024t} we see that if $|\lder^i g(b)|\notin \co_b$ for all $i<j+1$, then $\varepsilon_{i}(b) \in \co_b$ for all $i\le j+1$. Since $T$ is exponentially bounded, there is $k\in \bN$ such that $(\lder^{k+1}g)(b)<1$, and by picking the smallest such $k$, we also have $(\lder^ig)(b) \notin \co_b$ for all $i<k+1$ so 
        \[(\lder^{m+k+1}f)(b) = (\lder^{k+1}(g+\varepsilon)) = (\lder^{k+1}g + \varepsilon_{k+1})(b)\in (\lder^{k+1}g)(b)+\co_b<1.\]
        Thus we can conclude by Remark~\ref{rmk:special_cut_reformulation}.
    \end{proof}
\end{lemma}

The following Lemma and the subsequent Remark won't be used in the paper.

\begin{lemma}
    If $T$ is weakly transserial, then for all $n \in \bN$ and all $T$-definable $(n+1)$-ary functions $f$, there is $m\in \bN$ and $N \in \bN$ such that for every $\overline{x}\in \bE$, $\res_{\cO}\{t \in \cO: G_m^f(\overline{x},t)\notin \cO\}$ has less than $N$ elements.
    \begin{proof}
        By a compactness argument similar to the one of Lemma~\ref{lem:unary_functions}, it suffices to show it for unary parameter-definable functions.
        Consider the setting of Remark~\ref{rmk:special_cut_reformulation}.
        Note that since $\height(y,\partial_b)=\height(y, \partial_{1/b})=\height(y, -b^2\partial_b)$, it follows from Corollary~\ref{cor:wtransserial_gne}, that there is $m$ such that $G_m^f(1/b) \in \cO_b$, furthermore the $m$ can be chosen uniform across a family of definable functions $f_x$. So considering the family $f_r(t)=f(r+t)$, we can in particular find $m$ such that $G_m^{f_r}(b) \in \cO_b$ for all $r \in \bE$. But then this implies in particular that for all $r \in \cO$ $G_m^f(r+1/t)\in \cO$ for all $t \in (r_-, r_+)\setminus (r+\co)$ for some $r_-<r+\co<r_+$. So by the weak o-minimality of $T_\convex$ we get the thesis.  
    \end{proof}
\end{lemma}

\begin{remark}\label{rmk:transserial_for_T}
	Since $T_\convex = \Th(\bK \langle d \rangle, \CH(\bK))$ for $d>\bK\models T$, (\ref{eqn:wtransserial_f_m}) is in fact equivalent (taking $\bK$ to be the prime model of $T$) to the requirement that for all $T$-definable $f$ there are $m \in \bN$ and $N \in \bN$ such that for every $T$-definable curve $\overline{\gamma}: [0, \infty) \to \bK^n$, for all but finitely many $t$s $H^{f,\gamma}_m(t)\coloneqq\lim_{s \to \infty}G_m^f(\overline{\gamma}(s), t)<\infty$.
    Similarly (\ref{eqn:transserial_f_m}) is equivalent to the strengthening of this asserting that furthermore $H^{f,\gamma}_m(t)$ is also bounded at $\pm\infty$. 
\end{remark}

\subsection{Characterizations of transseriality} In this subsection we will prove Theorem~\ref{introthm:transserial_char}.

\begin{lemma}\label{lem:generic_G_bound}
	Suppose that $T$ is transserial and $(\bE, \cO) \models T_\convex$. Let $f: \bE \to \bE$ be $\bE$-definable, then there is $m\in \bN$ such that for every $a \in \bE$ there is a finite subset $F_a\subseteq \bE$, such that 
	\[\forall t \in \bE \setminus (F_a+a\cO), \; G_m^f(t) \prec 1/a \]
	\begin{proof}
        Consider $h(a,x, t)=af(x+at)$ and let $m$ be such that for all $a,x \in \bE$, $G_m^h(a,x,t)<1$ for all large enough $t$ in $\cO$.
        Also note that 
        \[G_m^h(a,x,t)=G_m^h(a, 0, x/a+t)=aG_m^h(1, 0, x+at).\]
        Then for all $a\in \bE^{>0}$ and $x \in \bE$, we can find $c_+ \in \bE^{>\cO}$ such that 
        \begin{equation}\label{eqn1:lem:generic_G_bound}
            G_{m+1}^h(a, x, t) =G_{m+1}^h(a, 0, x/a+t) \prec 1 \quad \text{for all $t \in [0, c_+)\setminus \cO$.}
        \end{equation}
        Now for all $a \in \bE^{\neq 0}$ the $(\bE, \cO)$-definable set
		\[D_a:=\{t \in \bE: G_{m+1}^h(a, 0, t) \succeq 1\}\]
		must be a finite union of convex subsets by the weak o-minimality of $(\bE, \cO)$, but cannot contain any interval $(c_0, c_1)$ with $c_1>\cO+c_0$ because otherwise for all $t \in [0, c_1-c_0) \setminus \cO$ we would have $G_{m+1}^h(a, 0, c_0+t)\succeq 1$, contradicting \ref{eqn1:lem:generic_G_bound}.
        It follows that there is a finite set $\tilde{F}_a$ such that $D_a \subseteq \tilde{F}_a+ \cO$. But then since $G_{m+1}^h(a, 0, t)=aG_{m+1}^h(a, 0, at)$ we have 
        \[\forall t \in \bE \setminus a D_a,\, aG_{m+1}^h(a, 0, at)\prec 1,\]
        whence the thesis.
	\end{proof}
\end{lemma}

\begin{corollary}\label{cor:generic_G_bound}
	Suppose that $T$ is transserial $(\bE, \cO) \models T_\convex$ and $M\subseteq \bE$ is a $\cO$-submodule. Let $f: \bE \to \bE$ be $\bE$-definable, then there is $m\in \bN$ and a finite subset $F_M \subseteq \bE$, such that
	\[ \forall t \in \bE \setminus (F_M +M), \; 1/G_m^f(t) \notin M\]
	\begin{proof}
		Let $(\bE_1, \cO_1)\succ(\bE, \cO)$ contain a $z$ such that $M=z \cO_1 \cap \bE$  for some $z \in \bE_1^{>0}$. There is a finite $F_z \subseteq \bE_1$ such that $G_m^f(t) \prec 1/z$ for all $t \in \bE_1 \setminus (F_z+z\cO)$. Picking $F_M$ finite such that $F_M+ M \supseteq (F_z+z\cO)\cap \bE$, we get that for all $t \in \bE \setminus F_M+M$, $1/G_m^f(t) \notin M$.
	\end{proof}
\end{corollary}

\begin{proposition}\label{prop:transserial_sOabs}
	Suppose that $T$ is transserial, $(\bE, \cO) \models T_\convex$ and $x \in \bE_* \succ \bE$ is strongly $\cO$-absorbed over $\bE$. Then all $y \in \bE \langle x \rangle$ have finite $\partial_x$-height in $(\bE\langle x \rangle, \cO_x, \partial_x)$.
	\begin{proof}
        Write $y=f(x)$ for an $\bE$-definable function $f$.
		Let $h = f^\dagger$, and $M=\big(\Br(x/\bE)(\bE\langle x \rangle):\cO\big)\cap \bE$. By Corollary~\ref{cor:generic_G_bound}, we can find $k \in \bN$  such that $G_k^h(t) \notin M$ for all $t \in \bE \setminus (F_M+M)$. Observe that since $x$ is strongly $\cO$-absorbed, $x\notin \bE +\big(\Br(x/\bE)(\bE\langle x \rangle):\cO\big)$, so there are $x_-, x_+ \in \bE$ with $x_-<x<x_+$ and $[x_-, x_+] \cap F_M + M =\emptyset$. It follows that for some $c\in \bE^{>M}$, $1/G_k^h(t) \ge c >M$ for all $t \in [x_-, x_+]$ and thus $1/G_k^h(x)>(\Br(x/\bE)(\bE \langle x \rangle):\cO)$, whence there is $0 \le m<k$ such that $1/\lder^{m+1}f(x) \notin (\Br(x/\bE)(\bE \langle x \rangle):\cO)$ and we can conclude by Proposition~\ref{prop:techincal}.
	\end{proof}
\end{proposition}

\begin{proposition}\label{prop:non-transserial_sOabs}
    Suppose that $T$ is exponential but not transserial, then there is a model $(\bE, \cO)\models T_\convex$, an $\bE$-definable function $f: \bE \to \bE$, and $x,y \in \bE_*\setminus \bE$ with $x$ $\cO$-weakly immediate, and $y$ weak $\cO$-limit, such that for all $m \in \bN$, $\val(\dagger^m f(x)) \notin \val \bE$ and $\val(\dagger^mf(y))\notin \val(\bE)$.
    \begin{proof}
        Since $T$ is not transserial, there is $(\bE, \cO)\models T_\convex$ and an $\bE$-definable unary function $h$ such that $G_m^h(t)\succeq 1$ for all $m \in \bN$ and all large enough $t \in \cO$. The function $f=\exp(h)$, will then satisfy $G_m^{\partial f}(t) \succeq 1$ for all $m \in \bN$ and all large enough $t \in \cO$.
        Now, up to moving to an $\aleph_1$-saturated extension, we can assume that $\lder^m \partial f \succeq 1$ on an interval $(c_-, c_+)$ with $c_- \in \cO$ and $c_+ > \cO$. But then any $\cO$-wim $x \in \bE_*\setminus \bE$ with $\Br(x/\bE)(\bE)=\cO$ and such that $c_- < x <c_+$ will be such that $1/\lder^m \partial f (x) \in \Br(p/\bE)(\bE\langle x \rangle)$ for all $m$. Similarly for any $\cO$-limit $y$ with $y \in (c_-, c_+)$.
    \end{proof}
\end{proposition}

\begin{lemma}\label{lem:gne-ppty_transitive}
    If $(\bE, \cO) \prec (\bE_1, \cO_1)$ and $(\bE_1, \cO_2) \prec (\bE_2, \cO_2)$ have the gne-property, then $(\bE, \cO) \prec(\bE_2, \cO_2)$ has the gne-property.
    \begin{proof}
        Let $x \in \bE_2$. If $\tp(x/\bE)$ is realized in $\bE_1$, say by some $x_1$, then then there is an $\bE$-definable gne $g$, such that $\val(g^{-1}(x_1)-\bE)\subseteq \val(\bE)$ and so $\val(g^{-1}(x)-\bE)\subseteq \val(\bE)$.
        So suppose that $\tp(x/\bE)$ is not realized in $\bE_1$.
        Then since $(\bE_1, \cO_2) \prec (\bE_2, \cO_2)$ has the gne-property, there is an $\bE_1$-definable gne $g$ such that $x=g(z)$ and $\val(z-\bE_1)\subseteq \val(\bE_1)$. Without loss of generality we can furthermore assume that $g$ is normal and essential at $z$. Write $g=g_0$ and $g_i=c_i+\sigma_i\exp(g_{i+1})$ for all $i<m$ where $m$ is the height of $g$. Note that if there is no $\tilde{c}_i\in \bE$ such that $g_i(z)\equiv_\bE \tilde{c}_i + \sigma_i\exp(g_{i+1}(z))$, since $\val(\exp(g_{i+1}(z)) \notin \val(\bE_1)$ we must have $c_i\equiv_\bE g_i(z)$ and $\tp(g_i(z)/\bE)$ would be realized in $\bE$. Thus since we assumed that $\tp(g_0(z)/\bE)$ is not realized in $\bE_1$, inductively we can find an $\bE$-definable gne $\tilde{g}$ such that $\tilde{g}(z)\equiv_\bE g(z)=x$.
    \end{proof}
\end{lemma}

\begin{proof}[Proof of Theorem~\ref{introthm:transserial_char}]
    \ref{introthmenum:transserial} $\Leftrightarrow$ \ref{introthmenum:exp-bdd_gne_at_special}
    is by Lemma~\ref{lem:transserial_vs_Wtransserial} and Corollary~\ref{cor:wtransserial_gne}.
    \ref{introthmenum:transserial} $\Rightarrow$ (\ref{introthmenum:gne_wim_1} \& \ref{introthmenum:gne_res_1})
    is Proposition~\ref{prop:transserial_sOabs}, and
    (\ref{introthmenum:gne_wim_1} or \ref{introthmenum:gne_res_1}) $\Rightarrow$ \ref{introthmenum:transserial} is Proposition~\ref{prop:non-transserial_sOabs}.
    \ref{introthmenum:gne_wim_1} $\Leftrightarrow$ \ref{introthmenum:gne_wim}
    and \ref{introthmenum:gne_res_1} $\Leftrightarrow$ \ref{introthmenum:gne_res}
    are by Lemma~\ref{lem:gne-ppty_transitive}.
\end{proof}

\subsection{Consequences of transseriality}\label{ssec:consequences}
This subsection contains the proof of Theorem~\ref{introthm:transserial_consequences}.

\begin{definition}\label{def:rosenlicht_levels}
    Let $T$ be an exponential o-minimal theory and $(\bE, \cO)\models T_\convex$. Recall that $x, y \in \bE^{> \cO}$ have the same \emph{Rosenlicht level} if there is a natural number $m$ such that $\val(\log_m x)=\val(\log_m y)$. We write $x\asymp_Ly$ when $x$ and $y$ have the same Rosenlicht level. $\asymp_L$ is an equivalence relation with convex classes, so the set of Rosenlicht levels is naturally ordered.
\end{definition}

\begin{proof}[Proof of Theorem~\ref{introthm:transserial_consequences}]
    (1). By Theorem~\ref{introthm:transserial_char}, if $(\bE, \cO)\preceq (\bE_*, \cO_*)$ is res-constructible or wim-constructible, then it has the gne-property. So given any $y \in \bE_*^{>\cO_*}$, we can write $y=g(z)$ for an $\bE$-definable gne and some $z$ such that $\val_{\cO_*}(z-\bE)\subseteq \val_{\cO_*}(\bE)$. Furthermore by Lemma~\ref{lem:gne-normalize}, we can assume that such $g$ is normal and essential at $z$. Write $g=g_0$ and $g_i=c_i+\sigma_i\exp(g_{i+1})$ for all $i<m$ where $m$ is the height of $g$. If all $c_i$s are $0$, then it must be $\sigma_i=1$ for all $i$ and thus we have $y=\exp_m(z)$ where $\val(z) \in \val(\bE)$. If instead $c_i\neq 0$ for some $i$, pick the minimum such $i$ and observe that by normality of $g$, $c_i\succ \exp(g_{i+1}(z))$ and hence $\log_i(y)=c_i+\sigma_i\exp(g_{i+1}(z))\sim c_i$. In any case we have $\val(\log_n y) \in \val(\bE)$ for some $n\in \bN$ and hence $y \asymp_L d$ for some $d \in \bE$.
    (2). Follows from Corollary~\ref{cor:finite_hight_and_br_ortho}.
\end{proof}

\begin{example}
    All simply exponential o-minimal theories are transserial. This follows from the equivalence $(4)\Leftrightarrow(1)$ in Theorem~\ref{introthm:transserial_char} and \cite[Prop.~5.27]{freni2024t}, or directly from strong quantifier elimination for a simply exponential theory (relative to suitable power-bounded reducts) and the equivalence $(3)\Leftrightarrow(1)$: in fact if $T$ is simply exponential, then up to interdefinability we can assume that $L=L_0 \cup \{\exp, \log\}$ where $T|_{L_0}$ is power-bounded and $T$ is model-complete and universally axiomatized (see e.g.\ \cite[Sec.~5.2]{freni2024t}), so all $T$-definable functions are given piecewise by terms. But then proceeding by induction on terms we can prove condition (3) in Theorem~\ref{introthm:transserial_char} as follows. Suppose that (3) holds for $y_1=h(b)$ and let $y=f(y_1)$. If $f \in \{\exp, \log\}$, then (3) still holds. So suppose $f$ is a term of $L_0(\bE)$. Then up to translations by elements in $\bE$, we can assume that $\val_{\cO_*^-}(h(b))\notin \val_{\cO_*^-}(\bE)$, so there are $\beta\in \Exponents(T|_{L_0})$, $c,d \in \bE$ such that $f(h(b))\equiv_\bE c+ dh(b)^\beta$ so the validity of (3) is preserved.
\end{example}

\section{Account of AI interactions}\label{sec:ai_interactions}

Gemini 3.1 pro was used to check for typos, grammar errors, and inconsistencies. It also pointed out some minor mathematical errors, that were easily fixed by the author. 

Part of the proof of logarithmic order-convexity in Proposition~\ref{prop:der_order-convexity} was suggested by an interaction with Gemini 3.1 pro. See comments to this section in the latex code for more details.

\bibliographystyle{abbrvnat}
\bibliography{Ref}

\end{document}